\documentclass[12pt]{amsart}
\usepackage[T1]{fontenc}
\usepackage[pdftex]{color,graphicx}
\usepackage{url}
\usepackage{subcaption}
\usepackage{orcidlink} 
\usepackage{hyperref}
\usepackage{mathtools}
\mathtoolsset{showonlyrefs}
\usepackage{anysize}
\usepackage{geometry}
\marginsize{2.0cm}{2.0cm}{2.0cm}{2.0cm}
\usepackage{amssymb}
\usepackage{bbm}
\usepackage{orcidlink}
\usepackage{algorithm}
\usepackage{algpseudocode}
\usepackage{float}
\usepackage{soul}
\theoremstyle{plain}
\numberwithin{equation}{section}
\newtheorem{theorem}{Theorem}[section]
\newtheorem{lemma}[theorem]{Lemma}
\newtheorem{corollary}[theorem]{Corollary}

\newtheorem{remark}[theorem]{Remark}
\newtheorem{example}[theorem]{Example}
\newtheorem{definition}[theorem]{Definition}
\newtheorem{hypothesis}[theorem]{Hypothesis}
\newcommand{\e}{\varepsilon}
\newcommand{\cQ}{\mathcal{Q}}
\newcommand{\bN}{\mathbb{N}}
\newcommand{\bR}{\mathbb{R}}
\newcommand{\cW}{\mathcal{W}_p}
\newcommand{\qQ}{\mathcal{Q}}
\newcommand{\ud}{\mathsf{d}}

\newcommand{\ii}{\mathsf{i}}

\begin{document}
{
\begin{tikzpicture}[remember picture, overlay]
\node[anchor=north west, xshift=1.5cm, yshift=-1.8cm] at (current page.north west){
\includegraphics[width=4cm]{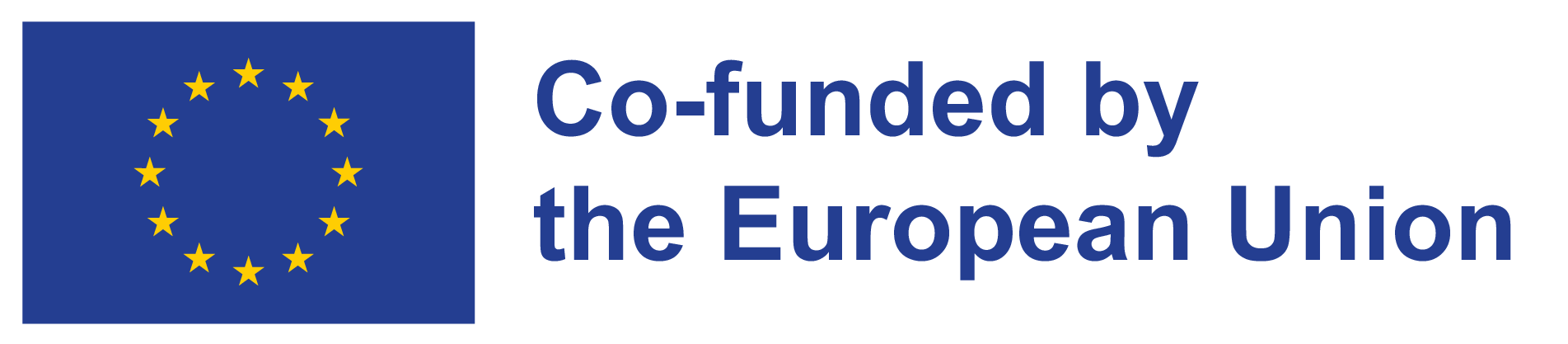} 
};
\end{tikzpicture}
}

\title[Ergodicity bounds for stable AR and ARMA processes]{Ergodicity bounds in the Sliced Wasserstein distance for Schur stable autoregressive  processes}

\author{Gerardo Barrera~\orcidlink{0000-0002-8012-2600}}
\address{
Center for Mathematical Analysis, Geometry and Dynamical Systems (CAMGSD), Mathematics Department, Instituto Superior T\'ecnico, Universidade de Lisboa, Lisboa, Portugal.
}
\email{gerardo.barrera.vargas@tecnico.ulisboa.pt}
\author{Paulo Henrique da Costa~\orcidlink{0009-0006-0941-9256}}
\address{Departamento de Matem\'atica, Universidade de Bras\'ilia, Bras\'ilia, Brasil. }
\email{phcosta@unb.br}

\author{Michael A. H\"ogele~\orcidlink{0000-0001-5744-0494}}
\address{
Departamento de Matem\'aticas, Universidad de los Andes, Bogot\'a, Colombia.}
\email{ma.hoegele@uniandes.edu.co}
\date{\today}


\begin{abstract}
Explicit calculations in dimension one show for Schur stable autoregressive processes with standard Gaussian noise 
that the ergodic convergence in the Wasserstein-$2$ distance is essentially given by the sum of the mean, which decays exponentially, and the standard deviation, which decays with twice the speed.
This paper starts by showing new upper and lower multivariate affine transport bounds for the Wasserstein-$r$ distance for $r$ greater and equal to $1$. These bounds allow to formulate a novel sufficient (non-Gaussian) ergodic interpolation condition for the mentioned mean-variance behavior to take place in case of more general Schur stable multivariate autoregressive processes. All ergodic estimates are non-asymptotic with completely explicit constants. The main applications are precise thermalization bounds for Schur stable $\mathsf{AR}(p)$ and $\mathsf{ARMA}(p,q)$ models 
in Wasserstein and Sliced Wasserstein distance. 
In the sequel we establish with the help of coupling techniques 
explicit upper and lower exponential bounds for more general multivariate Schur stable autoregressive processes.  This includes parallel sampling and the convergence of the empirical means. The utility of our results in particular for the Sliced Wasserstein distance are confirmed by multivariate numerical experiments. 
\end{abstract}

\maketitle


\section{\textbf{Introduction}}

We start with the motivating example for this article. Consider the autoregressive model of order one 
\begin{equation}\label{eq:au1}
X_{t}(x)=q X_{t-1}(x)+\sigma \xi_{t},\quad X_0(x)=x\in \mathbb{R},\quad t\in \mathbb{N}:=\{1,2,\ldots,\},
\end{equation}
where $q\in (-1,1)$, $(\xi_j)_{j\in \mathbb{N}}$ is a sequence of independent and identically distributed (for short i.i.d.) standard Gaussian random variables, and the noise intensity $\sigma\neq 0$.
By iteration, one can see that
\begin{equation}\label{eq:g1}
X_t(x)\stackrel{\mathsf{d}}{=}\mathcal{N}\left(q^t x, \sigma^2 \sum_{j=0}^{t-1}q^{2j}\right)
=\mathcal{N}\left(q^t x, \sigma^2\frac{1-q^{2t}}{1-q^2}\right),\quad t\in \mathbb{N}
\end{equation}
and its limiting law
\begin{equation}\label{eq:g2}
X_\infty
\stackrel{\mathsf{d}}{=}\mathcal{N}\left(0, \sigma^2\frac{1}{1-q^2}\right),
\end{equation}
where $\stackrel{\mathsf{d}}{=}$ denotes equality in law and $\mathcal{N}(m,\eta^2)$ denotes the Gaussian distribution with mean $m\in \mathbb{R}$ and variance $\eta^2>0$.
The Gaussianity in~\eqref{eq:g1} and~\eqref{eq:g2} allows us to have an explicit expression for the so-called Wasserstein distance of order two between the laws of $X_t(x)$ and $X_\infty$, see for instance~\cite{Gelbrich,GivensShortt}. More precisely, 
\begin{align}
\mathcal{W}_2(X_t(x),X_\infty)&=
\sqrt{|q|^{2t}x^2+\frac{\sigma^2}{1-q^2}\left(1-\sqrt{1-q^{2t}}\right)^2}
\\
&=\sqrt{\underbrace{|q|^{2t}x^2}_{\textrm{ mean component }}+\underbrace{\frac{\sigma^2}{1-q^2}\cdot\frac{q^{4t}}{\left(\sqrt{1-q^{2t}}+1\right)^2}}_{\textrm{ noise  
  component }}}.\label{eq:for1W}
\end{align}
The preceding formula exhibits a convergence pattern previously observed in~\cite{BarreraHoegele} for the Ornstein--Uhlenbeck process: the convergence of the means is of order $|q|^t$ while the convergence of the standard deviations of \eqref{eq:g1} to \eqref{eq:g2} in \eqref{eq:for1W} is of order $|q|^{2t}$, that is, for all $t\in \mathbb{N}$
\begin{equation}\label{eq:srsd}
\begin{split}
\left|\sqrt{\sigma^2\frac{1-q^{2t}}{1-q^2}}-\sqrt{\sigma^2\frac{1}{1-q^2}}\right|&=
\frac{\sigma}{\sqrt{1-q^2}}\left|\sqrt{1-q^{2t}}-1\right|\\
&=\frac{\sigma}{\sqrt{1-q^2}}\frac{|q|^{2t}}{\sqrt{1-q^{2t}}+1} = \frac{\sigma}{2 \sqrt{1-q^2}}|q|^{2t} + \mathsf{o}(|q|^{2t}).
\end{split}
\end{equation}
More generally and in higher dimensions one can deduce  similar rates for the convergence of the means, however, 
it is not straightforward as in~\eqref{eq:srsd} to obtain that the convergence of the matrix square root of the covariances is twice the rate of convergence of the means in~\eqref{eq:for1W}. Similar formulas as \eqref{eq:for1W} for the Wasserstein distance of order $r\geq 1$, $r\neq 2$, between Gaussian distributions are only available dimension one. In fact, in Lemma~\ref{lem:loc-scale} for any random variable $X$ with finite $r$-th absolute moment for some $r\geq 1$ we show that
\[
\mathcal{W}_r(X,m+\sigma X)=\left(\mathbb{E}[|m+(\sigma-1)X|^r]\right)^{1/r}.
\]
In other words, the synchronous coupling (also referred to as natural coupling or graph coupling), between the laws of $X$ and $m+\sigma X$, that is, $\omega\mapsto (X(\omega),m+\sigma X(\omega))$, is optimal for any cost function of the form $|\cdot|^r$, $r\geq 1$.
In particular, for any $r\geq 1$, $m_1,m_2\in \mathbb{R}, \sigma^2_1+\sigma^2_2\geq 0$,  we have
\begin{equation}\label{eq:fgr}
\begin{split}
\mathcal{W}_r(\mathcal{N}(m_1,\sigma^2_1),\mathcal{N}(m_2,\sigma^2_2))&=\left(\mathbb{E}\left[\left|(m_2-m_1)+(\sigma_2-\sigma_1)\mathcal{N}(0,1)\right|^r\right]\right)^{1/r}\\
&=\left(\mathbb{E}\left[\left|\mathcal{N}(m_2-m_1,(\sigma_2-\sigma_1)^2)\right|^r\right]\right)^{1/r},
\end{split}
\end{equation}
which seems to be new in the literature for $r\neq 2$. 
The right-hand side of \eqref{eq:fgr} can be computed explicitly for $r\in \mathbb{N}$  while for $r\geq 1$, $r\not\in \mathbb{N}$ can be expressed in terms of the so-called  Kummer's confluent hypergeometric function.
In particular, another example of interest is 
the symmetric $\alpha$-stable distribution $\mathcal{S}_\alpha$ with $\alpha\in (1,2]$ for which we have
\[
\mathcal{W}_r(\mathcal{S}_\alpha,m+\sigma \mathcal{S}_\alpha)=\left(\mathbb{E}[|m+(\sigma-1)\mathcal{S}_\alpha|^r]\right)^{1/r}
\]
for any $1\leq r<\alpha$, where no closed formula, without special functions, is known for $m\neq 0$.

In this manuscript, we quantify the convergence to equilibrium for the multivariate Schur-stable analogue of the autoregressive model~\eqref{eq:au1} for general noises $(\xi_j)_{j\in \mathbb{N}}$ with finite $r$-th absolute moment in the Wasserstein distance of order $r\geq 1$ and the respective Sliced Wasserstein distance of order $r\geq 1$. \\

Autoregressive models are an essential part of data analysis, see for instance \cite{HornJohnson,  proctor2016,tu14}, with abundant applications in time series analysis, statistics, and artificial intelligence \cite{box2015,  halko2011, hamilton1994,kolda2009}  modeling economic \cite{sims1980}, biological \cite{caswell2001}, demographic phenomena \cite{lee1992} among others \cite{uria2016}. \\

In case of Schur-stable interactions their ergodic behavior can be understood as a discretized version of a Hurwitz-stable multivariate Ornstein-Uhlenbeck process, however, with greater freedom in the driving noise distributions. 
We refer to \cite{BarreraHoegele, BarreraHoegelePardoPavlyukevich} for ergodic estimates on those kind of processes and their generalizations with different L\'evy drivers in finite and infinite dimensions. There (see \cite[Theorem~2.6]{BarreraHoegele}) and in \eqref{eq:au1} it is the Gaussianity of the driving noise which yields an \textit{additive} decomposition of the $\mathcal{W}_2$ error with different speeds, given by the mean of $X_t$, which is of order $q^t$ plus an error of higher order, $q^{2t}$ in this case, which represents the ergodic convergence of the stochastic convolution. Such an additive decomposition into mean and higher order rate is wrong in general, see Example~\ref{ex:Bernoulli}, instead we can only expect a single scale, which corresponds to a \textit{multiplicative} error in the sense that the error is proportional to $q^t \sqrt{|x|^2 + C}$, where $C$ is a positive constant depending on the noise distribution. 

This article aims at understanding when such a time scale separation occurs, since it indicates that essentially only the mean behavior is relevant after a logarithmically short time in the size of the noise parameters. The second main goal is to establish bounds for the generic case when this condition is not satisfied. 

The setting of this manuscript generalizes the introductory example from $1$-d stable to arbitrary multivariate situations with Schur-stable interaction matrices, from standard Gaussian perturbations to general distributions with a finite absolute moment $r>1$  (including $\alpha$-stable distributions for $1 < \alpha < 2$) and from the Pythagorean situation of $\mathcal{W}_2$ to general $\mathcal{W}_r$, $r\geq 1$.  \\
In Subsection~\ref{ss:affineergodicinterpolation} we identify a new sufficient condition called "ergodic interpolation" which guaranties that the rates of convergence in the Wasserstein-$r$ and the Sliced Wasserstein-$r$ distance have an additive decomposition into a mean and a higher order ergodic component. 
For $d=1$ this condition is equivalent to a self-similarity condition and is only satisfied for Gaussian drivers and $\alpha$-stable distributions. In addition, we give an example that does not satisfy ergodic interpolation, and by calculating the $\mathcal{W}_2$-error we show that the error does not exhibit an additive but only a multiplicative structure and the notion is not void. In higher dimensions, this notion seems to be non-trivial and is subject to further research.\\
We start with new affine transport estimates for the multivariate Wasserstein-$p$ distance, $p>1$, in Lemma~\ref{prop:affine}.
Combining those with the notion of ergodic interpolation we obtain the first main result in Theorem~\ref{thm:affine}. 
For Gaussian noise vectors we obtain a multivariate Schur-stable extension of the introductory example 
in Theorem~\ref{th:affineGauss}. In a string of corollaries, Corollaries~\ref{cor:projected}, \ref{cor:PWdII}, \ref{cor:sliced} and \ref{cor:aeiistable} extend the preceding theorems to the case of projected Wasserstein distances which can be significantly improved under an additional commutativity condition, the Sliced Wasserstein distance for the Gaussian case and, finally, the Sliced Wasserstein distance for isotropic $\alpha$-stable noise vectors.  \\
In Theorem~\ref{thm:genericARMA}  of Subsection~\ref{ss:genericexponentialbounds} we establish new non-asymptotic upper and lower exponential bounds on the Wasserstein-$p$ distance, $p\geq 1$ for arbitrary Schur-stable autoregessive processes driven by arbitrary random vectors with finite absolute moment $1\leq p \leq r$. Here, the upper and lower bounds are both exponential with the same rate in time, however, with different prefactors. In case of diagonalizable interaction matrix treated in Corollary~\ref{cor:genericARMA} and is generic as a function of the matrix components, these conditions boil down to specific bounds in terms of powers of the leading eigenvalue. Further applications are ergodic bounds for parallel sampling (Lemma~\ref{rem:parallelsampling}) and ergodic estimates for the empirical process of an autoregressive model in Corollary~\ref{cor:empirical}. \\
In Subsection~\ref{ss:matrixpowerprojection} we show the precise asymptotics of Schur stable matrix power projections $\qQ^t x$, $x\neq 0$, for general (complex-valued) square matrices $\qQ$ according to their Jordan normal form, since those naturally appear in the upper and lower bounds of the preceding subsections. In Subsection~\ref{ss:numerical} we illustrate the different results obtained in the previous sections for different matrices $\mathcal{Q}$ and indices $r$ of the respective ergodic Wasserstein-$r$ distance and its Sliced analogue. 

\textbf{Notation:} We denote by $\bN$ the set of positive integers numbers and $\bN_0$ denotes the set of non-negative integer numbers. When not stated otherwise we denote by $|\cdot |$ Euclidean norm in any finite dimensional vector space, which is clear from the context. 
For a random vector $X$ on a given probability space and a distribution on the image space of $X$ 
we shall $X \stackrel{\mathsf{d}}{=} \nu$ for $\mathbb{P}_X = \nu$. In this manuscript, the identity matrix in $\mathbb{R}^{d\times d}$ is always denoted by $I_d$. We denote by $GL(d, \mathbb{C})$ the group of invertible $d\times d$ square matrices (w.r.t. the matrix multiplication) with entries in $\mathbb{C}$.

\section{\textbf{The setup}}
\subsection{\textbf{The autoregressive model}} 
Let $(\xi_t)_{t\in \bN}$ be a sequence of i.i.d. random vectors in $\bR^d$ 
defined in a probability space $(\Omega,\mathcal{F},\mathbb{P})$
with given probability distribution 
$\nu$. 
We consider the multivariate autoregressive model $(X_t(x))_{t\in \mathbb{N}_0}$ given by the strong solution of the linear recurrence equation in $\bR^d$ 
\begin{equation}\label{eq:model}
X_t(x)=\cQ X_{t-1}(x)+\Sigma\xi_t\quad \textrm{ for all }\quad t\in \bN\quad \textrm{ with }\quad X_0(x)=x\in \bR^d,
\end{equation}
where $\Sigma\in \mathbb{R}^{d\times d}$.
We note that \eqref{eq:model} satisfies the pathwise representation
\begin{equation}\label{eq:repret}
X_t(x)=\cQ^t x+ \sum_{j=0}^{t-1} \cQ^j \Sigma\xi_{t-j},	\quad t\in \mathbb{N}, \quad X_0(x) = x,
\end{equation}
and due to i.i.d. noise $(\xi_j)_{j\in \mathbb{N}}$ we also have the weak representation
\begin{equation}\label{eq:reprelimit}
X_t(x)\stackrel{\mathsf{d}}=\cQ^t x+ \sum_{j=0}^{t-1} \cQ^j\Sigma\xi_{j}, \quad t\in \mathbb{N}, \quad X_0(x) = x. 
\end{equation}
The relevance to study the model \eqref{eq:repret} is given by the following processes, which are omnipresent in time series analysis and modeling, see for instance \cite{BroDav91}.

\medskip 

\noindent \textbf{Important classes of examples:}  \\
\noindent \textbf{(1) $\mathsf{AR}(p)$,  the autoregressive model of order $p\in \mathbb{N}$, 
$(Y_t)_{t\in \mathbb{N}_0}$.} For a given the initial string 
$Y_0=y_0,Y_1=y_1,\ldots,Y_{p-1}=y_{p-1}$ it is defined by
\[
Y_t=\sum_{j=1}^{p} \varphi_j Y_{t-j}+\varepsilon_t\quad \textrm{ for }\quad t\in \{p,p+1,\ldots,\},
\]
where $\varphi_1,\ldots,\varphi_p$ are real-valued model parameters and $(\varepsilon_t)_{t\geq p}$ are i.i.d. random variables.
The preceding model is of type of~\eqref{eq:model}. Indeed, 
$d=p$, $X_t=(Y_t,Y_{t-1},\ldots,Y_{t-(d-1)})^{T}\in \mathbb{R}^d$ for $t\in \{d,d+1,\ldots,\}$.
The matrix $\cQ$ is the transpose of a so-called companion matrix~\cite{HornJohnson} 
and given by 
\begin{equation}\label{eq:Qcomp}
\cQ=\begin{pmatrix}
\varphi_1 & \varphi_2 && \ldots & \varphi_d\\
1        & 0         && \ldots & 0\\
0        & \ddots    &\ddots&  & 0\\
\vdots    &           &&        & \vdots\\
0        & 0         &\dots & 1 & 0   
\end{pmatrix},
\end{equation}
$\Sigma=e_1\otimes e_1+\sum_{j=2}^{d}a_j e_j\otimes e_j$ where $\{e_1,\ldots,e_d\}$ is the canonical basis of $\mathbb{R}^d$, $a_2,\ldots,a_d$ are any non-negative real numbers, $\otimes $ denotes the usual Kronecker product, 
and $\xi_t=\varepsilon_t e_1\in \mathbb{R}^d$. This model can be extended by $q$ moving averages as follows. 

\noindent
\textbf{(2) $\mathsf{ARMA}(p,q)$, the autoregressive moving-average model, with $p\in \mathbb{N}$ autoregressive terms and $q\in \mathbb{N}$ moving average noise components.} More precisely,
\begin{equation}\label{d:arma}
Y_t=\sum_{j=1}^{p} \varphi_j Y_{t-j}+\varepsilon_t+\sum_{j=1}^q \theta_j\varepsilon_{t-j}
\quad \textrm{ for }\quad t\in \{p,p+1,\ldots,\},
\end{equation}
with initial string 
$Y_0=y_0,Y_1=y_1,\ldots,Y_{p-1}=y_{p-1}$,  $(\varepsilon_t)_{t\geq p-q}$ are i.i.d. random variables, $q\leq  p$, $\theta_1,\ldots,\theta_q$ are fixed numbers. 
This can be written in an enhanced version of the form of~\eqref{eq:model}
by its state space representation in $\mathbb{R}^{p+q}$ 
\begin{align}
&X_t = 
\left(
\begin{array}{c} 
Y_{t}\\ 
Y_{t-1}\\
\vdots \\
Y_{t-(p-1)}\\
\e_{t}\\
\e_{t-1}\\
\vdots\\
\e_{t-(q-1)}
\end{array} \right),
\quad 
\xi_t  =
\left(
\begin{array}{c} 
\e_{t} \\
0\\
\vdots \\
0\\
\end{array}  \right), 
\quad 
\Sigma = (e_1\otimes e_1) + (e_{p+1} \otimes e_{p+1})
\in \mathbb{R}^{(p+q)\times (p+q)},\nonumber\\
\end{align}
\begin{align}
&\widetilde{\cQ} 
= \left(\begin{array}{ccccc|cccc} 
\varphi_1  & \varphi_2    &  \dots    & \dots & \varphi_p     & \theta_1  & \dots     & \theta_{q-1}  & \theta_q \\
1       & 0         &  \dots    &        & 0         & 0         &\dots      &0              & \vdots \\
0       & 1         &  \ddots        &        & \vdots    & \vdots    &      &              & 0 \\
\vdots  &  \ddots        & \ddots     &   \ddots     &           &           &           &         &\vdots  \\
0       & \dots         & 0     &1       &  0         &0          &           &  \dots  & 0   \\
\hline
0       & \dots     & \dots     &         & 0         & 0         &           &\dots         &   0 \\
\vdots  &           &           &         &\vdots     &1          & \ddots    &     &\vdots \\
\vdots  &           &           &            &           &0     & \ddots         &\ddots              &0 \\
0       &\dots      &           &            &      0     &0           &0          &1              & 0
\end{array}\right)\in \mathbb{R}^{(p+q)\times (p+q)},\label{e:enhanced}
\end{align}
see for instance \cite[Proposition~2.1]{Stelzer}.

\noindent \subsection{\textbf{The Schur-stability hypothesis}}
Since we are interested in convergence to the dynamical equilibrium of the system, we assume the following stability condition for $\cQ$ in~\eqref{eq:model}.

\begin{hypothesis}[Schur stability]\label{hyp:hyperbolic} 
The eigenvalues of $\cQ$ have modulus strictly less than one, that is to say, the spectral radius of $\cQ$ satisfies
\[
\rho(\cQ):=\max_{1\leq i\leq d}|q_i|<1,
\]
where $q_1,\ldots,q_d\in \mathbb{C}$ are the eigenvalues of $\cQ$ (possibly repeated by algebraic multiplicity).
\end{hypothesis}

\medskip 
\begin{remark}
For the case of the $\mathsf{ARMA}(p,q)$ model the matrix $\widetilde{\qQ}$ given in~\eqref{e:enhanced} is stable if and only if 
the matrix $\cQ$ in~\eqref{eq:Qcomp} is stable. It can be seen by the fact that by the block diagonal structure 
\begin{equation}
\left(
\begin{array}{cc}
\qQ & \Theta \\
O & S\\
\end{array}
\right), \quad \Theta \in \mathbb{R}^{p\times q}, \quad S\in \mathbb{R}^{q\times q},
\end{equation}
of the matrix $\widetilde \qQ$ that its characteristic polynomial $P_{\widetilde{\qQ}}(z) := \mathsf{det}(\widetilde{\qQ} - z I_{p+q})$, $z\in \mathbb{C}$ satisfies (see also~\cite[2.14.14]{Bernsteinbook}
\[
P_{\widetilde{\qQ}}(z) = 
P_{\qQ}(z) P_{S}(z)
= P_{\qQ}(z) (-z)^q,\quad z\in \mathbb{C}. 
\]
In addition, due to the companion structure of $\cQ$ it is not hard to see (for instance \cite[Theorem~3.3.14]{HornJohnson} up to a relabelling) that 
\begin{equation}\label{eq:charpol}
P_{\qQ}(z)=z^p-\sum_{j=1}^{p}\varphi_{j}z^{p-j}, \quad z\in \mathbb{C}.
\end{equation}
Hence, the enhanced matrix $\widetilde{\cQ}$ is Schur stable if and only if the  matrix $\cQ$ is so.
\end{remark}

\medskip 
\begin{remark}\label{rem:p2}
The computational verification of Schur stability can be carried out via standard algorithms such as the Schur-Cohn test, the Jury test or the Bistritz test, see for instance Chapter~X in~\cite{Mardenbook}.
Nevertheless, a characterization of Schur stability for the  matrix $\cQ$ in~\eqref{eq:Qcomp} in terms of the coefficients $\varphi_1,\ldots,\varphi_p$ is hard to obtain in closed form. 
For instance, even in the case of $p=2$, the matrix $\cQ$ is Schur stable if and only if either
\[
|\varphi_1|+|\varphi_2|<1,\quad \textrm{ "diamond" shape,}
\]
or
\begin{align}\label{e:wing}
\begin{split}
&|\varphi_1|+|\varphi_2|\geq 1,\quad |\varphi_1|-1<|\varphi_2|<1,\quad \varphi^2_1 \varphi_2<0,\\
&2\arccos\left(\frac{1+\varphi^2_1-\varphi^2_2}{2|\varphi_1|}\right)+
\arccos\left(\frac{1-\varphi^2_1+\varphi^2_2}{2|\varphi_2|}\right)<\pi, \quad 
\textrm{ "wing" shape.}
\end{split}
\end{align}
\noindent For details we refer to~\cite{barreras,Cermak2015}. 

\noindent This criteria can be generalized to the situation of \eqref{eq:charpol}, which is a monic polynomial (highest order monomial has coefficient $1$), $\varphi_p \neq 0$,  
and exactly one additional index $\varphi_r\neq 0$, where $r$ and $p$ are co-primes, and all other coefficientes $\varphi_j = 0$ for all $j\neq p, r$. In this case the diamond (also known as Cohn domain) 
\[
|\varphi_1| + |\varphi_r| < 1
\]
is always contained in the stability region, see \cite{barreras}. See also \cite{Borweinbook, Prasolovbook, RahmanSchmeisser} for a comprehensive view on the subject.
\end{remark}

\medskip 
\begin{remark}
A sufficient condition for Schur stability of the  matrix $\cQ$  given in~\eqref{eq:Qcomp} is the following:
$-1<\varphi_1<\varphi_2<\dots<\varphi_p<0$,
see Enestr\"om--Kakeya Theorem in~\cite{Mardenbook}.
We point out that the preceding condition is sufficient but not necessary, see for instance the case $p=2$ in Remark~\ref{rem:p2}.
\end{remark}

\medskip 
\begin{remark}
For $p=3$ 
a complete characterization of Schur stability for \eqref{eq:Qcomp} is given in~Theorem~5 in \cite{Cermak2019} in terms of non-linear inequalities.
 For instance, an application of Theorem~5 in~\cite{Cermak2019} yields the sufficient condition for Schur stability
\[
|\varphi_1+\varphi_2\varphi_3|
+|\varphi_2+\varphi_1\varphi_3|<1-\varphi^2_3.
\]
While for $p=4$ there are still explicit formulas for the eigenvalues of $\cQ$ as functions of the coefficients by Ferrari's method up to our knowledge, a closed form characterization of Schur stability in terms of the coefficients is non-trivial to obtain. The particular case of $\varphi_2=0$ is also covered
in~Theorem~5 in \cite{Cermak2019}. A sufficient condition in this case reads
\[
|\varphi_1+\varphi_3\varphi_4|
+|\varphi_3+\varphi_1\varphi_4|<1-\varphi^2_4.
\]
For higher $p$ only special cases are known in closed form, see for instance~\cite{Cermak2019}.
\end{remark}

Note that the spectral radius is not a matrix norm, since it does not satisfy the triangular inequality. In order to translate Schur stability into a numerically traceable norm contractivity, 
we adapt the following result to our setting, see Lemma 5.6.10 in \cite{HornJohnson}. 

\begin{lemma}[Schur stability vs. norm contractivity] \label{lem:*}
Given $\cQ\in \mathbb{R}^{d\times d}$, $d\in \mathbb{N}$, with 
Schur triangulation 
$\cQ=U\Delta U^*$ and $J_\kappa:=\mathsf{diag}(\kappa,\kappa^2,\ldots,\kappa^d)$ for any $\kappa> \kappa_*:=\max\{1, \frac{\|\Delta\|_1}{1-\rho(\cQ)}\}$ fixed. 
Then it follows: 
\begin{enumerate}
    \item $\cQ$ is Schur stable if and only if there exists a matrix norm $\|\cdot\|_{*} := \|\cdot \|_{*, \kappa}$ such that $\|\cQ\|_{*}<1$. More precisely, 
    \[
    \|\cQ\|_{*, \kappa}\leq \rho(\cQ) + \frac{\|\Delta\|_1}{\kappa} < 1. 
    \]
    Such a matrix norm $\|\cdot\|_{*}$ can be constructed algorithmically via Schur triangularization \cite[Theorem~2.3.1]{HornJohnson} and the matrix $1$-norm (maximum column sum matrix norm).

    \item $\|\cdot\|_{*}$ is submultiplicative and 
    \[
   |Ax|\leq  K_d \|A\|_{*}|x|\quad \forall A\in \mathbb{R}^{d\times d}, x\in \mathbb{R}^d, \quad \textrm{ where } \quad K_d := (d\kappa^{d-1}\|U\|_1 \|U^{*}\|_1). 
    \]
    \item For the Frobenius norm $\|A\|_{F} := \sqrt{\sum_{i,j} {a_{i,j}^2}}$ for $A = (a_{i,j})\in \mathbb{R}^{d\times d}$ we have 
    \[
    \|A\|_F\leq C_* \|A\|_{*}, \quad \textrm{ where }\quad C_* := \sqrt{d} \|S\|_1 \|S^{-1}\|_1, \quad  S:= U J_\kappa.
    \]
    Moreover, $C_* \leq d^{\frac{5}{2}} \kappa^{d-1}$.
\end{enumerate}
\end{lemma}

\begin{proof}
(1) Note that in case of $\cQ$ being diagonalizable, $\cQ=UDU^{-1}$ with $U\in \mathsf{GL}(d,\mathbb{C})$, $D=\mathsf{diag}(q_1,\ldots,q_d)$ where $q_1,\ldots,q_d$ are the eigenvalues of $\cQ$,
one can choose 
\[
\|\cQ\|_{*}:=
\|UD U^{-1}\|_1=
\max_{1\leq j\leq d}{|q_j|}.
\]
For general $\cQ$, the Schur triangularization yields $\cQ =U\Delta U^{*}$, where $\Delta\in \mathbb{C}^{d\times d}$ is an upper triangular matrix with the eigenvalues of $\cQ$ on the diagonal, and $U\in \mathbb{C}^{d\times d}$ is a unitary matrix.
We now consider $J_\kappa:=\mathsf{diag}(\kappa,\kappa^2,\ldots,\kappa^d)\in \mathbb{R}^{d\times d}$
and define 
$\Delta_\kappa:=J^{-1}_\kappa \Delta J_\kappa$.
Note that
\begin{align}
\|\Delta_\kappa\|_1&=\max_{1\leq i\leq d} \sum_{j=1}^i
\kappa^{j-i}|\Delta_{j,i}|=
\max_{1\leq i\leq d} 
\left(|q_{i}|+
\sum_{j=1}^{i-1}
\kappa^{j-i}|\Delta_{j,i}|\right)\\
&\leq 
\max_{1\leq i\leq d} 
|q_{i}|+\frac{1}{\kappa}
\max_{1\leq i\leq d}
\sum_{j=1}^{i-1}|\Delta_{j,i}|\\
&\leq 
\rho(\cQ)+\frac{\|\Delta\|_1}{\kappa},
\end{align}
for any $\kappa\geq 1$
where $\Delta=(\Delta_{i,j})$.
On the one hand,
assume that $\rho(\cQ)<1$.
For $\kappa>\max\{1,\|\Delta\|_1/(1-\rho(\cQ))\}$ we have $\|\Delta_\kappa\|_1<1$.
Recall $U^*=U^{-1}$.
For $\kappa>\max\{1,\|\Delta\|_1/(1-\rho(\cQ))\}$ we define
\[
\|\cQ\|_{*}:=
\|(UJ_\kappa)^{-1}\cQ (UJ_\kappa)\|_1
=\|J^{-1}_\kappa U^{*}\cQ UJ_\kappa\|_1=\|J^{-1}_\kappa \Delta J_\kappa \|_1=\|\Delta_\kappa\|_1<1.
\]
Moreover, by Theorem 5.6.7. in \cite{HornJohnson} we have that 
\[
\|A\|_{*}:=\|A\|_{*, \kappa }
:= \|(UJ_\kappa)^{-1}A (UJ_\kappa)\|_1,\quad A\in \mathbb{C}^{d\times d}
\]
defines a norm.
On the other hand,
for the converse statement assume $\|\cQ\|_{*}<1$ for some $\kappa>0$.
Then Theorem~5.6.9.
in~\cite{HornJohnson} implies that
$\rho(\cQ)\leq \|\cQ\|_{*}$, which yields the Schur stability of $\cQ$.\\
\noindent (2) Using the submultiplicativity of $\|\cdot\|_{1}$ we have 
\[
\|AB\|_{*}\leq \|A\|_{*}\|B\|_{*}
\]
for any $A,B\in \mathbb{C}^{d\times d}$. 
Let $A\in \mathbb{R}^{d\times d}$ and $x\in \mathbb{R}^d$. For later use we observe the following calculations. By the Cauchy--Schwarz inequality we have
\begin{align}
|Ax|\leq |x| \|A\|_{F}\leq d|x| \|A\|_{1}.
\end{align}
Let $S:=UJ_\kappa$ and using the  submultiplicativity of $\|\cdot\|_{1}$  we have
\begin{align}
|A x| & \leq d |x| \|A\|_1
 = d|x| \|S(S^{-1}A S)S^{-1}\|_1
\leq 
d|x| \|S\|_1 \|S^{-1}\|_1 \|S^{-1}A S\|_1\\[2mm]
&\leq d|x| \|U\|_1 \|U^{-1}\|_1\|J_\kappa\|_1
\|J^{-1}_\kappa\|_1
 \|A\|_{*}
= (d\kappa^{d-1}\|U\|_1 \|U^{*}\|_1)
 \|A\|_{*}|x|
\end{align}
for all $x\in \mathbb{R}^d$.

\noindent (3) We recall that 
\begin{equation} \label{eq:F1}
\|A\|_{F}\leq \sqrt{d} \|A\|_1. 
\end{equation}
Since $S:=UJ_\kappa$ we continue 
\begin{align}
\|A\|_1 &= \|S (S^{-1} A S) S^{-1}\|_1 = \|S \Delta_\kappa S{{-1}}\|_1 = \|S\|_1 \|S^{-1}\|_1 \|\Delta_\kappa\|_1 \nonumber\\[2mm]
&= \|S\|_1 \|S^{-1}\|_1 \|A\|_{*}.\label{eq:1*}
\end{align}
Combining \eqref{eq:F1} and \eqref{eq:1*} we have that $C_* = \sqrt{d} \|S\|_1 \|S^{-1}\|_1$. 
In the sequel, we estimate $C_*$. The submultiplicativity of $\|\cdot \|_1$ and 
$\|U\|_F = \sqrt{d}$ for the unitary matrix $U$ yields 
\begin{align*}
C_* &= \sqrt{d} \|S\|_1 \|S^{-1}\|_1 
\leq \sqrt{d} \|U\|_1 \|J_\kappa\|_1 \|J_\kappa^{-1}\|_1  \|U\|_1\\[2mm]
&= \sqrt{d} \kappa^{d-1}  \|U\|_1  \|U^{-1}\|_1
\leq d^{\frac{3}{2}} \kappa^{d-1} \|U\|_F  \|U^{-1}\|_F
= d^{\frac{5}{2}} \kappa^{d-1}.
\end{align*}
\end{proof}

\medskip 
\begin{remark} Keep the notation of Lemma~\ref{lem:*}.
\begin{enumerate}
    \item 
The constant $C_*$ given in Lemma~\ref{lem:*} is $\sqrt{d}$ times the condition number of $S$ in the $1$-norm.   
\item By Lemma~\ref{lem:*}\,(2) we have for all $t\in \mathbb{N}$ that 
\begin{align}\label{e:Kd}
|\cQ^t x| & \leq d |x| \|\cQ^t\|_1\leq  K_d
 \|\cQ^t\|_{*}|x| 
 \leq 
 K_{d}
 \|\cQ\|^t_{*}|x|.
\end{align}
We also note that
\[
(\rho(\cQ))^t |x|\leq |\cQ^t x|.
\]
\end{enumerate}
\end{remark}

\noindent \subsection{\textbf{The moment hypothesis and ergodicity}}  Since our results are stated in the Wasserstein distance of order $r\geq 1$ introduced below, we assume the existence of the $r$-th absolute moment.
\begin{hypothesis}[Moment]\label{hyp:moment}
Assume that $\mathbb{E}[|\xi_1|^r]<\infty$ for some $r\geq 1$.
\end{hypothesis}
\noindent Under Hypothesis~\ref{hyp:hyperbolic} and Hypothesis~\ref{hyp:moment} for any initial datum $x\in \mathbb{R}^d$ the stochastic system~\eqref{eq:model} converges in law as $t$ tends to infinity to
\begin{equation}\label{eq:limite}
X_\infty:\stackrel{\mathsf{d}}{=}\sum_{j=0}^{\infty} \cQ^j\Sigma\xi_{j}.
\end{equation}
We denote the law of $X_\infty$ by $\mu$. For completeness this is shown in Lemma~\ref{lem:ergodicity} in Appendix~\ref{ap:ergodicity}. In fact, Lemma~\ref{lem:ergodicity} is formulated under Hypothesis~\ref{hyp:hyperbolic} 
and the existence of a logarithmic absolute moment, which is strictly weaker than Hypothesis~\ref{hyp:moment}. 

\noindent \subsection{\textbf{The Wasserstein-$r$ distance}} In the sequel, we define the Wasserstein distance of order $r\geq 1$ on the space of Borelian probabilities on $\mathbb{R}^d$ with finite $r$-th absolute moment by
\begin{equation}\label{eq:Wr}
\mathcal{W}_r(\mu_1,\mu_2)=\left(\inf_{\pi\in \Pi}\int_{\mathbb{R}^d\times \mathbb{R}^d}|u-v|^r \pi(\ud u, \ud v)\right)^{1/r}, 
\end{equation}
where $\Pi$ is the set of joint distributions (coupling)  between $\mu_1$ and $\mu_2$, that is, for all $\pi\in \Pi$ we have $\pi(\ud u,\mathbb{R}^d)=\mu_1(\ud u)$ and
$\pi(\mathbb{R}^d,\ud v)=\mu_2(\ud v)$. 
In a conscious abuse of notation, for $X_1$ and $X_2$ being random vectors taking values in $\mathbb{R}^d$ and with finite $r$-th absolute moments,
 we write
$\mathcal{W}_r(X_1,X_2):=\mathcal{W}_r(\mathbb{P}_{X_1},\mathbb{P}_{X_2})$,
where $\mathbb{P}_{X_1}$ and $\mathbb{P}_{X_2}$ are the distributions of $X_1$ and $X_2$, respectively.

\noindent For any coupling $\pi \in \Pi$  we denote the expectation with respect to $\pi$ by $\mathbb{E}_{\pi}[\cdot]$. More precisely, 
for any Borel-measurable and bounded function $f:\mathbb{R}^d\times \mathbb{R}^d\to \mathbb{C}$ we have
\[
\mathbb{E}_{\pi}[f(X_1,X_2)]:=\iint\limits_{\mathbb{R}^d\times \mathbb{R}^d} f(u,v) \pi(\ud u, \ud v).
\]
Let $g:\mathbb{R}^d\to \mathbb{C}$ be a Borel-measurable and bounded function. By the definition of $\Pi$,
we point out that
\begin{equation}\label{eq:coupl}
\mathbb{E}_{\pi}[g(X_1)]=\mathbb{E}_{\pi'}[g(X_1)]\quad
\textrm{ and }\quad
\mathbb{E}_{\pi}[g(X_2)]=\mathbb{E}_{\pi'}[g(X_2)]\quad
\end{equation}
for any $\pi,\pi'\in \Pi$. In other words, expectations of a function that only depends on one marginal do not depend on the choice of the coupling, and hence without loss of generality, we denote the expectation with respect a reference coupling by $\mathbb{E}$. For further readings we  refer to classical texts such as \cite{Panaretos, Villani}. 

\noindent \textbf{Elementary bounds of the Wasserstein distance:} 
Note that by definition ~\eqref{eq:Wr} gives the upper bound
\begin{equation}
\mathcal{W}_r(X_1,X_2)\leq \left(\mathbb{E}_{\pi}[|X_1-X_2|^r]\right)^{1/r}\quad \textrm{ for any }\quad \pi \in \Pi.
\end{equation}
For a random vector $X = (X_1, \dots, X_d)$ with values in $\mathbb{R}^d$ and $\mathbb{E}[|X|^2] <\infty$
we denote by $\mathsf{Cov}(X) = (\mathsf{Cov}(X_i, X_j))_{i,j}$ its covariance matrix. In addition for $m\in \mathbb{R}^d$ and $C\in \mathbb{R}^{d\times d}$ being symmetric, positive definite, $\mathcal{N}(m, C)$ denotes $d$-dimensional normal distribution with mean $m$ and covariance matrix $C$. 
We recall the following result. In \cite[Theorem~2.1]{Gelbrich} the author derives the following Gaussian lower bounds for the $\mathcal{W}_2$-distance 
\begin{align}
&\mathcal{W}_2^2(X_1, X_2)\nonumber\\
&\qquad \geq \mathcal{W}_2^2(\mathcal{N}(\mathbb{E}[X_1], \mathsf{Cov}(X)), \mathcal{N}(\mathbb{E}[X_2], \mathsf{Cov}(Y)))\label{e:Gelbrich}\\
&\qquad = |\mathbb{E}[X_1]- \mathbb{E}[X_2]|^2 + \mathsf{Trace}(\mathsf{Cov}(X_1)+\mathsf{Cov}(X_2)-2 (  \mathsf{Cov}(X_1)^\frac{1}{2} \mathsf{Cov}(X_2) \mathsf{Cov}(X_1)^{\frac{1}{2}})^\frac{1}{2}),
\label{e:Gelbrich2}
\end{align}
where $\mathsf{Trace}(M)=\sum_{j=1}^d M_{j,j}$ for any square matrix $M\in \mathbb{R}^{d\times d}$.
Since $\mathcal{W}_r(X_1, X_2)\geq \mathcal{W}_2(X_1, X_2)$ for $r\geq 2$ by Jensen's inequality and the so-called Bures distance between positive definite matrices (see~\cite{Bhatia})
$$\mathsf{Trace}(\mathsf{Cov}(X_1)+\mathsf{Cov}(X_2)-2 (\mathsf{Cov}(X_1)^\frac{1}{2} \mathsf{Cov}(X_2) \mathsf{Cov}(X_1)^{\frac{1}{2}})^\frac{1}{2})\geq 0,$$ we have $\mathcal{W}_r(X_1, X_2) \geq |\mathbb{E}[X_1]-\mathbb{E}[X_2]|$, $r\geq 2$. 
The remarkable formula \eqref{e:Gelbrich} has the disadvantage to require the square root of typically noncommuting matrices $\mathsf{Cov}(X)$ and $\mathsf{Cov}(Y)$, which turns out to be cumbersome in practice. 
In fact, 
Jensen's inequality implies for any $r\geq 1$ 
and any coupling $\pi \in \Pi$ that
\begin{equation}\label{eq:kj}
\begin{split}
\left|\mathbb{E}_\pi[X_1]-\mathbb{E}_\pi[X_2]\right|&=
\left|\mathbb{E}_\pi[X_1-X_2]\right|\leq \mathbb{E}_\pi[|X_1-X_2|]\leq \left(\mathbb{E}_\pi[|X_1-X_2|^r] \right)^{1/r}.
\end{split}
\end{equation}
Using~\eqref{eq:coupl} and optimizing over all $\pi\in \Pi$
we deduce 
\begin{equation}\label{e:meandifference}
\left|\mathbb{E}[X_1]-\mathbb{E}[X_2]\right|\leq 
\mathcal{W}_r(X_1,X_2)\quad \textrm{ for }\quad r\geq 1.
\end{equation}
The left-hand side of \eqref{e:meandifference}
is sometimes referred to as \textit{the engineer's distance}.
More basic properties of the Wasserstein distances are gathered in Lemma~\ref{lem:basic} in Appendix~\ref{ap:A}. Furthermore, in \cite{Kelbert} more upper and lower bounds for the Wasserstein distance between multivariate Gaussians are established. 

\noindent \subsection{\textbf{The Sliced Wasserstein-$r$ distance}}
Wasserstein distances of order $r\geq 1$ between the laws of random vectors with values in $\mathbb{R}^d$ are notoriously costly to approximate, while in dimension one there are explicit quantile formulas. The notion of sliced Wasserstein distance of order $r\geq 1$ introduced in \cite{RPC2010} is based on a spherical decomposition and balances these aspects.  
 The ideas consists of measuring the Wasserstein distance of the respective random vectors projected on a vector on the unit sphere and subsequently averaging the scalar Wasserstein-$r$ distances of the laws radially projected random vectors 
uniformly on the unit sphere in $\mathbb{R}^d$. For more details we refer to \cite{BDC25}. 
For simplicity we introduce the sliced Wasserstein-$r$ distance in terms of random vectors. 

\begin{definition}[Sliced Wasserstein distance]
For the laws of random vectors $X, Y$ with values $\mathbb{R}^d$ we define 
\[
S\mathcal{W}_r(X, Y) := \frac{1}{A_d}\left(\int_{|v|= 1} \mathcal{W}_r^r(\langle X, v\rangle, \langle Y, v\rangle) \ud \mathcal{H}_{d}(v)\right)^\frac{1}{r},
\]
where $\mathcal{H}_{d}$ denotes the (non-normalized) canonical surface measure on the unit sphere $\{|v| = 1\}\subset \mathbb{R}^d$ and 
$$A_d =\mathcal{H}_{d}(\{|v|=1\}) = \frac{2\pi^{\frac{d}{2}}}{\Gamma(\frac{d}{2})}.$$ 
\end{definition}

\noindent We note that $\mathcal{W}_r(\langle X, v\rangle, \langle Y, v\rangle)$ has explicit formulas in terms of the quantiles of the respective distributions, see for instance \cite{BDC25}.

\medskip 
\begin{remark}\label{rem:sliced} Fix $r\geq 1$ and consider $X, Y$ with finite $r$-th absolute moment. 
\begin{enumerate}
\item \noindent We note that $\mathcal{W}_r(\langle X, v\rangle, \langle Y, v\rangle)$ has explicit formulas in terms of the quantiles of the respective distributions, see for instance \cite{BDC25}.
\item Following \cite{CS24} we remark that $S\mathcal{W}_r$ and $\mathcal{W}_r$ generate the same topology and that always 
\[
\mathcal{W}_r(X, Y)\geq S\mathcal{W}_r(X, Y).
\]
\item By Jensen's inequality for $1 \leq s \leq r$ 
\[
S\mathcal{W}_s(X, Y)\leq S\mathcal{W}_r(X, Y). 
\]
\item We further note that by \eqref{e:meandifference} we have 
\begin{align*}
\mathcal{W}_r(\langle X, v\rangle, \langle Y, v\rangle)
\geq |\mathbb{E}[\langle X-Y, v\rangle]| = |\langle \mathbb{E}[X-Y], v\rangle|
= |\langle \mathbb{E}[X]-\mathbb{E}[Y], v\rangle|.
\end{align*}
Moreover we obtain the lower bound
\begin{align*}
S\mathcal{W}_r(X,Y) 
&\geq \frac{1}{A_d}\left(\int_{\{|v|=1\}} |\langle \mathbb{E}[X-Y], v\rangle|^r\ud \mathcal{H}_{d}(v)\right)^\frac{1}{r}\\
&=  |\mathbb{E}[X-Y]| \frac{1}{A_d}\left(\int_{\{|v|=1\}} \Big|\langle \frac{\mathbb{E}[X-Y]}{|\mathbb{E}[X-Y]|}, v\rangle\Big|^r\ud \mathcal{H}_{d}(v)\right)^\frac{1}{r}\\
&=  |\mathbb{E}[X-Y]| \frac{1}{A_d}\left(\int_{\{|v|=1\}} |\langle w, v\rangle|^r\ud \mathcal{H}_{d}(v)\right)^\frac{1}{r}\\
&\geq |\mathbb{E}[X-Y]| \frac{1}{A_d}\int_{\{|v|=1\}} |\langle w, v\rangle|\ud \mathcal{H}_{d}(v)
\end{align*}    
for any $w\in \mathbb{R}^d$ with $|w| = 1$ due to the rotational invariance of $\mathcal{H}_{d}$. 
In particular, for the first unit vector $w = e_1\in \mathbb{R}^d$ we have for the integral 
\begin{align*}
\int_{\{|v|=1\}} |\langle e_1, v\rangle|^r \frac{\ud\mathcal{H}_{d}(v)}{A_d}
&= \int_{\{|v|=1\}} |v_1|^r\frac{\ud\mathcal{H}_{d}(v)}{A_d}= 2 \pi^{\frac{d-1}{2}}\frac{ \Gamma\left(\frac{r + 1}{2}\right)}{\Gamma\left(\frac{r + d}{2}\right)} \frac{\Gamma(\frac{d}{2})}{2 \pi^{\frac{d}{2}}} = \frac{ \Gamma\left(\frac{r + 1}{2}\right)}{\Gamma\left(\frac{r + d}{2}\right)} \frac{\Gamma(\frac{d}{2})}{\pi^{\frac{1}{2}}}.
\end{align*}
\end{enumerate}
\end{remark}

\bigskip 
\section{\textbf{Main results}}
\subsection{\textbf{Ergodic interpolation}}\label{ss:affineergodicinterpolation}\hfill\\
\noindent We start with the following auxiliary estimate, which to our knowledge is new in the literature. 
\begin{lemma}[Affine transport estimates]\label{prop:affine}
Let $X$ be a random vector in $\mathbb{R}^d$,
$v\in \mathbb{R}^d$ be a deterministic vector and $R\in \mathbb{R}^{d\times d}$.
Assume that Hypothesis~\ref{hyp:moment} for some $r\geq 1$. 
Then for any $1\leq q\leq r$ it follows 
\begin{equation}
|v+(R-I_d)\mathbb{E}[X]|\leq \mathcal{W}_q(X,v+RX)\leq 
(\mathbb{E}[|v+(R-I_d)X)|^{q}])^{1/q}.
\end{equation}
\end{lemma}

The proof is given in Appendix~\ref{a:proofs3.1}.

\medskip 

\begin{remark}~
\begin{enumerate}
    \item We stress that our estimates do not rely on the Brenier Theorem and are hence universal in the sense of minimal requirements.  
    \item In the special case of $r=q=2$, $X$ having finite second moments and $R$ being symmetric and positive semidefinite, the authors in \cite{Chafai} recently also show with the help of the Brenier transport map the converse inequality of \eqref{e:Gelbrich} for an affine transport map $x\mapsto Rx +v$ and hence the identity 
\begin{align}\label{e:Chafai}
\mathcal{W}_2^2(X,v+RX) = \mathsf{Trace}(\Sigma_1 +\Sigma_2 - 2 R \Sigma_1) + |v + (R-I_d)\mathbb{E}[X]|^2,
\end{align}
where $\Sigma_1 = \mathsf{Cov}(X)$ and $\Sigma_2 = \mathsf{Cov}(RX) = R \mathsf{Cov}(X) R^T$. 
\end{enumerate}
\end{remark}

We state the first main result. It reflects the different rates of convergence for the mean and the variance laid out in the formulas \eqref{eq:for1W} and \eqref{eq:srsd} of the motivating example. 

\begin{theorem}[Ergodicity bounds under ergodic interpolation]\label{thm:affine}
Assume Hypothesis~\ref{hyp:hyperbolic} and Hypothesis~\ref{hyp:moment} for some $r\geq 1$.
In addition, assume that for each $t>0$ and $x\in \mathbb{R}^d$ there exist deterministic $A_t(x)\in \mathbb{R}^{d\times d}$ and deterministic $v_t(x)\in \mathbb{R}^d$ satisfying the ergodic interpolation between $x$ and $X_\infty$ as follows
\begin{equation}\label{e:affine-ergodi-interpolation}
X_t(x)\stackrel{\mathsf{d}}{=}A_t(x) X_\infty+v_t(x).
\end{equation}
Then for any $1\leq p\leq r$ it follows that
\begin{equation}\label{eq:thineq}
\begin{split}
|(A_t(x)-I_d)\mathbb{E}[X_\infty]+v_t(x)|\leq \mathcal{W}_p(X_t(x),X_\infty)
&\leq (\mathbb{E}[|(A_t(x)-I_d)X_\infty+v_t(x)|^{p}])^{1/p}\\
&\leq \|A_t(x)-I_d\|_F (\mathbb{E}[|X_\infty|^{p}])^{1/p}+|v_t(x)|.
\end{split}
\end{equation} 
\end{theorem}

\begin{proof}
By Hypothesis~\ref{hyp:hyperbolic} and Hypothesis~\ref{hyp:moment} there exists $X_\infty = \lim_{t\to\infty} X_t(x)$ in distribution, which is shown in Lemma~\ref{lem:ergodicity} in the appendix.  The remainder of the proof follows directly from the chain of inequalities given in 
Lemma~\ref{prop:affine} with the help of Minkowski's inequality and the matrix norm  submultiplicativity for $R=A_t(x)$ and $v=v_t(x)$.
\end{proof}

\medskip 
\begin{remark}[Common affine transport reference]\label{rem:commonOT}
Let $X$ and $X'$ be two random vectors taking values in $\mathbb{R}^d$ having finite $r$-th absolute moment for some $r\geq 1$.
Assume that $X\stackrel{\mathsf{d}}{=}AZ+v$ and $X'\stackrel{\mathsf{d}}{=}A'Z'+v'$, where $Z$ and $Z'$ are random vectors taking values in $\mathbb{R}^d$,
$A,A'\in \mathbb{R}^{\ell\times d}$ are deterministic matrices and $v,v'\in \mathbb{R}^d$ are deterministic vectors. In addition, $Z\stackrel{\mathsf{d}}{=}Z'$.
The synchronous coupling $Z''\stackrel{\mathsf{d}}{=}Z'\stackrel{\mathsf{d}}{=}Z$ with the help of Minkowski's inequality and the matrix norm  submultiplicativity yields 
\begin{equation}
\begin{split}
\mathcal{W}_p(X,X')&\leq (\mathbb{E}[|(AZ''+v)-(A'Z''+v')|^p])^{1/p}\\
&\leq 
(\mathbb{E}[|(A-A')Z''+(v-v')|^p])^{1/p}\\
&\leq
\|A-A'\|_F (\mathbb{E}[|Z''|^p])^{1/p}+|v-v'|
\end{split}
\end{equation}
for all $1\leq p\leq r$.
In particular, Theorem~\ref{thm:affine} can be recovered taking $Z\stackrel{\mathsf{d}}=X_\infty$, $A=A_t(x)$, $v=v_t(x)$, $Z'=X_\infty$, $A'=I_d$ and $v'=0$.
\end{remark}

\medskip

\begin{remark}
If $X\stackrel{\mathsf{d}}{=}aZ+v$ and $X'\stackrel{\mathsf{d}}{=}a'Z+v'$ for some constants $a,a',v,v'\in \mathbb{R}$ and $Z$ being a random variable with finite second moment with mean $m$ and variance $\sigma^2$.
Using the common reference coupling induce by $Z$ 
for $p=2$ and $d=1$ we have 
\begin{equation}
\begin{split}
\mathcal{W}^2_2(X,X')&\leq \mathbb{E}[|aZ+v-(a'Z+v')|^2]=\mathbb{E}[|(a-a')Z+(v-v')|^2]\\
&=|a-a'|^2\mathbb{E}[|Z|^2]+2(a-a')(v-v')\mathbb{E}[Z]+|v-v'|^2\\
&=|a-a'|^2(\sigma^2+m^2)+2(a-a')(v-v')m+|v-v'|^2\\
&=\sigma^2|a-a'|^2+|(a-a')m+(v-v')|^2\\
&=\mathcal{W}^2_2(\mathcal{N}(\mathbb{E}[X],\mathsf{Var}(X)),\mathcal{N}(\mathbb{E}[X'],\mathsf{Var}(X'))),
\end{split}
\end{equation}
which with the help of~\eqref{e:Gelbrich} yields
\[
\mathcal{W}^2_2(X,X')=\sigma^2|a-a'|^2+|(a-a')m+(v-v')|^2.
\]
\end{remark}

\begin{lemma}[Optimal coupling for location-scale distributions]\label{lem:loc-scale}
Let $X$ be random variable with finite $r$-th absolute moment for some $r\geq 1$.
For any $m\in \mathbb{R}$ and $\sigma\in \mathbb{R}$ it follows that 
\[
\mathcal{W}_r(X,m+\sigma X ) =\left(\mathbb{E}[|m+(\sigma-1)X|^r]\right)^{1/r}.
\]
\end{lemma}

\begin{proof}
For $\sigma=0$,  the proof is straightforward.  
For any $r\geq 1$ we recall that 

\[
\mathcal{W}_r^r(X,Y) = \int_0^1 |F_{X}^{\leftarrow}(u) - F^{\leftarrow}_{Y}(u)|^r \, \mathrm{d} u,
\]
where $F^{\leftarrow}_{X}$ and $F^{\leftarrow}_{Y}$ are the quantiles of the law of $X$ and $Y$, respectively.
We note that
\[
\mathcal{W}_r^r(X,Y) = \mathbb{E}[|F_{X}^{\leftarrow}(U)-F_{Y}^{\leftarrow}(U)|^r], 
\]
where $U$ is uniformly  distributed on $[0,1]$.
We point out that $X\stackrel{\mathsf{d}}{=}F_{X}^{\leftarrow}(U)$ and $Y\stackrel{\mathsf{d}}{=}F_{Y}^{\leftarrow}(U)$.
For $Y\stackrel{\mathsf{d}}{=}m+\sigma X$ we have 
\[
\mathcal{W}_r^r(X,m+\sigma X ) = \mathbb{E}[|F_{X}^{\leftarrow}(U)-F_{m+\sigma X}^{\leftarrow}(U)|^r]. 
\]
Note that for any $u\in (0,1)$ and $m\in \mathbb{R}$ and  $\sigma> 0$ (without loss of generality) we have 
\[
F^{\leftarrow}_{m+\sigma X}(u)=m+\sigma F^{\leftarrow}_{X}(u),
\]
which yields
\[
\mathcal{W}_r^r(X,m+\sigma X ) = \mathbb{E}[|m+(\sigma-1)F^{\leftarrow}_{X}(U)|^r]=\mathbb{E}[|m+(\sigma-1)X|^r].
\]
\end{proof}

\medskip
\begin{example}\label{ex:Bernoulli}
For $d=1$ consider the i.i.d. symmetric Rademacher $(\xi_t)_{t\in \mathbb{N}}$, $\xi_t \stackrel{\mathsf{d}}{=} \frac{1}{2}\delta_{-1} + \frac{1}{2}\delta_1$ and the respective autoregressive model 
\[
X_{t}(x) = \frac{1}{2}X_{t-1} + \frac{1}{2}\xi_t, \qquad X_0(x) = x. 
\]
It known from the literature (via the dyadic representation of the natural numbers) that $X_{\infty} \stackrel{\mathsf{d}}{=} U([-1,1])$. 
The discrete approximation of the absolutely continuous limiting law and the affine shape of the ergodic interpolation relation \eqref{e:affine-ergodi-interpolation} are obviously incompatible. However, in the Appendix~\ref{a:Bernoulli} we calculate 
\[
\mathcal{W}_2(X_t, X_\infty) = \frac{1}{2^t} \sqrt{x^2 + \frac{1}{3}},
\]
that is, there is no additive separation between the scale mean component which is $\frac{1}{2^t} |x|$ and the additional noise convolution component. We refer to Appendix~\ref{a:Bernoulli} for the precise calculations. 
This is confirmed by numerical validation. 
\begin{figure}[h!]
    \centering
\includegraphics[width=0.55\textwidth]{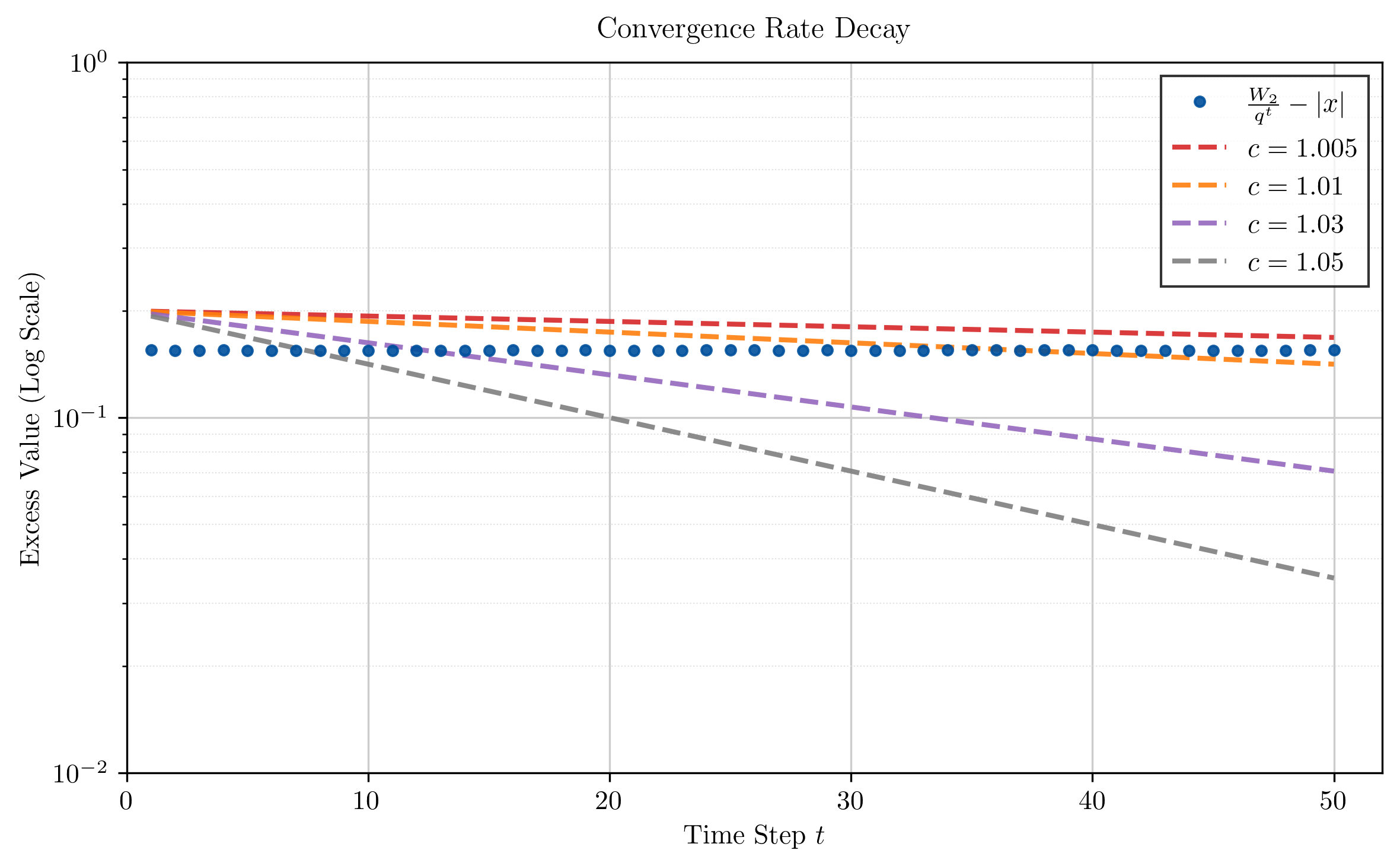}
    \caption{The plot compares the simulated ratio $\mathcal{W}_2/(1/2)^t$ against hypothetical decay bounds $|x| + C (1/2)^{t(c-1)}$.  The blue logarithmic data cross exponential bounds with small exponents $c-1 = 0.05, 0.03, 0.01, 0.005$.} 
    \label{fig:w2_excess1}
\end{figure}
\end{example}

\noindent We recall that the norm $\|\cdot \|_* = \|\cdot \|_{*, \kappa}$ and the constant $C_* = C_{*, \kappa}$ are constructed for a fixed value of $\kappa > \kappa_*$ in Lemma~\ref{lem:*}.   

\begin{theorem}[Ergodic interpolation for Gaussian noise]\label{th:affineGauss}
Assume that $\cQ$ satisfies Hypothesis~\ref{hyp:hyperbolic}.
Assume that $\xi_1$ has a Gaussian distribution with vector mean $m$ and covariance matrix $\Xi$.
In addition, assume that $\Sigma_\infty$ is invertible.
For any deterministic $B\in \mathbb{R}^{\ell \times d}$, $x\in \mathbb{R}^d$, $r\geq 1$ and $t\in \mathbb{N}$ it follows that 
\begin{equation}\label{eq:thgauss}
\begin{split}
|B\cQ^t (x-(I_d-\cQ)^{-1}\Sigma m)|&\leq 
\mathcal{W}_r(B X_t(x), B X_\infty )\\
&\leq |B\cQ^t (x-(I_d-\cQ)^{-1}\Sigma m)|\\
&\quad +\frac{\|B\|_F^2 \|\Sigma\|_F^2 \|\Xi\|_F}{\lambda_{-}}
 \left(\mathbb{E}[|\mathcal{N}|^r]\right)^{\frac{1}{r}} \frac{C_*^2}{1-\|\cQ\|_*^{2}}\|\cQ\|_*^{2t},
\end{split}
\end{equation}
where $C_*$ is given in Lemma~\ref{lem:*}, $\lambda_{-}$ is the smallest eigenvalue of $\Sigma_\infty$, $\mathcal{N}$ has standard Gaussian distribution on $\mathbb{R}^d$ and $\|\cdot \|_F$ is the Frobenius norm. Moreover, for any $t\in \mathbb{N}$ we have 
\begin{align*}
0&\leq \mathcal{W}_r(B X_t(x), B X_\infty ) - |B\cQ^t (x-(I_d-\cQ)^{-1}\Sigma m)|\\
&\leq \frac{\|B\|_F^2 \|\Sigma\|_F^2 \|\Xi\|_F}{\lambda_{-}}
 \left(\mathbb{E}[|\mathcal{N}|^r]\right)^{\frac{1}{r}} \inf_{\kappa >\kappa_*}\frac{C_*^2}{1-\|\cQ\|_*^{2}}\|\cQ\|_*^{2t}.
\end{align*}
\end{theorem}

\noindent The proof is given in Appendix~\ref{a:proofs3.1}.

\medskip 
\begin{remark}~
\begin{enumerate}
 \item 
Note that $\Sigma_\infty$
 solves a matrix Lyapunov equation, which has a positive definite unique solution when for instance  $\Xi$ is positive definite. Otherwise,  
 non-singularity of $\Sigma_\infty$ can be ensured in terms of controllability assumptions such as the Kalman conditions.
 \item Note that for $r= 2$ we have the exact formula \eqref{e:Gelbrich2}, while up to our knowledge for $r\neq 2$ and $d \geq 2$ no closed formula is known. Moreover, the bounds in \eqref{eq:thgauss} 
can be evaluated without the costly calculation of square root of powers of the covariance matrix, which appear in \eqref{e:Gelbrich2}. 
 \item By Lemma~\ref{lem:*} we have 
 \begin{align*}
 \inf_{\kappa >\kappa_*}\frac{C_*^2}{1-\|\cQ\|_*^{2}}\|\cQ\|_*^{2t} 
 &\leq  \inf_{\kappa >\kappa_*}\frac{d^5 \kappa^{2(d-1)} (\rho(\cQ) + \frac{\|\Delta\|_1}{\kappa})^{2t} }{1-(\rho(\cQ) + \frac{\|\Delta\|_1}{\kappa})^2}. 
 \end{align*}
\end{enumerate} 
\end{remark}

In the following we apply Theorem~\ref{th:affineGauss} to coordinate projections. 

\begin{corollary}[Projected Wasserstein distance I]\label{cor:projected}
Assume that $\cQ$ satisfies Hypothesis~\ref{hyp:hyperbolic}.
Assume that $\xi_1$ has a Gaussian distribution with vector mean $m$ and covariance matrix $\Xi$.
In addition, assume that $\Sigma_\infty$ is invertible.
For any $x\in \mathbb{R}^d$, $v\in \mathbb{R}^d$, $v\neq 0$, $r\geq 1$ and $t\in \mathbb{N}$, it follows that 
\begin{equation}\label{eq:cor}
\begin{split}
|\langle v,\cQ^t (x-(I_d-\cQ)^{-1}\Sigma m)\rangle|
&\leq 
\mathcal{W}_r(\langle v, X_t(x) \rangle,\langle v, X_\infty \rangle)\\
&\leq |\langle v,\cQ^t (x-(I_d-\cQ)^{-1}\Sigma m)\rangle|\\
&\quad+C_*^2\frac{|v|^2 \|\Sigma\|_F^2\|\|\Xi\|_F}{\sqrt{\langle v,(\Sigma_t+\Sigma_\infty) v \rangle}}
\frac{\|\cQ\|_*^{2t}}{1-\|\cQ\|_*^{2}}\frac{\sqrt{2}(\Gamma((r+1)/2))^{\frac{1}{r}}}{\pi^{\frac{1}{2r}}}\\
&\leq |\langle v,\cQ^t (x-(I_d-\cQ)^{-1}\Sigma m)\rangle|\\
&\quad+
\frac{|v| \|\Sigma\|_F^2\|\|\Xi\|_F}{\sqrt{\lambda_-}}
\frac{\sqrt{2}(\Gamma((r+1)/2))^{\frac{1}{r}}}{\pi^{\frac{1}{2r}}} \frac{C_*^2}{1-\|\cQ\|_*^{2}} \|\cQ\|_*^{2t},
\end{split}
\end{equation}
where $\Sigma_t$ and $\Sigma_\infty$ are the covariance matrices of $X_t(0)$ and $X_\infty$, respectively, $C_*$ is given in Lemma~\ref{lem:*}, $\lambda_{-}$ is the smallest eigenvalue of $\Sigma_\infty$, and $\Gamma$ is the Gamma function. 
Moreover, for any $t\in \mathbb{N}$ we have
\begin{align*}
0&\leq \mathcal{W}_r(\langle v, X_t(x) \rangle,\langle v, X_\infty \rangle) -|\langle v,\cQ^t (x-(I_d-\cQ)^{-1}\Sigma m)\rangle|\\
&\leq \frac{|v| \|\Sigma\|_F^2\|\|\Xi\|_F}{\sqrt{\lambda_-}}
\frac{\sqrt{2}(\Gamma((r+1)/2))^{\frac{1}{r}}}{\pi^{\frac{1}{2r}}} \inf_{\kappa >\kappa_*} \frac{C_*^2}{1-\|\cQ\|_*^{2}} \|\cQ\|_*^{2t}.
\end{align*}
\end{corollary}

\noindent The proof is given in Appendix~\ref{a:proofs3.1}.

\begin{corollary}[Projected Wasserstein distance II]\label{cor:PWdII}
Assume that $\cQ$ satisfies Hypothesis~\ref{hyp:hyperbolic}.
Assume that $\xi_1$ has a Gaussian distribution with vector mean $m$ and covariance matrix $\Xi$. In addition, assume $\cQ\Sigma\Xi \Sigma^T=\Sigma\Xi \Sigma^T \cQ$ and that $\Sigma_\infty$ is invertible.
For any $x\in \mathbb{R}^d$, $v\in \mathbb{R}^d$, $v\neq 0$, $r\geq 1$ and $t\in \mathbb{N}$ it follows that 
\begin{equation}\label{eq:PWDII}
\begin{split}
\mathcal{W}_r(\langle v, X_t(x) \rangle,\langle v, X_\infty \rangle)&\leq |\langle v,\cQ^t (x-(I_d-\cQ)^{-1}\Sigma m)\rangle|\\
&\quad+\frac{1}{\sqrt{v^T\Sigma_\infty v}}
\frac{\sqrt{2}(\Gamma((r+1)/2))^{1/r}}{\sqrt{\pi}^{1/r}} |(\Sigma\Xi \Sigma^T S)^{1/2}\cQ^t v|^2
\end{split}
\end{equation}
for $t\geq 0$, where  $\Sigma_\infty$ are the covariance matrices of $X_t(0)$, $\Gamma$ is the Gamma function, $\mathcal{N}$ has standard Gaussian distribution on $\mathbb{R}^d$ and 
\[
S:=\sum_{j=0}^{\infty}\cQ^j(\cQ^j)^T=(I-\cQ \cQ^{T})^{-1}.
\]
\end{corollary}

\noindent By Corollary~\ref{cor:projected}, Remark~\ref{rem:sliced} and Minkowski's inequality we obtain the following consequence.  
\begin{corollary}[Sliced Wasserstein-$r$ bounds]\label{cor:sliced}
Assume that $\cQ$ satisfies Hypothesis~\ref{hyp:hyperbolic}.
Assume that $\xi_1$ has a Gaussian distribution with vector mean $m$ and covariance matrix $\Xi$.
In addition, assume that $\Sigma_\infty$ is invertible.
For any $x\in \mathbb{R}^d$, $r\geq 1$ and $t\in \mathbb{N}$, it follows that 
\begin{equation}\label{eq:cor}
\begin{split}
&|\cQ^t (x-(I_d-\cQ)^{-1}\Sigma m)| \frac{ \Gamma\left(\frac{r + 1}{2}\right)}{\Gamma\left(\frac{r + d}{2}\right)} \frac{\Gamma(\frac{d}{2})}{\pi^{\frac{1}{2}}} = \frac{1}{A_d} \left(\int_{\{|v| = 1\}} |\langle v,\cQ^t (x-(I_d-\cQ)^{-1}\Sigma m)\rangle|^r \ud \mathcal{H}_{n}(v)\right)^\frac{1}{r}\\[3mm]
&\leq 
S\mathcal{W}_r(X_t(x), X_\infty)\\
&\leq \frac{1}{A_d} \left(\int_{\{|v|=1\}}  |\langle v,\cQ^t (x-(I_d-\cQ)^{-1}\Sigma m)\rangle|^r  \ud \mathcal{H}_{n}(v)\right)^\frac{1}{r}+\frac{C_*^2\|\Sigma\|_F^2 \|\Xi\|_F}{(1-\|\cQ\|_*^{2}) \sqrt{\lambda_-}} 
\frac{\sqrt{2}(\Gamma(\frac{r+1}{2}))^{\frac{1}{r}}}{\pi^{\frac{1}{2r}}} \|\cQ\|_*^{2t}\\
&= |\cQ^t (x-(I_d-\cQ)^{-1}\Sigma m)| \frac{ \Gamma\left(\frac{r + 1}{2}\right)}{\Gamma\left(\frac{r + d}{2}\right)} \frac{\Gamma(\frac{d}{2})}{\pi^{\frac{1}{2}}}+\frac{ \|\Sigma\|_F^2 \|\Xi\|_F}{\sqrt{\lambda_-}}
\frac{\sqrt{2}(\Gamma(\frac{r+1}{2}))^{\frac{1}{r}}}{\pi^{\frac{1}{2r}}} \frac{C_*^2}{1-\|\cQ\|_*^{2}}\|\cQ\|_*^{2t},
\end{split}
\end{equation}
where $\Sigma_t$ and $\Sigma_\infty$ are the covariance matrices of $X_t(0)$ and $X_\infty$, respectively, $C_*$ is given in Lemma~\ref{lem:*}, $\lambda_{-}$ is the smallest eigenvalue of $\Sigma_\infty$, and $\Gamma$ is the Gamma function. Moreover, for any $t\in \mathbb{N}$ we have
\begin{align*}
0&\leq  S\mathcal{W}_r(X_t(x), X_\infty) - |\cQ^t (x-(I_d-\cQ)^{-1}\Sigma m)| \frac{ \Gamma\left(\frac{r + 1}{2}\right)}{\Gamma\left(\frac{r + d}{2}\right)} \frac{\Gamma(\frac{d}{2})}{\pi^{\frac{1}{2}}}\\
&\leq \frac{ \|\Sigma\|_F^2 \|\Xi\|_F}{\sqrt{\lambda_-}}
\frac{\sqrt{2}(\Gamma(\frac{r+1}{2}))^{\frac{1}{r}}}{\pi^{\frac{1}{2r}}} \inf_{\kappa>\kappa_*}\frac{C_*^2}{1-\|\cQ\|_*^{2}}\|\cQ\|_*^{2t}.
\end{align*}
\end{corollary}

\medskip 

\begin{corollary}[Ergodic interpolation for isotropic $\alpha$-stable noise perturbations]\label{cor:aeiistable}
Assume that $\cQ=\lambda I_d$ with $|\lambda|<1$ and $\Sigma=I_d$. Assume
that the random vector $\xi_0$ has an isotropic $\alpha$-stable distribution on $\mathbb{R}^d$ for some $\alpha \in (1,2]$, that is,
the characteristic function of $\xi_0$ is given by
$\mathbb{E}[e^{\ii \langle \xi,u\rangle}]=e^{-c_0|u|^\alpha}$ for all $u\in \mathbb{R}^d$ and some $\alpha \in (1,2]$ and $c_0>0$.
Then for each $t>0$
 it follows that 
\begin{equation}\label{eq:scala}
X_t(x)\stackrel{\mathsf{d}}{=}\lambda^t x+A(t)X_\infty, \quad \textrm{ where }\quad A(t)=a(t)I_d,\quad a(t):=\left(1-|\lambda|^{\alpha t}\right)^{1/\alpha}.
\end{equation}
In particular,  for any $1\leq q< \alpha$ it follows that
\begin{equation}\label{eq:thineqdos2}
\begin{split}
|\lambda|^t |x|\leq \mathcal{W}_q(X_t(x),X_\infty)
&\leq (\mathbb{E}[|(a(t)-1)X_\infty+\lambda^t x|^{q}])^{1/q}\\
&\leq |\lambda|^t |x|+\sqrt{d}(1-a(t))(\mathbb{E}[|X_\infty|^{q}])^{1/q}\\
&\leq |\lambda|^t |x|+\Lambda(t)(\mathbb{E}[|X_\infty|^{q}])^{1/q} ,
\end{split}
\end{equation} 
where with the help of Bernoulli's inequality we have
\[
\Lambda(t) = \Lambda_{\lambda, \alpha,d}(t)=\sqrt{d}|\lambda|^{\alpha t}
\]
and
\[
\mathbb{E}[|X_\infty|^{q}]
= \frac{2^{q}\Gamma \left(\frac{1+q}{2}\right)
  \Gamma\left(1 - \frac{q}{\alpha}\right)}{\sqrt{\pi}}
  \left(
     \frac{c_0}{1 - |\lambda|^{\alpha}}
  \right)^{q/\alpha},
  \qquad 1 \leq q < \alpha.
\]
\end{corollary}

\medskip 
\begin{remark}
For $x=0$ we note that the right-hand side of~\eqref{eq:thineqdos2} is of the order $|\lambda|^{\alpha t}$, which mimics the behavior of the motivating example~\eqref{eq:au1}, in particular, the convergence rate in~\eqref{eq:srsd} for the case $q=|\lambda|$ and $\alpha=2$.
\end{remark}

\bigskip 
\subsection{\textbf{Generic exponential bounds for Schur stable autoregressive processes}}\label{ss:genericexponentialbounds}\hfill\\

 The ergodic interpolation property \eqref{e:affine-ergodi-interpolation} which guarantees the fast variance convergence 
 is rather specific and depends crucially on the distribution of the underlying noise. 
 If we may drop this condition and go over to general Schur stable distributions 
 the price to pay is to loose the mean/variance behavior, however, we still obtain meaningful explicit non-asymptotic exponential upper and lower bounds.

In the sequel, we use coupling techniques to establish non-asymptotic bounds on the ergodic rate to the dynamical equilibrium. Due to the robustness of the coupling methods the following theorem does not require explicit formulas of the Wasserstein distance (which are typically no accessible) nor Gaussianity of the underlying noise nor a specific shape of the optimal coupling. 

\begin{theorem}[Schur stable autoregressive model]\label{thm:genericARMA}
Assume that $\cQ$ satisfies Hypothesis~\ref{hyp:hyperbolic} and $(\xi_t)_{t\geq 0}$ satisfies Hypothesis~\ref{hyp:moment} for some $r\geq 1$. Consider $(X_t(x))_{t\geq 0}$ the solution of \eqref{eq:model}. 
Then for any $x\in \mathbb{R^d}$, $1\leq p \leq r$ and $t\geq 0$ it follows that
\begin{equation}\label{e:genericARMA}
\begin{split}
\left |\cQ^t(x-(I-\cQ)^{-1}\mathbb{E}[\Sigma \xi_1])\right |
&\leq  \cW(X_t(x),X_\infty) \leq 
\begin{cases}
\mathbb{E}[|\cQ^t(x-X_\infty)|]\leq K_d \|\cQ\|_*^t \mathbb{E}[|x-X_\infty|],\\[2mm]
|\cQ^t x|+K_d
\left(\mathbb{E}\left[|\Sigma\xi_1|^p \right]\right)^{1/p}
\frac{\|\cQ\|^{t+1}_{*}}{(1-\|\cQ\|^p_{*})^{1/p}},
\end{cases}
\end{split}
\end{equation}
where 
\begin{equation}
\mathbb{E}[|x-X_\infty|]\leq 
|x|+\frac{K_d\mathbb{E}\left[|\Sigma\xi| \right]\|\cQ\|_{*}}{1-\|\cQ\|_{*}} 
\end{equation}
and the constant $K_d$ is given Lemma~\ref{lem:*}. 
\end{theorem}
\noindent The proof is given in Appendix~\ref{a:proofs3.2}.

\medskip 
\begin{remark} By the inverse triangular inequality for the $L^p$ norm we have 
for any coupling $\pi$ between and $X_t$ and $X_\infty$ the universal lower bound 
\begin{align}
|(\mathbb{E}[|X_t(x)|^p])^{1/p}-(\mathbb{E}[|X_\infty|^p])^{1/p}|
\leq \mathbb{E}_\pi[|X_t(x)- X_\infty|^p]^\frac{1}{p}.
\end{align}
Minimizing over all couplings yields 
\begin{align}
|(\mathbb{E}[|X_t(x)|^p])^{1/p}-(\mathbb{E}[|X_\infty|^p])^{1/p}|\leq \cW(X_t(x),X_\infty). 
\end{align}
\end{remark}

\medskip 
\begin{corollary}\label{cor:schursliced}
Assume that $\cQ$ satisfies Hypothesis~\ref{hyp:hyperbolic} and $(\xi_t)_{t\geq 0}$ satisfies Hypothesis~\ref{hyp:moment} for some $r\geq 1$. Consider $(X_t(x))_{t\geq 0}$ the solution of \eqref{eq:model}. 
Then for any $x\in \mathbb{R^d}$, $1\leq p \leq r$ and $t\geq 0$ it follows that
\begin{equation}\label{e:genericARMAsliced}
\begin{split}
\frac{\Gamma(\frac{d}{2})}{\Gamma\left(\frac{d+1}{2}\right)\pi^{\frac{1}{2}}} |\cQ^t(x-(I-\cQ)^{-1}\mathbb{E}[\Sigma \xi_1])|
&\leq  S\cW(X_t(x),X_\infty) \\
&\leq 
\begin{cases}
\frac{\Gamma(\frac{d}{2})}{\Gamma\left(\frac{d+1}{2}\right)\pi^{\frac{1}{2}}} \mathbb{E}[|\cQ^t(x-X_\infty)|],\\[2mm]
\frac{\Gamma(\frac{d}{2})}{\Gamma\left(\frac{d+1}{2}\right)\pi^{\frac{1}{2}}} |Q^t x| + K_d
\left(\mathbb{E}\left[|\Sigma\xi|^p_1 \right]\right)^{1/p}\frac{\|\cQ\|^{t+1}_{*}}{(1-\|\cQ\|^p_{*})^{1/p}},
\end{cases}
\end{split}
\end{equation}
where 
\begin{equation}
\mathbb{E}[|x-X_\infty|]\leq 
|x|+\frac{K_d\mathbb{E}\left[|\Sigma\xi| \right]\|\cQ\|_{*}}{1-\|\cQ\|_{*}} 
\end{equation}
and the constant $K_d$ is given Lemma~\ref{lem:*}. 
\end{corollary}
\noindent The proof is given in Appendix~\ref{a:proofs3.2}.\\

\noindent In the sequel we show Wasserstein estimates for $\cQ$ being diagonalizable, parallel sampling and empirical means. Corresponding results for sliced Wasserstein distance can be deduced straightforwardly. 

\begin{definition}[Diagonalizable]\label{def:generic}
We say that $\cQ\in \mathbb{R}^{d\times d}$ is diagonalizable, if there exists $U\in GL(d, \mathbb{C})$ such that 
\[
U^{-1} \cQ U \quad \textrm{ is a diagonal matrix}. 
\]
\end{definition}

\medskip 
\begin{remark}~
\begin{enumerate}
  \item We do not assume that $\cQ$ is a normal matrix, hence $U$ needs not be orthogonal nor unitary. 
    \item If $\cQ$ is a diagonalizable and Schur stable, 
    then $\max_{i=1, \dots, d} |q_i| < 1$ where 
 $q_1,q_2,\ldots,q_{d-1},q_{d}\in \mathbb{C}$ are the eigenvalues of $\cQ$. 
\item Without loss of generality, we can assume that 
\[
0\leq |q_1|\leq |q_2|\leq 
|q_3|\leq |q_4|\leq \cdots\leq |q_{d-1}|\leq |q_{d}|<1.
\]
Note that $\cQ$ has real entries, which implies that all complex eigenvalues come in complex conjugate pairs.
Recall that the spectral radius of $\cQ$ is $\rho:=|q_d|$.
Since it has $d$ eigenvalues, there exists a basis of eigenvectors $\{w_1,\ldots,w_d\}$ for $\mathbb{C}^d$ satisfying $\cQ w_j=q_jw_j$  and $|w_j|=1$ 
(in a conscious abuse of notation we use $|\cdot|$ for the norms over 
$\mathbb{R}^d$ or $\mathbb{C}^d$) for $j\in \{1,\ldots,d\}$,  see Theorem~1.3.9~in~\cite{HornJohnson}. 
That is,
$\cQ=UDU^{-1}$, where $D=\mathsf{diag}(q_1,\ldots,q_d)$ and  $U=(w_1| w_2|\cdots| w_{d-1}| w_d)\in \mathbb{C}^{d\times d}$.
\end{enumerate}
\end{remark}

\medskip 
\begin{lemma}[Lyapunov exponent]\label{lem:Lyapunov}
Let $\cQ \in \mathbb{R}^{d\times d}$ be diagonalizable satisfying Hypothesis~\ref{hyp:hyperbolic}. Then
for each $z\in \mathbb{R}^d$ we have 
\begin{equation}\label{eq:cotaarriba}
\|U^{-1}\|^{-2}_{F} \sum_{j=1}^d |q_j|^{2t}  |(U^{-1}z)_j|^2
\leq |\cQ^t z|^2
\leq \|U\|^{2}_{F} \sum_{j=1}^d |q_j|^{2t}  |(U^{-1}z)_j|^2,
\end{equation}
where $(U^{-1}z)_j$ is the $j$-th entry of the vector $U^{-1} z$ and $\|\cdot \|_{F}$ is 
Frobenius norm in $\mathbb{C}^{d\times d}$. 
\end{lemma}
\noindent The proof is given in Appendix~\ref{a:proofs3.2}.

\medskip 

\begin{remark}
If additionally, we assume that $\cQ$ is a normal matrix, then we have
\begin{align}
|\cQ^t z|^2&=|\cQ^t z|^2=|UD^t\widetilde{z}(0)|^2=|D^t\widetilde{z}(0)| =\sum_{j=1}^{d} |\lambda_j|^{2t}|\widetilde{z}(0)|^2.
\end{align}
\end{remark}

\begin{corollary}\label{cor:genericARMA}
Under the hypotheses of Theorem~\ref{thm:genericARMA} we additionally assume that 
$\cQ$ is diagonalizable in the sense of Definition~\ref{def:generic}. Then we have 
\begin{equation}
\begin{split}
&\|U^{-1}\|_{_{\mathrm{F}}} \sqrt{\sum_{j=1}^d |q_j|^{2t}  |(U^{-1}(x-(I-\cQ)^{-1}\Sigma\mathbb{E}[\xi_1]))_j|^2}\\
&\leq  \cW(X_t(x), X_\infty) \\
&\leq 
\begin{cases}
\|U\|_{\mathrm{F}}
\sum_{j=1}^d
 |q_j|^{t}  \mathbb{E}\left[|(U^{-1}(x-X_\infty))_j|
 \right],\\[2mm]
\|U\|_{\mathrm{F}} \sum_{j=1}^d |q_j|^{t}  |(U^{-1}x)_j|+d
\left(\mathbb{E}\left[|\Sigma\xi|^p\right]\right)^{1/p}
\frac{|q_d|^{t+1}}{(1-|q_d|^p)^{1/p}},
\end{cases}
\end{split}
\end{equation}
where 
\begin{equation}
\mathbb{E}[|x-X_\infty|]\leq 
|x|+\frac{K_d\mathbb{E}\left[|\Sigma\xi|\right]|q_d|}{1-|q_d|}.
\end{equation}
\end{corollary}

\begin{proof}
We apply Lemma~\ref{lem:Lyapunov} to the bound \eqref{e:genericARMA}.     
\end{proof}

An important application of our generic bounds are bounds for (independent) parallel sampling. 
We start with the following generic lemma, whose bounds can be treated in the case of autoregressive parallel sampling by Theorem~\ref{thm:genericARMA}.  

\begin{lemma}[Parallel Sampling]\label{thm:parallel}
Consider a family of i.i.d. ergodic stochastic processes $(X_t^{(i)})_{t\geq 0}$, $i=1, \dots, n$ 
having finite $p$-th absolute moments for some $p\geq 1$. Then we have 
\begin{equation}
\begin{split}
\sqrt{n}|\mathbb{E}[X^{(1)}_t-X^{1}_\infty]|
 \leq \cW((X^{(1)}_t, \cdots , X^{(n)}_t ), (X_\infty^{(1)}, \cdots , X_\infty^{(n)}) )
\leq \sqrt{n} \cW(X_t^{(1)}, X^{(1)}_\infty).
\end{split}
\end{equation}
Moreover for $p=2$ we have
\begin{equation}\label{e:Panaretos}
\begin{split}
 \mathcal{W}^2_2((X^{(1)}_t(x), \cdots , X^{(n)}_t(x)) ,(X_\infty^{(1)}, \cdots , X_\infty^{(n)}) )=\sum_{j=1}^n \mathcal{W}^2_2(X^{(j)}_t(x),X_\infty).
\end{split}
\end{equation}
\end{lemma}
\noindent The proof is given in Appendix~\ref{a:proofs3.2}.

\medskip 
\begin{remark}[Parallel autoregressive sampling]\label{rem:parallelsampling}\hfill
\begin{enumerate}
    \item For a sampling of ergodic autoregressive processes satisfying the conditions of Theorem~\ref{thm:genericARMA} the preceding lemma yields non-asymptotic quantitative upper and lower bounds. 
    \item If we assume additionally that $\cQ$ is diagonalizable, the upper and lower bounds can be written explicitly in terms of the eigendecomposition of $\cQ$, see Corollary~\ref{cor:genericARMA}. 
\end{enumerate}
\end{remark}

\begin{corollary}[Ergodic estimates for the empirical process]\label{cor:empirical}
Consider an i.i.d. sequence of autoregressive processes $(X^{(j)}_t)_{t\geq 0}$ for $j\in \mathbb{N}$ satisfying the hypotheses of Theorem~\ref{thm:genericARMA}. 

Then the empirical mean $S_t^{(n)} := \frac{1}{n}\sum_{j=1}^n X^{(j)}_t(x)$ satisfies the following 
multivariate autoregressive equation 
\begin{equation}\label{e:empirical}
S^{(n)}_{t+1}=\cQ S^{(n)}_t+\Sigma \zeta^{(n)}_{t+1}, \quad t\in \mathbb{N}_0,
\end{equation}
where $(\zeta^{(n)}_t)_{t\geq 0}$ is an i.i.d. sequence given by 
\[
\zeta^{(n)}_t=\frac{1}{n}\sum_{j=1}^{n} \xi^{(j)}_t,\quad t\in \mathbb{N}\quad \textrm{ and }\quad S^{(n)}_0=x.
\]
Hence it satisfies the following bounds
\begin{equation}
\begin{split}
\left |\cQ^t(x-(I-\cQ)^{-1}\mathbb{E}[\Sigma \zeta^{(n)}_1])\right |
&\leq  \cW(S^{(n)}_t(x),S^{(n)}_\infty) \leq 
\begin{cases}
\mathbb{E}[|\cQ^t(x-S^{(n)}_\infty)|]\leq K_d \|\cQ\|_*^t \mathbb{E}[|x-S^{(n)}_\infty|],\\[2mm]
|\cQ^t x|+K_d \|\Sigma\|_F
\left( \mathbb{E}\left[|\xi^{(1)}_1|^p \right]\right)^{1/p}
\frac{\|\cQ\|^{t+1}_{*}}{(1-\|\cQ\|^p_{*})^{1/p}},
\end{cases}
\end{split}
\end{equation}
where $S^{(n)}_\infty = \frac{1}{n} \sum_{j=1}^n X_\infty^{(j)}$. 
\end{corollary}
\begin{proof}[Proof of Corollary \ref{cor:empirical}]
Equation~\ref{e:empirical} is satisfied by direct calculation due to the linearity. 
Hence, Theorem~\ref{thm:genericARMA} implies the statement. It only remains to estimate 
$(\mathbb{E}[|\zeta^{(n)}_1|^p])^{1/p}$. Note that $\mathbb{E}[\zeta^{(n)}_t]=\mathbb{E}[ \xi^{(1)}_1]$ and using  the Minkowski inequality we have
\[
(\mathbb{E}[|\zeta^{(n)}_1|^p])^{1/p}\leq 
(\mathbb{E}[|\xi_1^{(1)}|^p])^{1/p}\quad \textrm{ for all }\quad n\in \mathbb{N}.
\]
\end{proof}

\medskip 
\begin{remark}
In case of $\cQ$ being additionally diagonalizable, the bounds of Corollary~\ref{cor:empirical} can be sharpened with the help of Corollary~\ref{cor:genericARMA}.
\end{remark}

\bigskip 
\subsection{\textbf{The asymptotics of Schur stable matrix power projections}}\label{ss:matrixpowerprojection}
For completeness we provide the following non-asymptotic estimates for the matrix powers of Jordan matrices. It is a basic result in Linear Algebra, that for any $\cQ\in \mathbb{R}^{d\times d}$ there exists the Jordan normal form of $J_{\mathcal{Q}}$, that is, there exists an invertible matrix $P$ such that 
\[
P \mathcal{Q} P^{-1} = J_{\mathcal{Q}}, 
\]
where $J_{\mathcal{Q}} = \mathsf{diag}(J_{d_1}(\lambda_1), \dots J_{d_\ell}(\lambda_\ell))$, $\ell\in  \{1, \dots, d\}$, $\lambda_1,\dots, \lambda_d\in \mathbb{C}$ y $\sum_{i=1}^\ell d_i = d$, 
see \cite[Theorem~3.1.11]{HornJohnson}. Moreover for any $t\in \mathbb{N}$ 
\[
\mathcal{Q}^t = P^{-1} J_{\mathcal{Q}}^t P = P^{-1}~ \mathsf{diag}(J_{d_1}(\lambda_1)^t, \dots J_{d_\ell}(\lambda_\ell)^t)\,P. 
\]
In this section we complete the quantification of the ergodic bounds in the main results of this article in Theorem~\ref{thm:affine}, Theorem~\ref{thm:genericARMA} and the respective corollaries. 
For this sake we show quantitative bounds on $|\mathcal{Q}^t x|$, $x\neq 0$, for $\mathcal{Q}$ having Jordan normal. 

\medskip 
\begin{remark}
Clearly, for $\mathcal{Q}\in \mathbb{R}^{d\times d}$ being a single Jordan block, 
the eigenvalue is necessarily real (complex eigenvalues always come in pairs of complex conjugate pairs). 
However, this setting makes perfectly sense for $\mathcal{Q}\in \mathbb{C}^{d\times d}$. In fact, this result is the main building block for the general result in $\mathcal{Q}\in \mathbb{R}^{d\times d}$ with more than one Jordan block, for which complex eigenvalues may appear. 

\end{remark}

\begin{lemma}[Non-asymptotic estimates of the matrix power of a Jordan block]\hfill\\
\label{lem:jordanlast}
Assume $\mathcal{Q}\in \mathbb{C}^{d\times d}$ consists of a single Jordan block with eigenvalue $q_1 = q = r e^{i\theta}\neq 0$.  
Then for $t\geq 2(d-1)$ and $x\in \mathbb{R}^d$ with $x_d \neq 0$ we have 
\begin{equation}\label{e:Jaraquant}
\begin{split}
&\Big|\frac{\mathcal{Q}^t x}{ |q|^{t-(d-1)} \binom{t}{d-1} } -x_{d} e^{i(t-(d-1))\theta}  e_1^T\Big|\\
&\leq \frac{d-1}{t-d+2} \sqrt{ \Big(\sum_{j=0}^{d-2} |q|^{(d-1)-j} x_{j+1}\Big)^2 + \sum_{k=2}^{d} \Big(\sum_{j=0}^{d-k}  |q|^{(d-1)-j} x_{j+k}\Big)^2}, 
\end{split}
\end{equation}
where $e_1$ is the first canonical unit vector in $\mathbb{R}^d$. 
\end{lemma}

\medskip 
\begin{remark}
Note that 
\begin{align*}
&\lim_{t\to\infty} \Big|\Big(\sum_{j=0}^{d-2} \frac{\binom{t}{j}}{\binom{t}{d-2}}  q^{(d-1)-j} x_{j+1}, 
\sum_{j=0}^{d-2} \frac{\binom{t}{j}}{\binom{t}{d-2}} q^{(d-1)-j} x_{j+2}, \dots, 
\sum_{j=0}^{d-d} \frac{\binom{t}{j}}{\binom{t}{d-2}} q^{(d-1)-j}x_{j+d}\Big)^T\Big|\\
&= |q| \sqrt{ x_{d-1}^2 +  x_{d}^2}. 
 \end{align*}
The preceding limit illustrates, that whenever $x_d\neq 0$ the error term on the right-hand side of \eqref{e:Jaraquant} is non-zero with the asymptotic rate of order $t^{-1}$, since $x_d$ appears for $k=2$ and $j= d-2$.  
\end{remark}

\begin{corollary}\label{cor:Jara1}
Assume $\mathcal{Q}\in \mathbb{C}^{d\times d}$ consists of a single Jordan block with eigenvalue $q_1 = q = r e^{i\theta} \neq 0$, $\theta \in [0, 2\pi)$.  
Then for $x\in \mathbb{R}^d$, $x\neq 0$, and $j_* := \max\{j\in \{1, \dots, d\}~|~|x_j|> 0\}$ we have 
for $t\geq \max\{d-1, 2(j_*-2)\}$
\begin{equation}\label{e:Jaraquant2}
\begin{split}
&\Big|\frac{\mathcal{Q}^t x}{ |q|^{t-(j_*-1)} \binom{t}{j_*-1} } -x_{j_*} e^{i(t-(j_*-1))\theta}  e_{d-j_*+1}^T\Big|\\
&\leq \frac{j_*-1}{t-j_*+2} \sqrt{ \Big(\sum_{j=0}^{j_*-2} |q|^{(j_*-1)-j} x_{j+1}\Big)^2 + \sum_{k=2}^{j_*} \Big(\sum_{j=0}^{j_*-k}  |q|^{(j_*-1)-j} x_{j+k}\Big)^2}, 
\end{split}
\end{equation}
where $e_1, \dots, e_d$ is the canonical basis of $\mathbb{R}^d$. 
\end{corollary}

\medskip 
\begin{remark}
We point out that in~\eqref{e:Jaraquant} and~\eqref{e:Jaraquant2} the respective vectors 
\[x_{d} e^{i(t-(d-1))\theta}  e_1^T, \quad 
x_{j_*} e^{i(t-(j_*-1))\theta}  e_{d-j_*+1}^T, 
\]
do not have a limit for $\theta \in (0, 2\pi)$ as $t\to \infty$, however, their respective norms are constant, that is to say, 
\[
|x_{d} e^{i(t-(d-1))\theta}  e_1^T| = |x_{d}| >0,\quad  |x_{j_*} e^{i(t-(j_*-1))\theta}  e_{d-j_*+1}^T| = |x_{j_*}|>0.
\] Therefore, the asymptotic behaviors of 
$|\mathcal{Q}^t x|$ are given by 
\[
|q|^{t-(j_*-1)} \binom{t}{j_*-1} |x_{d}|,\quad |q|^{t-(j_*-1)} \binom{t}{j_*-1} |x_{j_*}|.
\]
\end{remark}

\begin{proof}[Proof of Lemma~\ref{lem:jordanlast}]
Recall that for $t\geq d-1$ 
\begin{align*}
\mathcal{Q}^t x 
&=  \Big(\sum_{j=0}^{d-1} \binom{t}{j} q^{t-j} x_{j+1}, 
\sum_{j=0}^{d-2} \binom{t}{j} q^{t-j} x_{j+2}, \dots, 
\sum_{j=0}^{d-d} \binom{t}{j} q^{t-j}x_{j+d}\Big)^T.
\end{align*}
By definition of $j_*$ we have 
\begin{align*}
\mathcal{Q}^t x 
&=  \Big(\sum_{j=0}^{j_*-1} \binom{t}{j} q^{t-j} x_{j+1}, 
\sum_{j=0}^{j_*-2} \binom{t}{j} q^{t-j} x_{j+2}, \dots, 
\sum_{j=0}^{j_*-d} \binom{t}{j} q^{t-j}x_{j+d}\Big)^T, 
\end{align*}
where the sum $\sum_0^{k}$, $k<0$, is defined to be empty. 
Then for $t\geq d-1$ we have 
\begin{align*}
&\Big|\frac{\mathcal{Q}^t x}{ |q|^{t-(j_*-1)} \binom{t}{j_*-1} } -x_{d} e^{i(t-(j_*-1))\theta}  e_1^T\Big|\\
&= \frac{\binom{t}{j_*-2}}{\binom{t}{j_*-1}} 
\Big|\Big(\sum_{j=0}^{j_*-2} \frac{\binom{t}{j}}{\binom{t}{j_*-2}}  q^{(j_*-1)-j} x_{j+1}, 
\sum_{j=0}^{j_*-2} \frac{\binom{t}{j}}{\binom{t}{j_*-2}} q^{(j_*-1)-j} x_{j+2}, \dots, 
\sum_{j=0}^{j_*-d} \frac{\binom{t}{j}}{\binom{t}{j_*-2}} q^{(j_*-1)-j}x_{j+j_*}\Big)^T\Big|.
\end{align*}
Note that due to monotonicity we have for $t\geq  2 (j_*-2)$ that 
\[ 
\frac{\binom{t}{j}}{\binom{t}{j_*-2}}\leq 1 \quad j=0, \dots, j_*-2. 
\]
Hence 
\begin{align*}
&\sup_{t\geq  2 (j_*-2)} \Big|\Big(\sum_{j=0}^{j_*-2} \frac{\binom{t}{j}}{\binom{t}{j_*-2}}  q^{(j_*-1)-j} x_{j+1}, 
\sum_{j=0}^{j_*-2} \frac{\binom{t}{j}}{\binom{t}{j_*-2}} q^{(j_*-1)-j} x_{j+2}, \dots, 
\sum_{j=0}^{j_*-d} \frac{\binom{t}{j}}{\binom{t}{j_*-2}} q^{(j_*-1)-j}x_{j+j_*}\Big)^T\Big|\\ 
&= \sqrt{ \Big(\sum_{j=0}^{j_*-2} |q|^{(j_*-1)-j} x_{j+1}\Big)^2 + \sum_{k=2}^{j_*} \Big(\sum_{j=0}^{j_*-k}  |q|^{(j_*-1)-j} x_{j+k}\Big)^2}.
\end{align*}
\end{proof}

\medskip 

\begin{remark}
Assume $\mathcal{Q}\in \mathbb{R}^{d\times d}$, $d= 2N$, $N\in \mathbb{N}$, consists of a single pair of Jordan blocks with eigenvalue complex-valued $q_\pm = r e^{\pm i\theta} \neq 0$, $\theta \in [0, 2\pi)$.  
Observe if $x\neq 0$, $x = (x_{+},x_{-})$, with $x_{\pm}\in \mathbb{R}^N$. If $x_{+}= 0$ or $x_{-} = 0$, then 
the computation of $Q^t x$ is reduced to the case of Corollary~\ref{cor:Jara1}. Hence in the subsequent lemma we assume that $x\neq 0$ and $x_{+}\neq 0$ or $x_{-} \neq 0$. 
Define 
\begin{equation}\label{d:jmasmenos}
\begin{split}
 j_+ &:= \{\max\{j\in \{1, \dots, N\}~|~|x_{+,j}|> 0\},\\
 j_- &:= \{\max\{j\in \{1, \dots, N\}~|~|x_{-,j}|> 0\},\\
 j_* &:= \max\{j_+, j_-\}.
 \end{split}
\end{equation}
\end{remark}

\medskip 
\begin{corollary}
Assume $\mathcal{Q}\in \mathbb{C}^{d\times d}$, $d= 2N$, $N\in \mathbb{N}$, consists of a single pair of Jordan blocks with eigenvalue complex-valued $q_\pm = r e^{\pm i\theta} \neq 0$, $\theta \in [0, 2\pi)$.  
Then for $x\in \mathbb{R}^d$, $x = (x_+, x_-)$, $x_+, x_-\neq 0$, and 
$j_*$ defined in \eqref{d:jmasmenos} we have 
for $t\geq \max\{d-1, 2(j_+-2)\}$
\begin{equation}\label{e:Jaraquant232}
\begin{split}
&\Big|\frac{\mathcal{Q}^t x}{ |q|^{t-(j_+-1)} \binom{t}{j_+-1} } 
-\Big(\mathbf{1}\{j_*=j_+ \} x_{+,j_*}   e^{i(t-(j_*-1))\theta}  e_{N-j_*+1}^T + \mathbf{1}\{j_* = j_-\}x_{-,j_*} e^{-i(t-(j_*-1))\theta}  e_{2N-j_*+1}^T\Big)\Big|\\
&\leq \frac{j_+-1}{t-j_++2} \sqrt{ \Big(\sum_{j=0}^{j_+-2} |q|^{(j_+-1)-j} x_{j+1}\Big)^2 + \sum_{k=2}^{j_+} \Big(\sum_{j=0}^{j_+-k}  |q|^{(j_+-1)-j} x_{j+k}\Big)^2}, 
\end{split}
\end{equation}
where $e_1, \dots, e_d$ is the canonical basis of $\mathbb{R}^d$. 
\end{corollary}

\bigskip 
\subsection{\textbf{Numerical experiments}}\label{ss:numerical}

In this section, we experimentally validate the obtained theoretical bounds in the previous sections. 

\subsubsection{\textbf{Autoregressive Moving-Average Process $\mathsf{ARMA}(2,1)$}}

\medskip
    
\noindent Consider an $\mathsf{ARMA}(2,1)$ as defined in \eqref{d:arma} with parameters $\varphi_1 = 0.6$, $\varphi_2 = 0.3$, and $\theta_1 = 0.5$ and standard normal perturbations. The model is embedded in a 3-dimensional state-space using the following augmented companion matrix
    \[
    \mathcal{Q} = \begin{pmatrix}
        0.6 & 0.3 & 0.5 \\
        1 & 0 & 0 \\
        0 & 0 & 0
    \end{pmatrix}.
    \]
The initial state is given by $x = (1, -1, 1.5)$ and the slicing vector is given by $v = (1, 1, 1)$. For the matrix $\mathcal{Q}$ the associated eigenvalues are
    \[
        \lambda_1 \approx 0.9245,  \qquad \lambda_2 \approx -0.3245,  \qquad \lambda_3 = 0.
    \]  
Figure \ref{fig:arma} compares the precise values of the Wasserstein-2 distance with the upper and lower bounds obtained in 
Theorem \ref{th:affineGauss}.
    
    \begin{figure}[H]
        \centering
        \includegraphics[width=0.85\textwidth]{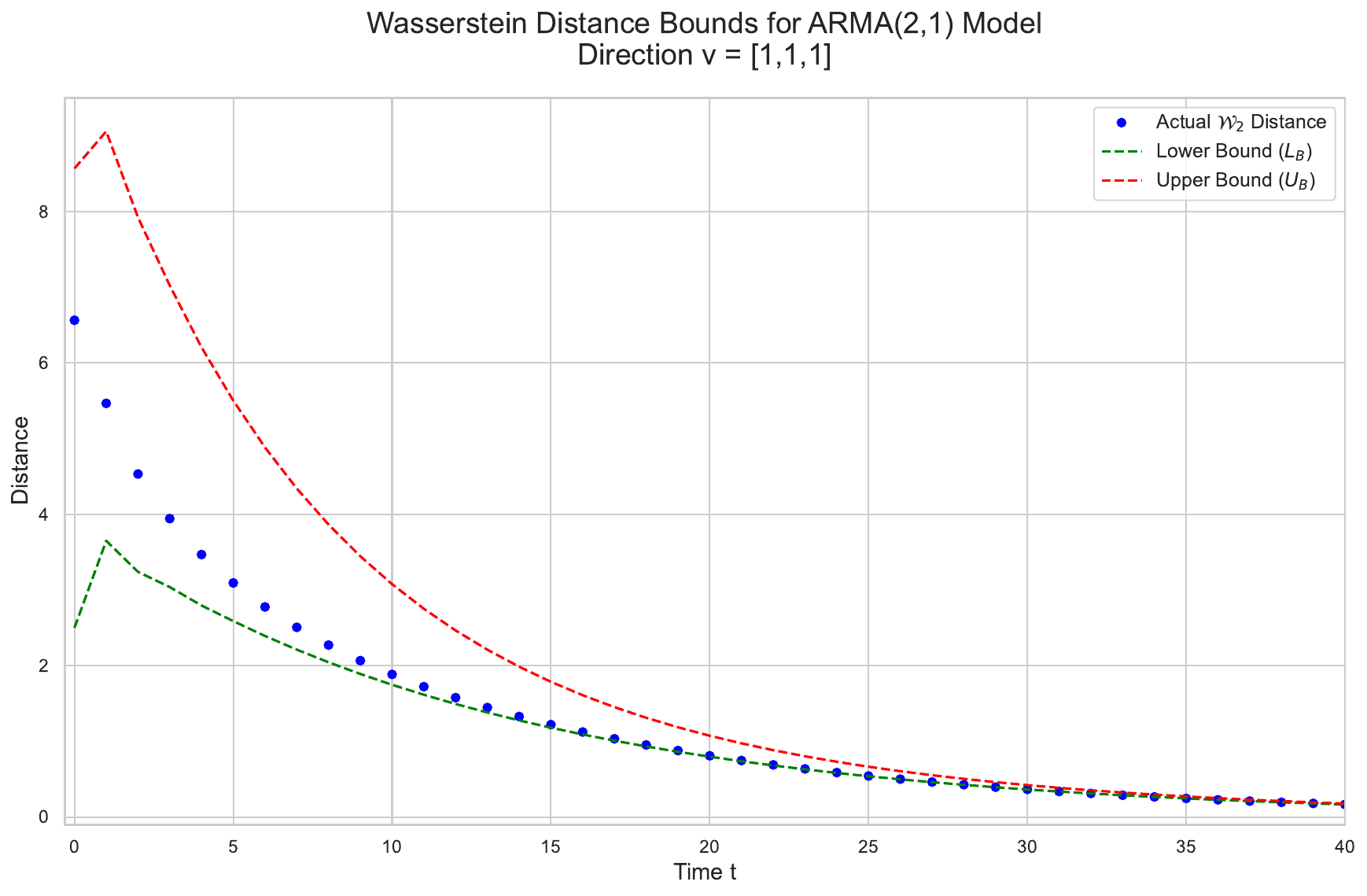}
        \caption{$\mathcal{W}_2$ distance bounds for the $\mathsf{ARMA}(2,1)$ system.}
        \label{fig:arma}
    \end{figure}

\bigskip
\subsubsection{\textbf{The gradient case: diagonal and symmetric systems}}

\hfill\\
\begin{figure}[h!]
    \begin{minipage}{.4\textwidth} 
        \textbf{Decoupled interaction matrix $\mathcal{Q}$:} 
        \[
    \mathcal{Q} = \begin{pmatrix}
        0.8 & 0 & 0 \\
        0 & 0.5 & 0 \\
        0 & 0 & -0.3
    \end{pmatrix}
        \]
        with initial state  $x = (2, -1, 1.5)$ and slicing vector $v = (1, 1, 1)$. 
    \end{minipage}
    \hfill 
    \begin{minipage}{.4\textwidth} 
        \textbf{Symmetrically coupled matrix $\mathcal{Q}$:}  
    \begin{equation*}
    \begin{split}
        \mathcal{Q} &= \frac{1}{90}\begin{pmatrix}
            49 & -28 & -10 \\
            -28 & 25 & -38 \\
            -10 & -38 & 16
        \end{pmatrix} \\
        &\approx \begin{pmatrix}
            0.5444 & -0.3111 & -0.1111 \\
            -0.3111 & 0.2778 & -0.4222 \\
            -0.1111 & -0.4222 & 0.1778
        \end{pmatrix}
    \end{split}
\end{equation*}
    with initial state $x = (2, -1, 1.5)$ and a slicing vector $v = (1, 1, 1)$. The eigenvalues of $\mathcal{Q}$ are given by
\[
        \lambda_1 = 0.8,\qquad \lambda_2 = 0.5, \qquad \lambda_3 = -0.3.
\]
    \end{minipage} 
\end{figure}
\hfill 
    
    \begin{figure}[H]
 \begin{minipage}{.48\textwidth} 
      \centering
        \includegraphics[width=0.85\textwidth]{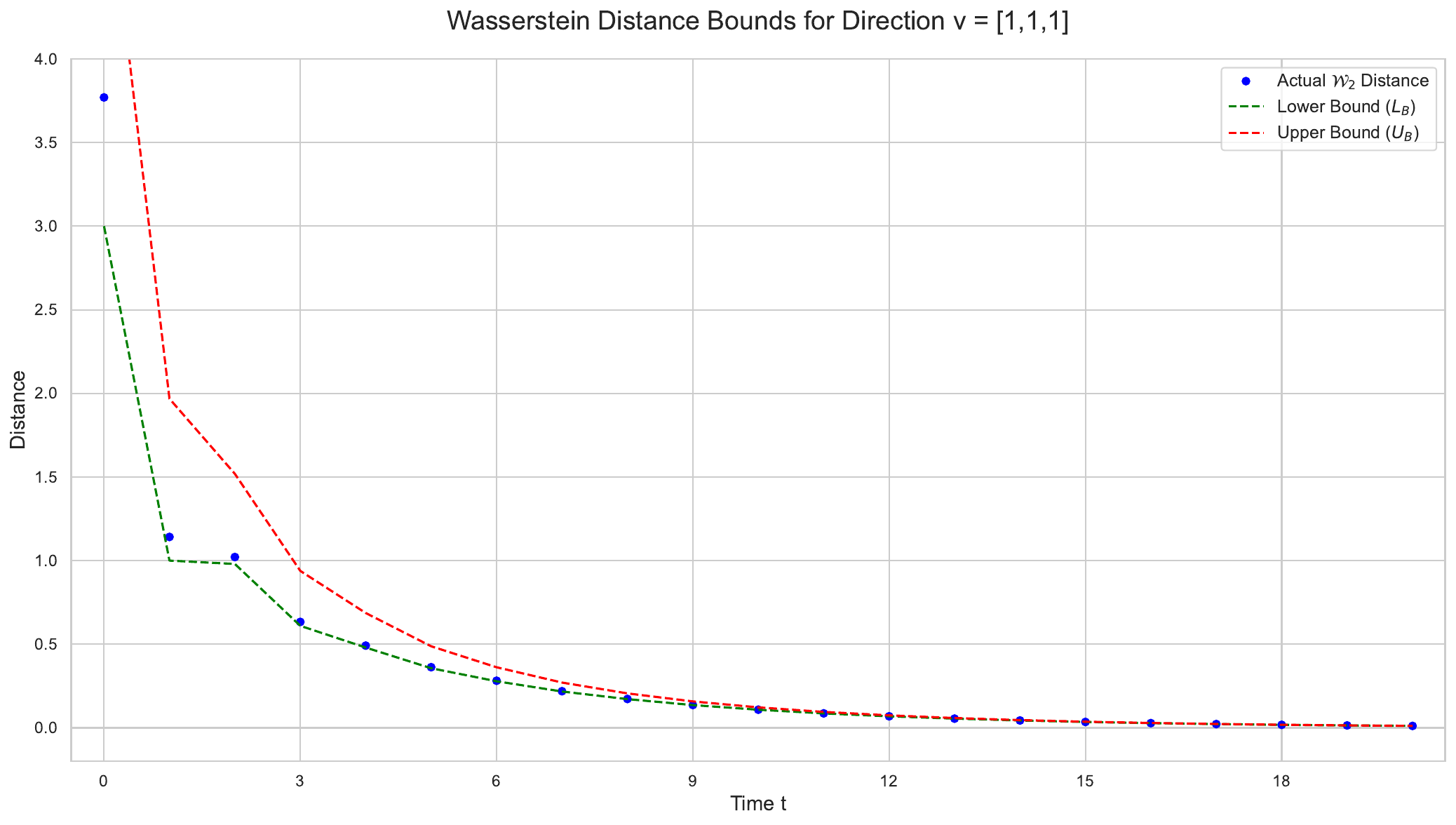}
        \caption{$\mathcal{W}_2$ distance bounds for a diagonal system.}
        \label{fig:diagonal}
\end{minipage}
\hfill
    \begin{minipage}{.48\textwidth} 
        \centering
        \includegraphics[width=0.8\textwidth]{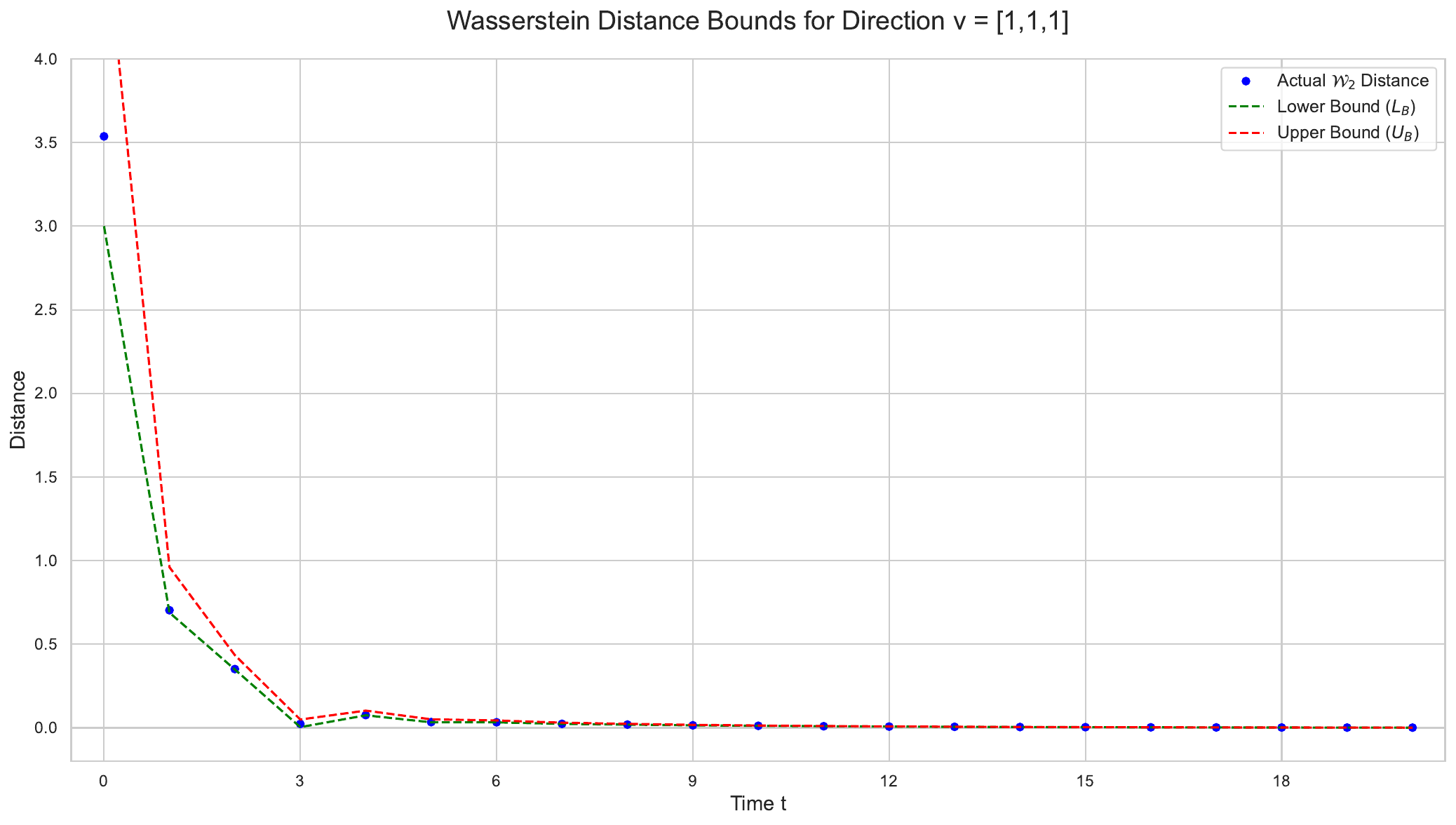}
        \caption{$\mathcal{W}_2$ distance bounds for a non-diagonal symmetric system with the same eigenvalues as in Figure 
        \ref{fig:diagonal}.}
        \label{fig:symmetric_general}
\end{minipage}
    \end{figure}

\subsubsection{\textbf{Non-diagonalizable system (3d, 1 Jordan block)}}
We compare the differences when the eigenvalues are
small against bigger ones.
\begin{figure}[h!]
    \begin{minipage}{.48\textwidth} 
        \centering 
        \textbf{Case 1.} The simulated system matrix is $\mathcal{Q}_1$:
        \[
        \mathcal{Q}_1 = \begin{pmatrix}
            0.2 & 1 & 0 \\
            0 & 0.2 & 1 \\
            0 & 0 & 0.2
        \end{pmatrix}.
        \]
    \hfill\
    \hfill\\[2cm]
    \end{minipage}
    \hfill 
    \begin{minipage}{.48\textwidth} 
        \centering
        \textbf{Case 2.} The simulated system matrix is $\mathcal{Q}_2$:
        \[
        \mathcal{Q}_2 = \begin{pmatrix}
            0.9 & 1 & 0 \\
            0 & 0.9 & 1 \\
            0 & 0 & 0.9
        \end{pmatrix}.
        \]
    \end{minipage} 
\end{figure}    
The initial state is $x = (1, 1, 1)$. We see that the behavior changes considerably in the case of small eigenvalues and
larger ones. For small eigenvalues the convergence is fast comparing a smaller eigenvalue (in complex modulus) with a bigger one, see Figure
\ref{fig:jordan3d_combinadas}.
    
   \begin{figure}[!htbp]
    \centering
    \begin{subfigure}[b]{0.48\textwidth}
        \centering
        \includegraphics[width=\textwidth]{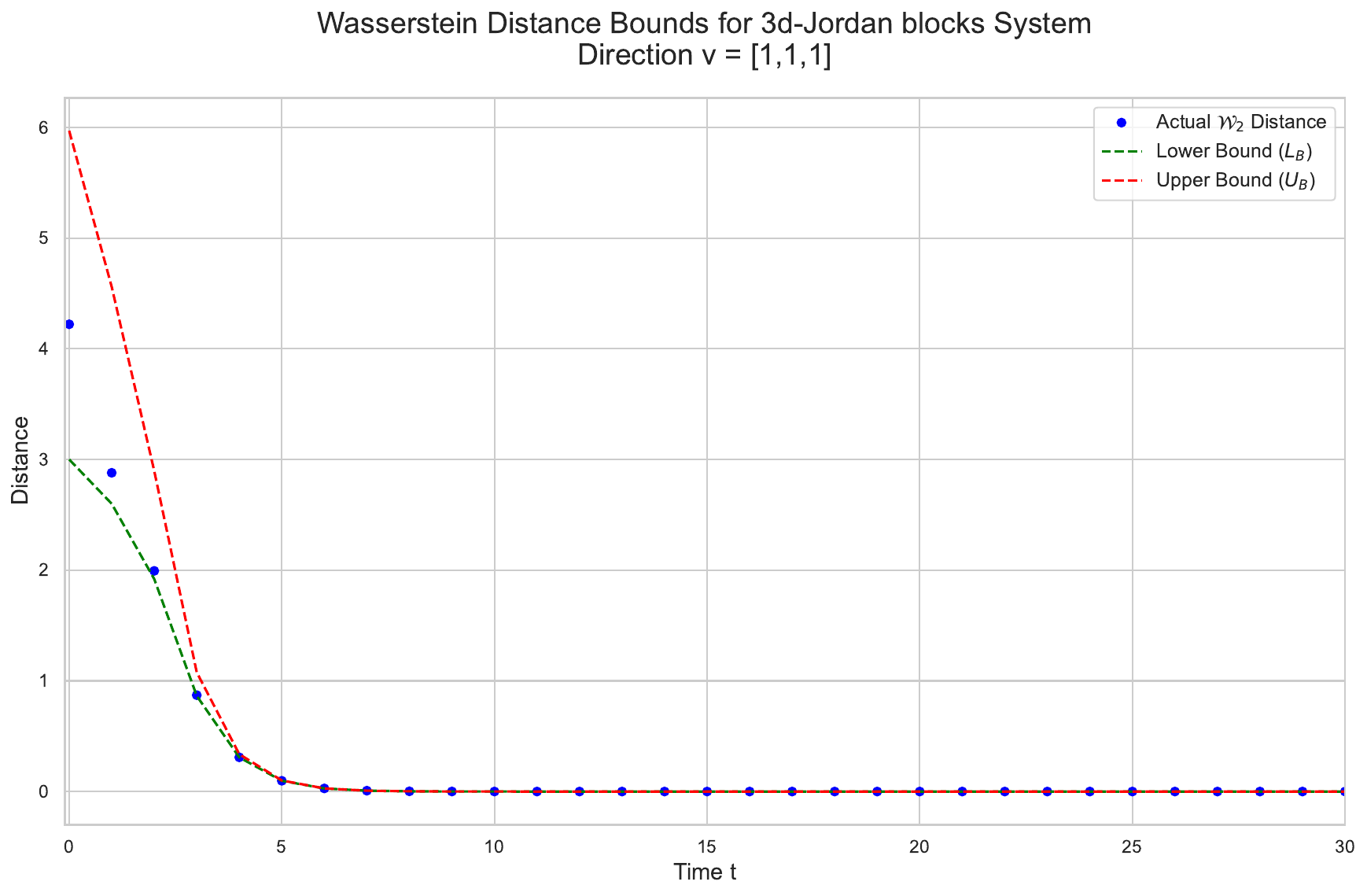}
        \caption{3d-System: ``small eigenvalues''.}
        \label{fig:jordan3d_small} 
    \end{subfigure}
    \hfill 
    \begin{subfigure}[b]{0.48\textwidth}
        \centering
        \includegraphics[width=\textwidth]{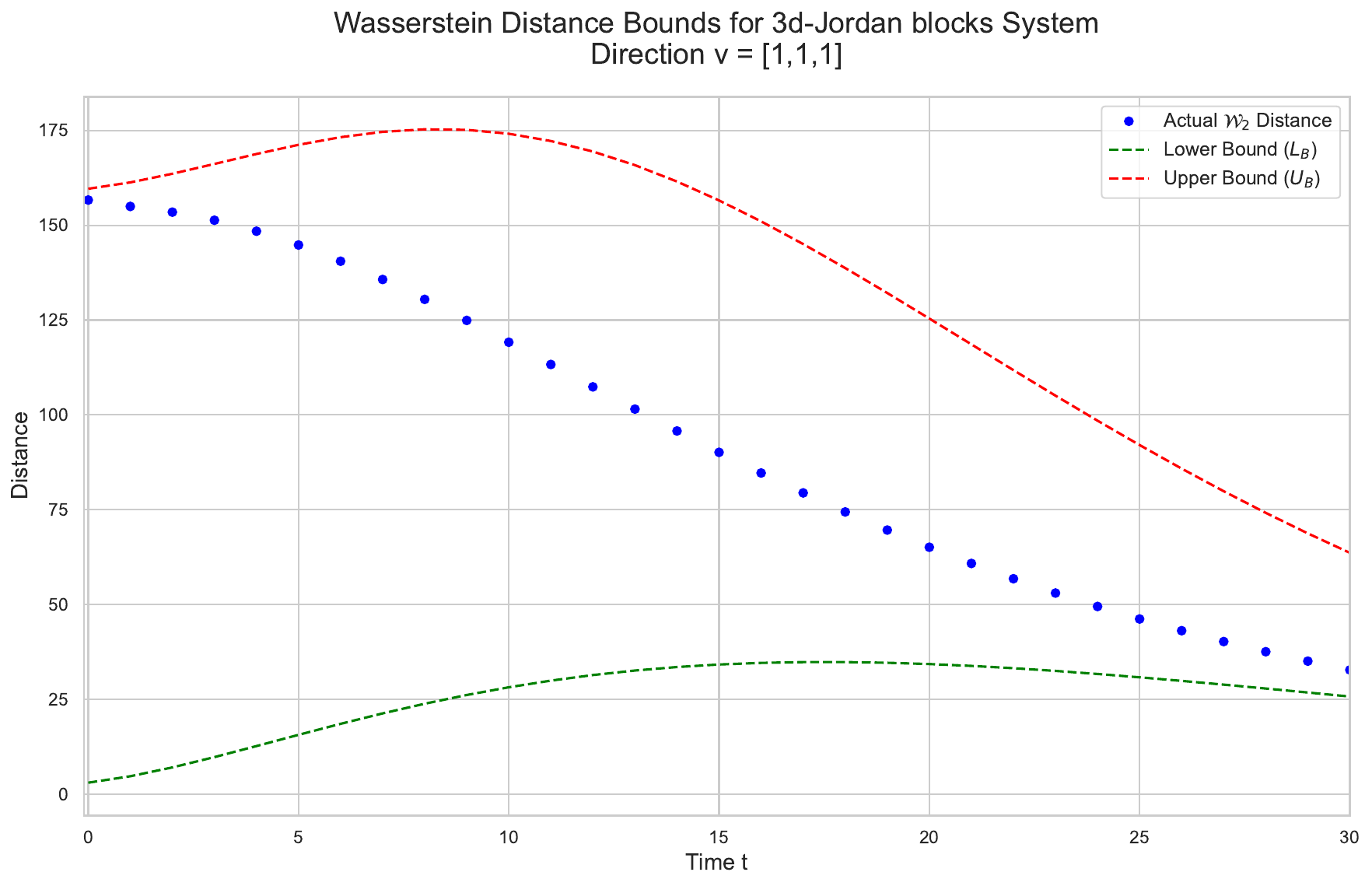}
        \caption{3d-System: ``big eigenvalues''.}
        \label{fig:jordan3d_big} 
    \end{subfigure}
    \caption{Wasserstein for 3d non-diagonalized Systems.}
    \label{fig:jordan3d_combinadas} 
\end{figure}

\subsubsection{\textbf{Oscillatory system (complex eigenvalues) for $\mathcal{W}_2$}}
\label{ssec:oscillatory}
We simulate a 3d-system exhibiting oscillatory behavior. The parameters are set to $r=2$ (for the $\mathcal{W}_2$ distance), $v = (1, 1, 1)$, $x = (1, 1, 1)$.
The dynamics is governed by the interaction matrix $\mathcal{Q}$
\[
\mathcal{Q} = \begin{pmatrix}
0.8 & -0.1 & 0.0 \\
0.1 & 0.8 & 0.0 \\
0.0 & 0.0 & 0.5
\end{pmatrix}.
\]
This system has oscillatory dynamics because the upper-left $2 \times 2$ block has the complex conjugate eigenvalues $\lambda \approx 0.8 \pm 0.1\ii$. We note that
$|\lambda| \approx 0.806$, see Figure~\ref{fig:oscillatory}.
\begin{figure}[H]
    \centering
    \includegraphics[width=0.8\textwidth]{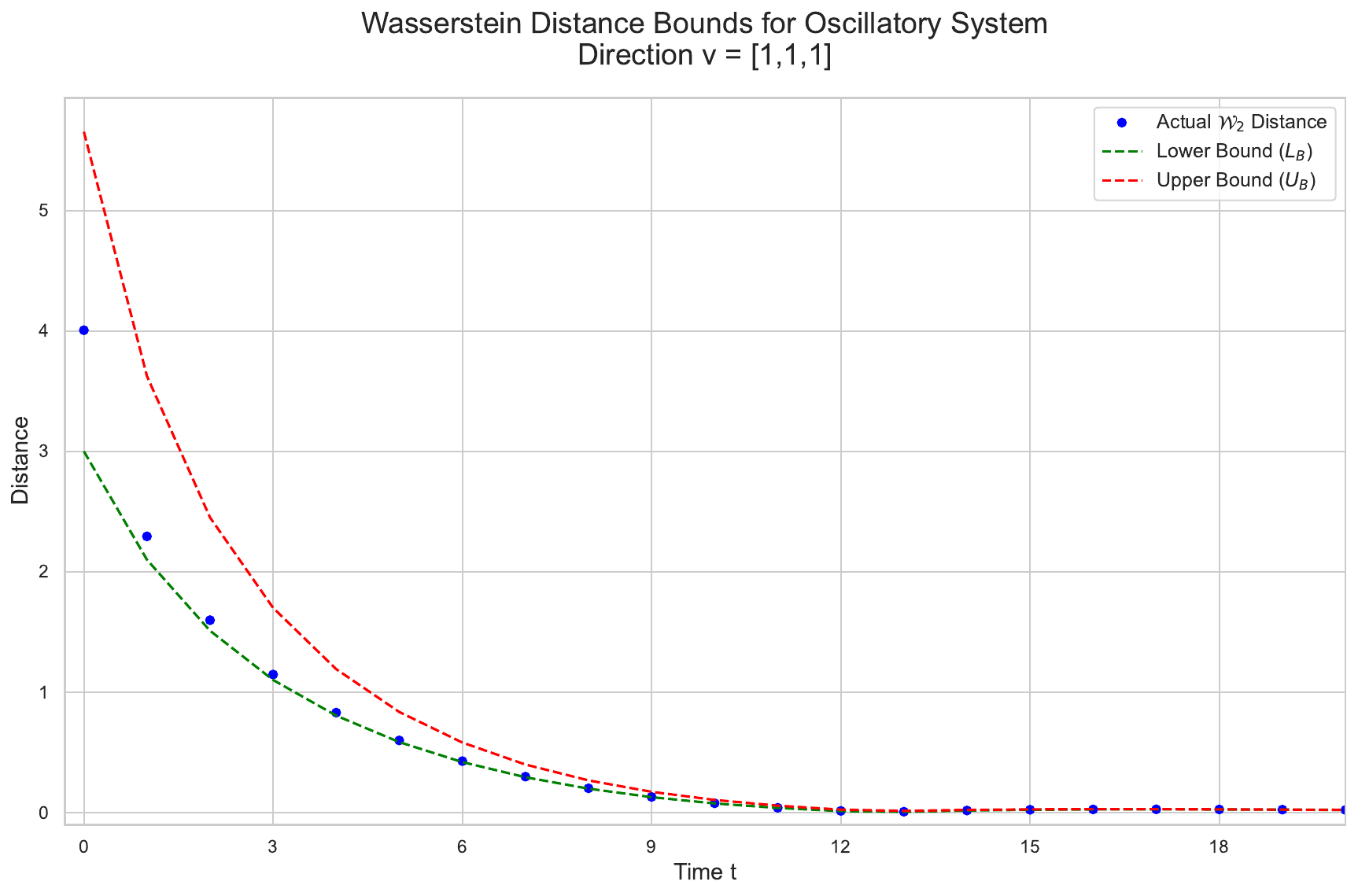}
    \caption{Bounds validation for the oscillatory model.}
    \label{fig:oscillatory}
\end{figure}    
    
\subsubsection{\textbf{Non-diagonalizable 4d Jordan system for $\mathcal{W}_2$}}
We analyze a stable 4d-system constructed with two $2\times2$ Jordan blocks, which can exhibit significant transient growth. The interaction matrix $\mathcal{Q}$ is given by
    \[
    \mathcal{Q} = \begin{bmatrix}
        0.9 & 1 & 0 & 0 \\
        0 & 0.9 & 0 & 0 \\
        0 & 0 & 0.8 & 1 \\
        0 & 0 & 0 & 0.8
    \end{bmatrix}.
    \]
The initial state  is given by $x = (1, 0, 1, 0)$ and the slicing vector by $v = (1, 1, 1, 1)$. 
    \begin{figure}[H]
        \centering
        \includegraphics[width=0.8\textwidth]{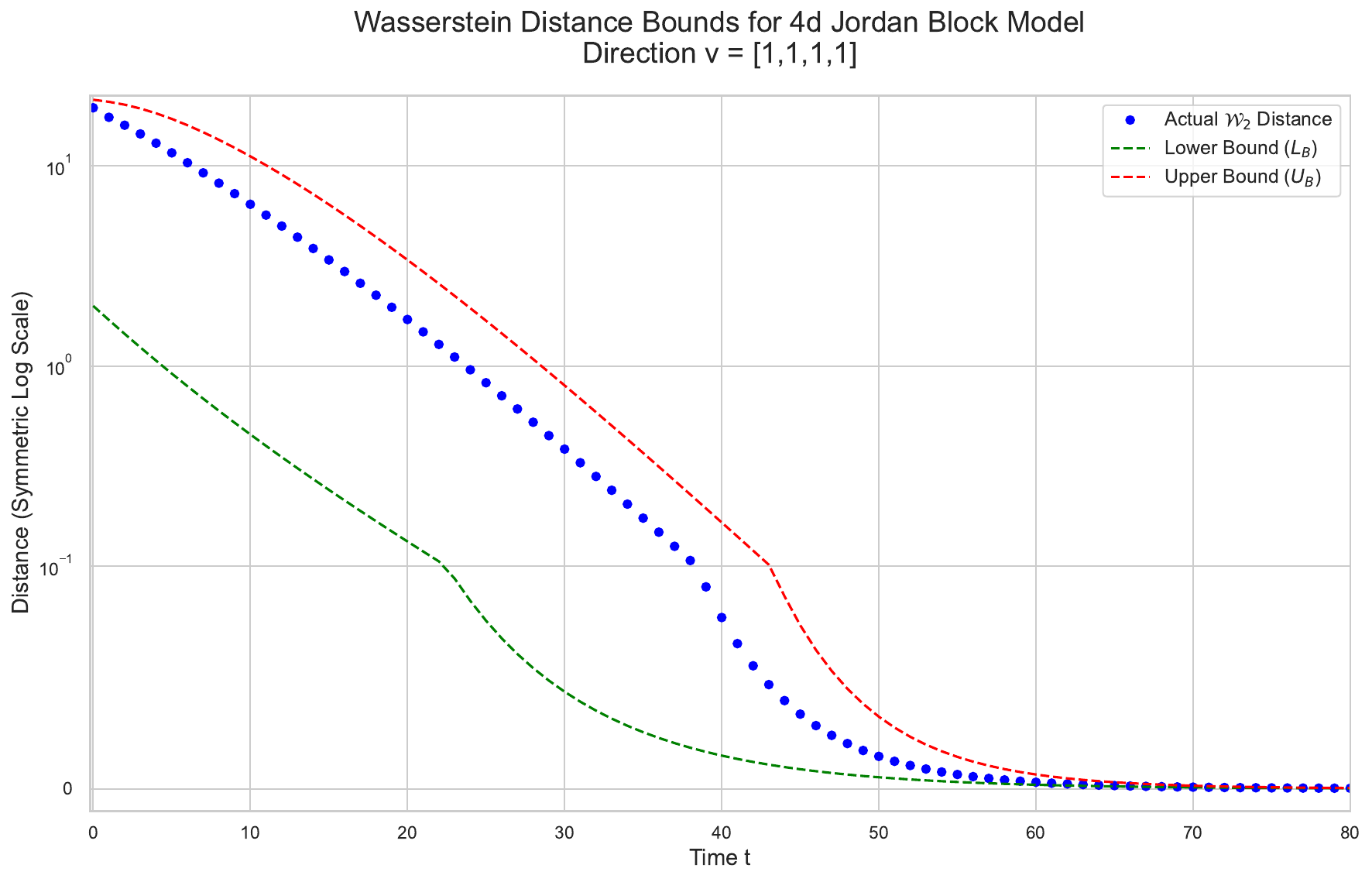}
        \caption{Wasserstein distance bounds for the 4d non-diagonalizable system.}
        \label{fig:jordan4d}
    \end{figure}
    
\subsubsection{\textbf{Non-diagonalizable 4d Jordan system for $\mathcal{W}_r$}}
In comparison to the previous cases, we compute the bound of $\mathcal{W}_r$ for $r=3$ and $r=4$. In general, there is
no closed formula. In 1d
the Wasserstein distance, $\mathcal{W}_r(\mathbb{P}_1, \mathbb{P}_2)$ for any $r\geq 1$, can be computed 
 as follows
\[
\mathcal{W}_r(\mathbb{P}_1, \mathbb{P}_2) = \left( \int_0^1 |F^{\leftarrow}_{1}(u) - F^{\leftarrow}_2(u)|^r \ud u \right)^{1/r},
\]
where $F^{\leftarrow}_i$, $i=1,2$ is the quantile function of the respective distribution of $\mathbb{P}_i$.
That is,
\[
F^{\leftarrow}_i(u) = \inf \{x\in \mathbb{R}:\, F_i(x):=\mathbb{P}_i((-\infty,x])\geq  u\},\quad 0 < u < 1,
\]
see for instance \cite{ Santambrogio2015,Villani03}.
Straightforward computations, see Lemma~\ref{lem:loc-scale}, yield for any $r\geq 1$ 
\begin{align*}
    \mathcal{W}_r(\mathcal{N}(m, \sigma^2),\mathcal{N}(0, \sigma_{\infty}^2))
                            & = \left( \int_0^1 |m + (\sigma - \sigma_\infty) \Phi^{-1}(u)|^r \ud u \right)^{1/r},
\end{align*}
where $\Phi$ is the cumulative distribution function of the standard Normal distribution.
In the sequel, we compute $\mathcal{W}_r$ for $r=1,3,4$. In general, the integrals of such type  can be evaluated using numerically integration, see \cite{Quarteroni2007}.
For initial state $x = (1, 0, 1, 0)$ and a slicing vector $v = (1, 1, 1, 1)$ we obtain the following results in  Figure \ref{fig:jordan_4d_all_r}.

\begin{figure}[!htbp]
\centering 
    \begin{subfigure}[b]{0.7\textwidth} 
        \centering
    \includegraphics[width=\textwidth]{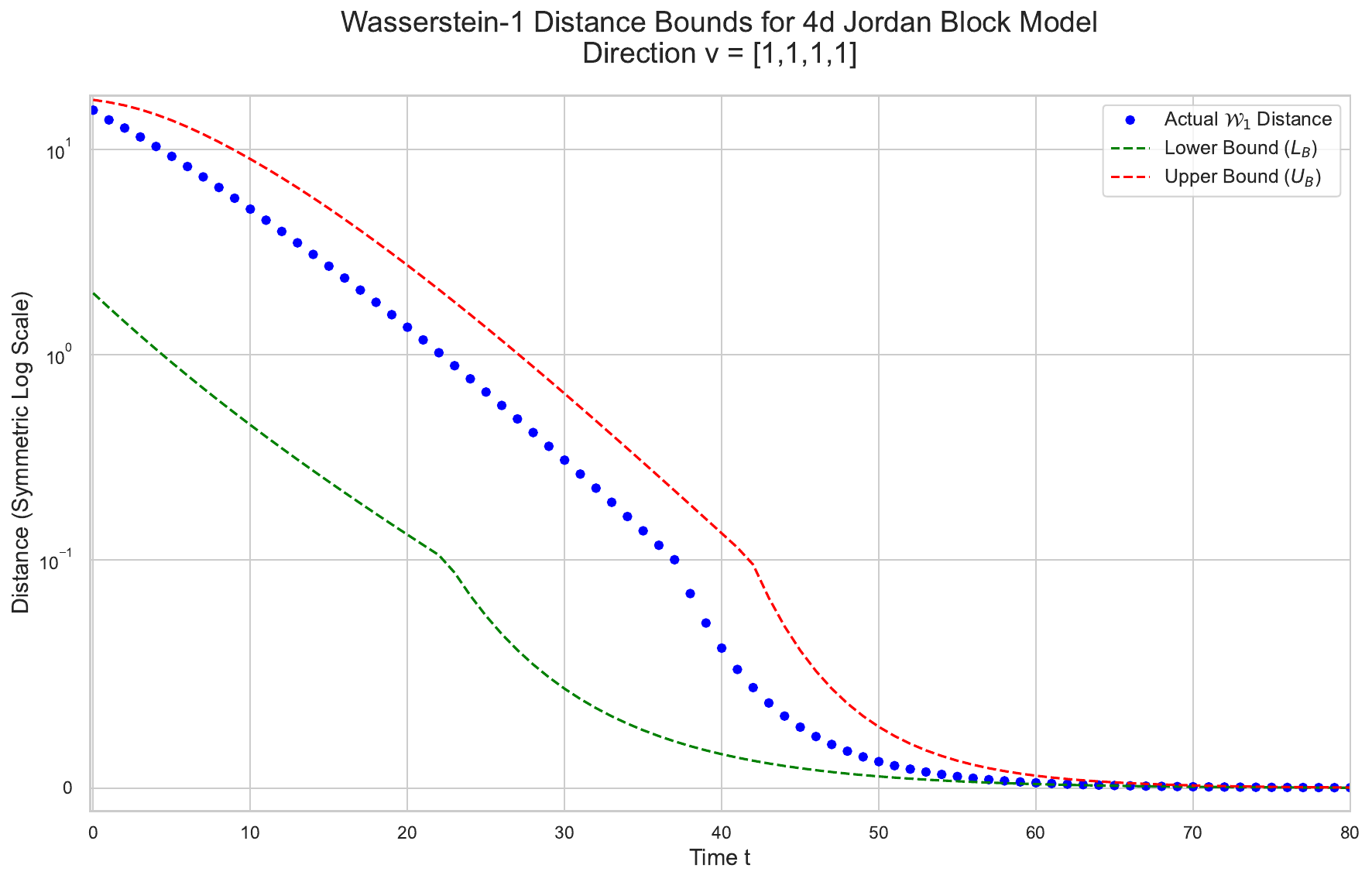}
        \caption{Wasserstein distance $\mathcal{W}_1$.}
        \label{fig:jordan4d_w1}
    \end{subfigure} 
    \vfill 
    \begin{subfigure}[b]{0.49\textwidth}
        \centering
        \includegraphics[width=\textwidth]{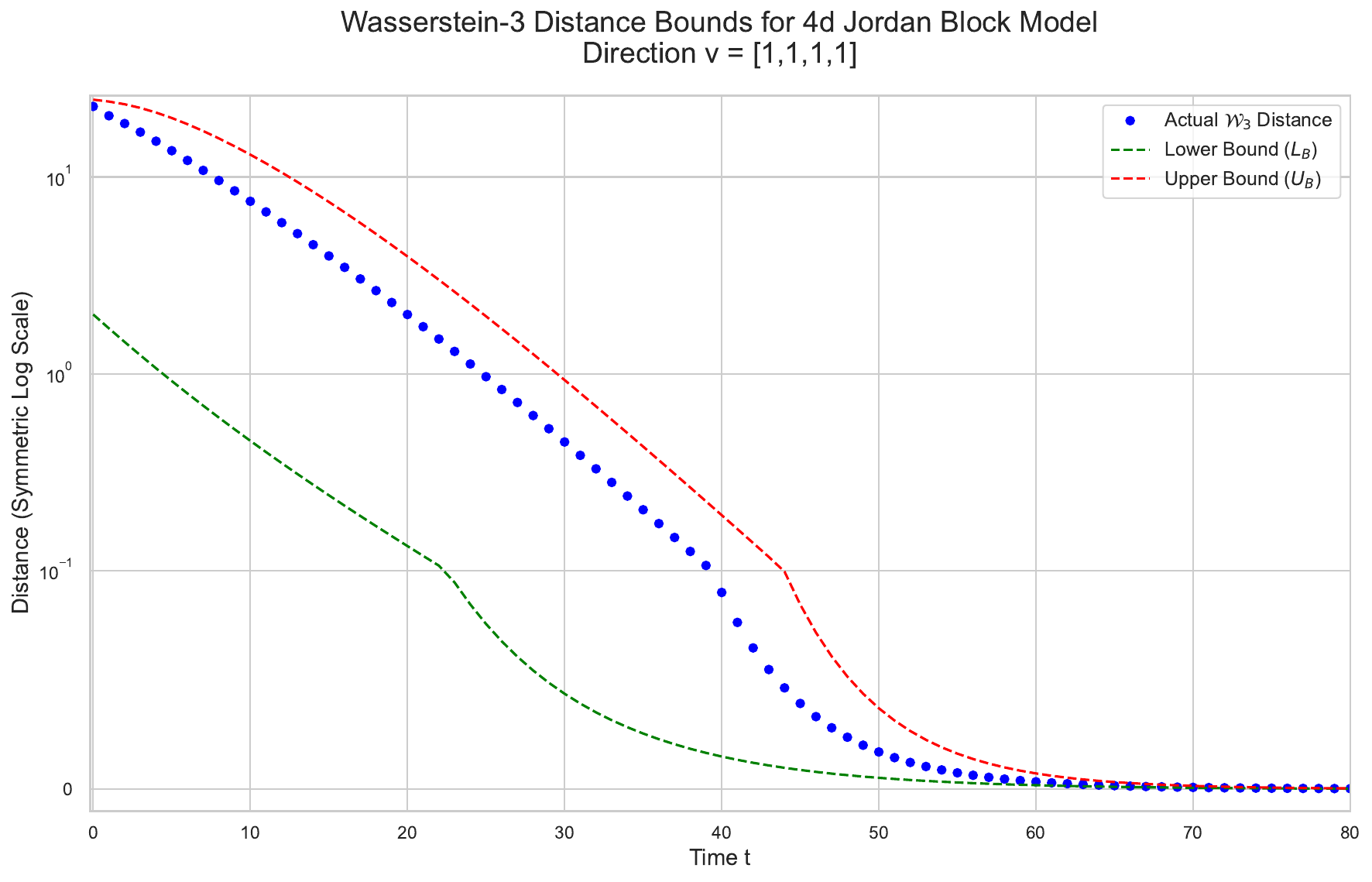}
        \caption{Wasserstein distance $\mathcal{W}_3$.}
        \label{fig:jordan3d_w3} 
    \end{subfigure}
    \hfill 
    \begin{subfigure}[b]{0.49\textwidth}
        \centering
    \includegraphics[width=\textwidth]{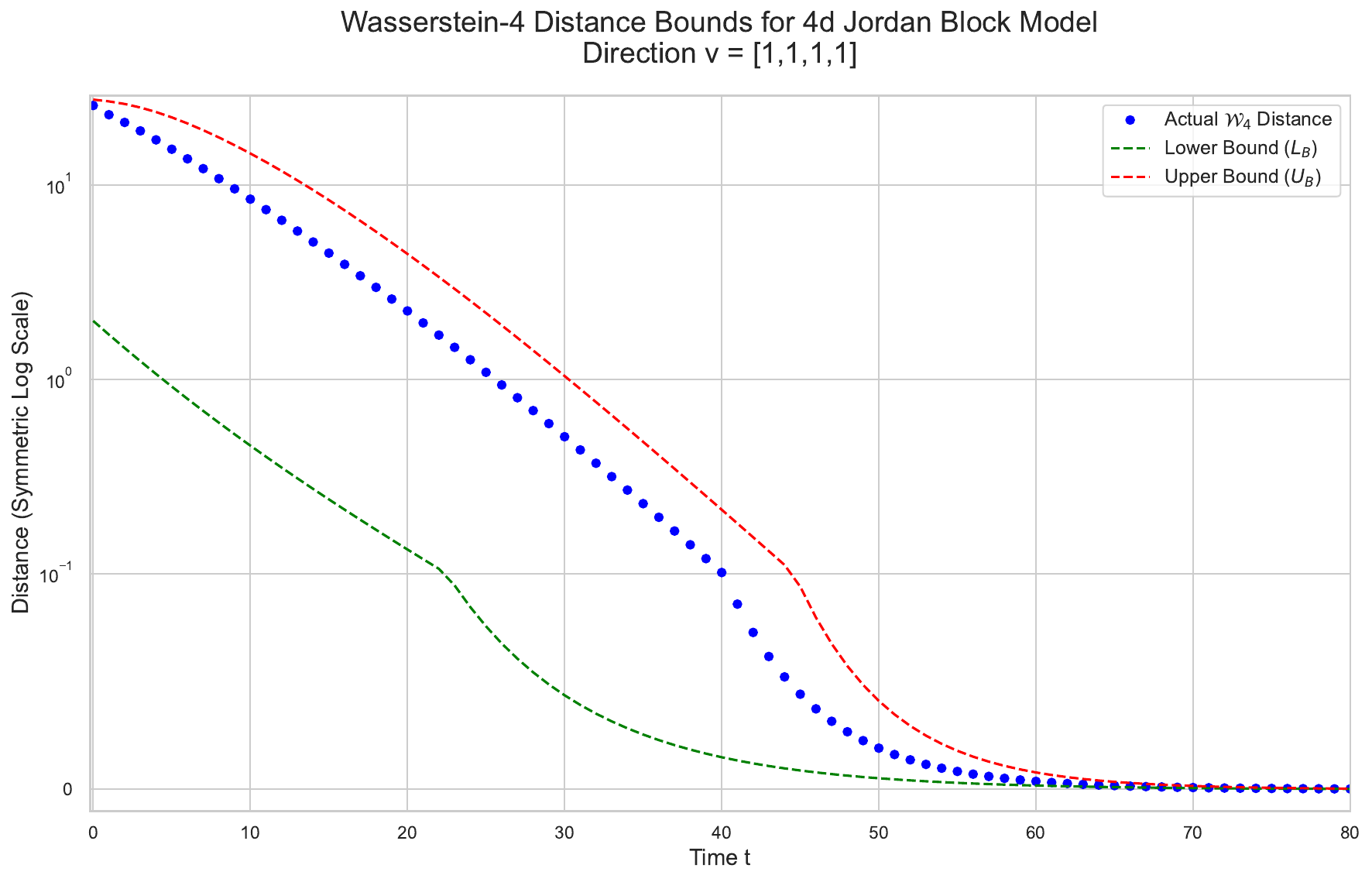}
        \caption{Wasserstein distance $\mathcal{W}_4$.}
        \label{fig:jordan4d_w4} 
    \end{subfigure}
    \caption{Wasserstein distance bounds for the 4d non-diagonalizable system for $r=1, 3, 4$.}
    \label{fig:jordan_4d_all_r}
\end{figure}

\subsubsection{\textbf{Validation for $\mathcal{W}_3$ and $\mathcal{W}_4$ distances for the oscillatory model}}
\label{ssec:oscillatoryw_w3_w4}

\noindent To further test the robustness of the bounds and their dependence on the order $r$, we extend the analysis of the oscillatory system in Subsection \ref{ssec:oscillatory} to the $\mathcal{W}_3$ and $\mathcal{W}_4$ distances. We consider the same system parameters as in the $\mathcal{W}_2$ 
computation (i.e., the same $\mathcal{Q}$, $x$, and $v$) to allow for a direct comparison. The only change is setting $r=3$ and $r=4$ respectively, which modifies the constant prefactor $C_r$ in the upper bound $U_B$ (see Appendix \ref{ap:numerics}) and requires the numerical integration for the actual distance.

The simulation results are presented in Figure~\ref{fig:oscillatory_w3_w4}. The bounds continue to hold for these higher-order distances, correctly enveloping the oscillatory convergence path.

\begin{figure}[H]
    \centering
    \begin{subfigure}[b]{0.49\textwidth}
        \centering
        \includegraphics[width=\textwidth]{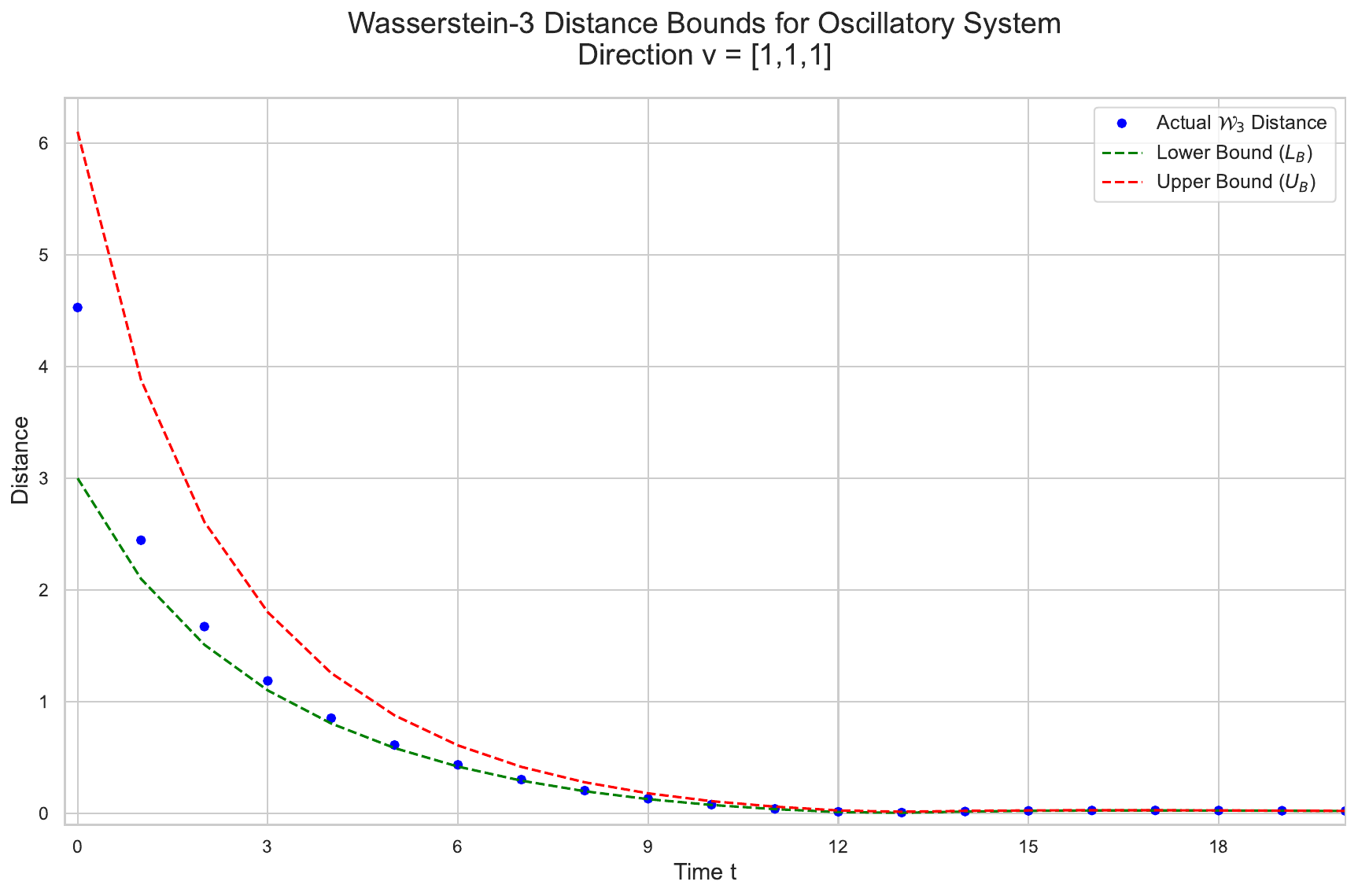}
        \caption{Bounds validation for $\mathcal{W}_3$ distance.}
        \label{fig:oscillatory_w3}
    \end{subfigure}
    \begin{subfigure}[b]{0.49\textwidth}
        \centering
        \includegraphics[width=\textwidth]{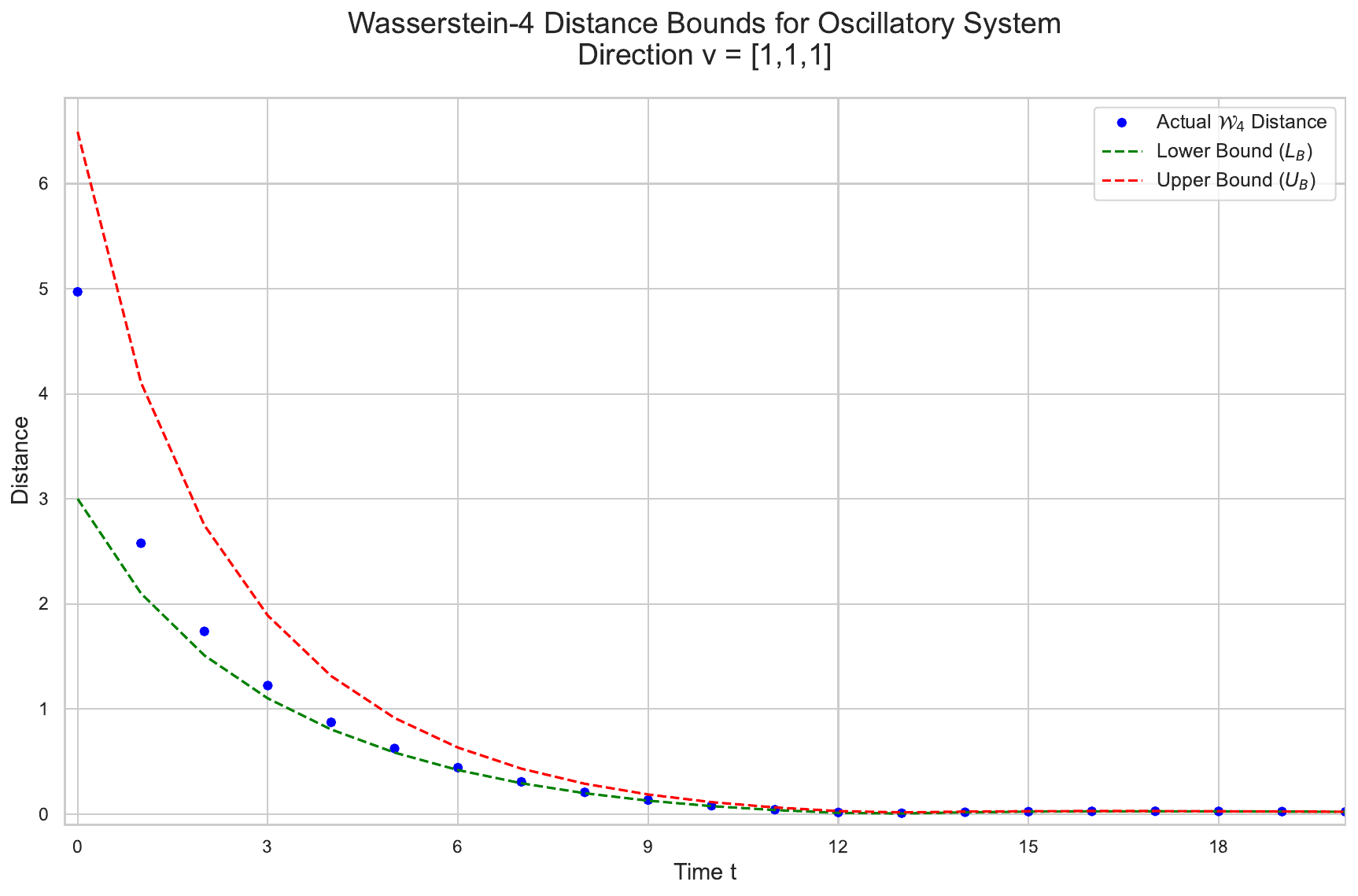}
        \caption{Bounds validation for $\mathcal{W}_4$ distance.}
        \label{fig:oscillatory_w4}
    \end{subfigure}
    \caption{Bounds validation for the oscillatory system using the $\mathcal{W}_3$ and $\mathcal{W}_4$ distances.}
    \label{fig:oscillatory_w3_w4}
\end{figure}

\appendix

\section{\textbf{Ergodicity}}\label{ap:ergodicity}

\begin{lemma}\label{lem:ergodicity}
Under Hypothesis~\ref{hyp:hyperbolic} and Hypothesis~\ref{hyp:moment} for any initial datum $x\in \mathbb{R}^d$ the stochastic process
\begin{equation}
X_t(x)=\cQ^t x+ \sum_{j=0}^{t-1} \cQ^j \Sigma\xi_{t-j},	\quad t\in \mathbb{N}.
\end{equation}
 converges in law as $t$ tends to infinity to
\begin{equation}
X_\infty\stackrel{\mathsf{d}}{=}\sum_{j=0}^{\infty} \cQ^j\Sigma\xi_{j}.
\end{equation}
\end{lemma}

\begin{proof}
In fact, the characteristic function of $X_t(x)$ is given by
\begin{equation}\label{eq:charac}
\begin{split}
\mathbb{E}\left[\exp\left(\ii \left\langle \sum_{j=0}^{t-1} \cQ^j \Sigma\xi_{j}, u\right\rangle\right)\right]&=
\prod_{j=0}^{t-1}
\mathbb{E}\left[\exp\left(\ii \left\langle  \cQ^j\Sigma\xi_j, u\right\rangle\right)\right]\\
&=
\prod_{j=0}^{t-1}
\mathbb{E}\left[\exp\left(\ii \left\langle  \Sigma\xi, 
(\cQ^*)^j u\right\rangle\right)\right]\\
&=
\prod_{j=0}^{t-1}
\psi_{\Sigma\xi} 
((\cQ^*)^j u)=\prod_{j=0}^{t-1}
\exp(\log(\psi_{\Sigma\xi} 
((\cQ^*)^j u)))\\
&=
\exp(\sum_{j=0}^{t-1}\log(\psi_{\Sigma\xi} 
((\cQ^*)^j u))).
\end{split}
\end{equation}

\noindent Let $\ell\in \mathbb{R}$ be fixed. Then there exist complex numbers $\theta_1=\theta_1(\ell)$ satisfying
$|\theta_1|\leq 1$
 such that
\begin{align}\label{eq:taylor}
e^{\ii \ell}&=1+\ii \ell+\frac{\theta_1}{2}\ell^2,
\end{align}
see~\cite[p.310, Ch.7, Proof of Theorem~7.3.1.]{Ashbook}.
Moreover, for $u\in \mathbb{C}$ with $|u|\leq \frac{1}{2}$ there exists a complex number $\theta_2=\theta_2(u)$ with $|\theta_2|\leq 1$ such that
\begin{equation}\label{eq:loga}
\log(1+u)=u+\theta_3 |u|^2,
\end{equation}
where $\log$ denotes the principal branch of the logarithm.
For any $x,z\in \mathbb{R}^d$ we define $\ell=\langle x,z \rangle$. By \eqref{eq:taylor} we have
\begin{align}\label{eq:taylor1}
e^{\ii \langle x,z \rangle}&=
1+\ii \langle x,z \rangle+\frac{\theta_1}{2}\langle x,z \rangle^2.
\end{align}
For any $r>0$ fixed, we set $D_r:=\{y\in \mathbb{R}^d:\,|y|\leq r\}$, and $\mathbbm{1}_{D_r}$ denotes the indicator function.
We then write
\begin{align}\label{eq:expan}
e^{\ii \langle x,z \rangle}
&=
\left(1+\ii \langle x,z \rangle \mathbbm{1}_{D_r}(x)\right)+g(x,z).
\end{align}
By \eqref{eq:taylor1} with the help of the Cauchy--Schwarz inequality we have 
\begin{align*}
|g(x,z)|&\leq \left|
e^{\ii \langle x,z \rangle}-
1-\ii \langle x,z \rangle
\right|\mathbbm{1}_{D_r}(x)+\left|
e^{\ii \langle x,z \rangle}-
1\right|\mathbbm{1}_{D^c_r}(x)\\
&\leq \left(\frac{1}{2}|x|^2|z|^2\right)
\mathbbm{1}_{D_r}(x)+2\cdot\mathbbm{1}_{D^c_r}(x).
\end{align*}
We show that 
\[
\lim\limits_{t\to \infty}\exp(\sum_{j=0}^{t}\log(\psi_{\Sigma\xi} 
((\cQ^*)^j u)))=
\exp(\sum_{j=0}^{\infty}\log(\psi_{\Sigma\xi} 
((\cQ^*)^j u))),
\]
that is, it is enough to check that
\[
\lim\limits_{t\to \infty}\sum_{j=j_0}^{t}\log(\psi_{\Sigma\xi} 
((\cQ^*)^j u))=
\sum_{j=j_0}^{\infty}\log(\psi_{\Sigma\xi} 
((\cQ^*)^j u))
\]
for some fixed $j_0\in \mathbb{N}$ large enough.
We note that 
\begin{align}
\psi_{\Sigma\xi} 
((\cQ^*)^j u)&=\log(\mathbb{E}[e^{\ii\langle \Sigma\xi, (\cQ^*)^j u \rangle }])
=
\log\left(\mathbb{E}\left[1+\ii \langle \Sigma\xi,(\cQ^*)^j u \rangle \mathbbm{1}_{D_{r_j}}(\Sigma\xi)+g(\Sigma\xi,(\cQ^*)^j u)\right]\right)\\
&=
\log\left(1+\ii \mathbb{E}[\langle \Sigma\xi,(\cQ^*)^j u \rangle \mathbbm{1}_{D_{r_j}}(\Sigma\xi)]+g(\Sigma\xi,(\cQ^*)^j u)]\right),
\end{align}
where $r_j>0$ is arbitrary.
Let 
\[
v_j:=\ii \mathbb{E}[\langle \Sigma\xi,(\cQ^*)^j u \rangle \mathbbm{1}_{D_{r_j}}(\Sigma\xi)]+g(\Sigma\xi,(\cQ^*)^j u).
\]
and observe that for $|u|\leq 1$
\begin{align}
|v_j|&\leq C|\lambda|^j|u|\mathbb{E}[|\Sigma\xi| \mathbbm{1}_{D_{r_j}}(\Sigma\xi)]
+\frac{1}{2}
C^2|\lambda|^{2j}|u|^2
\mathbb{E}[|\Sigma\xi|^2\mathbbm{1}_{D_{r_j}}(\Sigma\xi)]
+2\cdot \mathbb{P}(|\Sigma\xi|>r_j)\\
&\leq C|u|r_j|\lambda|^j
+\frac{1}{2}C^2|u|^2 r^2_j|\lambda|^{2j}
+2\cdot \mathbb{P}(|\Sigma\xi|>r_j)<\frac{1}{2}
\end{align}
for any $j\geq  j_0$ and $j_0$ large enough.
By \eqref{eq:loga} it is only needed to show that 
\begin{align*}
\sum_{j\geq j_0} |v_j| <\infty,
\end{align*}
that is,
\[
\sum_{j\geq j_0} r_j|\lambda|^j<\infty \quad \textrm{ and }\quad
\sum_{j\geq j_0} \mathbb{P}(|\Sigma\xi|>r_j)<\infty.
\]
Let $r_j=\rho^{-j}-1>1$ with $0<|\lambda|<\rho<1$ and $j\geq j_0$.
\begin{align}\label{eq:conv1}
\sum_{j\geq j_0} \mathbb{P}(|\Sigma\xi|>r_j)&=
\sum_{j\geq j_0} \mathbb{P}(\log(|\Sigma\xi|+1)>\log(r_j+1))\\
&=\sum_{j\geq j_0} \mathbb{P}(\log(|\Sigma\xi|+1)>j\log(1/\rho))<\infty,
\end{align}
due to $\mathbb{E}[\log(|\Sigma\xi|+1)]<\infty$.
By the Weierstrass M-test we have that 
\[
\sum_{j=1}^{\infty}\log(\psi_{\Sigma\xi} 
((\cQ^*)^j u))
\] 
converges absolutely and uniformly on $|u|\leq 1$. Moreover, the dominated convergence theorem yields that
the map $u\mapsto \sum_{j=1}^{\infty}\log(\psi_{\Sigma\xi} 
((\cQ^*)^j u))$
is continuous at zero.
This combined with \eqref{eq:conv1} yields with the help of the L\'evy continuity theorem the convergence in law of \eqref{eq:reprelimit} to \eqref{eq:limite} as $t\to \infty$.
\end{proof}

\bigskip 
\section{\textbf{Basic properties of the Wasserstein distance}}\label{ap:A}

\begin{lemma}\label{lem:basic}
Let $X$ and $Y$ be random vectors taking values in $\mathbb{R}^d$ with finite $r$-th absolute moments for some $r\geq 1$.
Then 
\begin{enumerate}
\item The Wasserstein distance defines a metric, in the sense of being definite, symmetric and satisfying
the triangle inequality.
\item The Wasserstein distance is translation invariant, in that
\[
\mathcal{W}_r(u+X,Y)=\mathcal{W}_r(X,Y-u)
\]
for all deterministic vectors $u\in \mathbb{R}^d$.
\item The Wasserstein distance is homogeneous, in that
\[
\mathcal{W}_r(cX,cY)=|c|\mathcal{W}_r(X,Y)
\]
for any constant $c\in \mathbb{R}$.
\item The Wasserstein distance satisfies the following shift linearity property
\[
\mathcal{W}_r(u+X,X)=|u|,
\]
for any deterministic vectors $u\in \mathbb{R}^d$.
\item Characterization: Let $(X_n)_{n\in \mathbb{N}}$ be a sequence of random vectors with finite $r$-th absolute moments.
Then the following are equivalent:
\begin{itemize}
\item[(i)] $\mathcal{W}_r(X_n,X)\to 0$ as $n\to \infty$.
\item[(ii)] $X_n\to X$ in distribution as $n\to \infty$, and
 $\mathbb{E}[\|X_n\|^r]\to \mathbb{E}[\|X\|^r] $ as $n\to \infty$.
\item[(iii)] $X_n\to X$ in distribution as $n\to \infty$, and
 $(\mathbb{E}[\|X_n\|^r])_{n\in \mathbb{N}}$
is uniformly integrable.
\end{itemize}
\end{enumerate}
\end{lemma}

\bigskip 
\section{\textbf{Proofs}}\label{a:proofs}

\subsection{\textbf{Proofs of Subsection~\ref{ss:affineergodicinterpolation}}}\label{a:proofs3.1}

\begin{proof}[\textbf{Proof of Lemma~\ref{prop:affine}: }] 
Let $\mathcal{T}$ be any coupling between $X$ and $v+RX$.
Then we have
\begin{equation}
\begin{split}
|v+(R-I_d)\mathbb{E}[X]|&=|\mathbb{E}[X]-\mathbb{E}[v+RX]|=
|\mathbb{E}_{\mathcal{T}}[X]-\mathbb{E}_{\mathcal{T}}[v+RX]|\\
&\leq 
\mathbb{E}_{\mathcal{T}}[|X-(v+RX)|].
\end{split}
\end{equation}\
 Using Jensen's inequality and optimizing over all $\mathcal{T}$ we obtain
\begin{equation}
|v+(R-I_d)\mathbb{E}[X]|\leq \mathcal{W}_1(X,v+RX)\leq 
\mathcal{W}_q(X,v+RX)
\end{equation}
for any $1\leq q\leq r$. 
\end{proof}

\begin{proof}[\textbf{Proof of Theorem~\ref{th:affineGauss}}: ]
Since $\xi_1$ has a Gaussian distribution with vector mean $m$ and covariance matrix $\Xi$, the representation~\eqref{eq:reprelimit} yields that the distribution of $BX_t(x)$ is also Gaussian with vector mean 
\begin{align}\label{eq:meant}
m_t(x):&=B\left(\cQ^tx+\sum_{j=0}^{t-1}\cQ^j \Sigma m\right)=B\cQ^tx+B\left(\sum_{j=0}^{t-1}\cQ^j\right)\Sigma m\\
&=B\cQ^tx+B(I_d-\cQ^{t})(I_d-\cQ)^{-1}\Sigma m,
\end{align}
where in the last equality we have used the Neumann sum,
 and covariance matrix 
\begin{align}\label{eq:neumann0}
\Sigma_t&=B\mathsf{Cov}\left(\sum_{j=0}^{t-1} \cQ^j\Sigma\xi_{j}\right) B^T=\sum_{j=0}^{t-1}B\mathsf{Cov}\left( \cQ^j\Sigma\xi_{j}\right) B^T\\
&=\sum_{j=0}^{t-1}B\mathsf{Cov}\left( \cQ^j\Sigma\xi_{1}\right) B^T
=\sum_{j=0}^{t-1}B \cQ^j\Sigma\Xi(B\cQ^j\Sigma)^T,
\end{align}
where $\Xi$ is the covariance matrix of $\xi_1$.
Hence, the distribution of $BX_\infty$ is Gaussian with
with vector mean 
\begin{align}\label{eq:meaninf}
m_\infty:=B(I_d-\cQ)^{-1}\Sigma m
\end{align}
 and covariance matrix 
\begin{align}\label{eq:neumann1}
\Sigma_\infty&
=B\left(\sum_{j=0}^{\infty} \cQ^j\Sigma \Xi (\cQ^j\Sigma)^T\right) B^T.
\end{align} 
We observe that 
$BX_\infty\stackrel{\mathsf{d}}{=}\Sigma^{1/2}_\infty \mathcal{N}+m_\infty$ and
\begin{align*}
BX_t(x)&\stackrel{\mathsf{d}}{=}\Sigma^{1/2}_t \mathcal{N}+m_t(x)\stackrel{\mathsf{d}}{=}\Sigma^{1/2}_t\Sigma^{-1/2}_\infty (BX_\infty-m_\infty)+m_t(x)\\
&\stackrel{\mathsf{d}}{=}\Sigma^{1/2}_t\Sigma^{-1/2}_\infty BX_\infty+(m_t(x)-\Sigma^{1/2}_t\Sigma^{-1/2}_\infty m_\infty)
\end{align*}
for any $r\geq 1$.
By the upper bound in~\eqref{eq:thineq} of Theorem~\ref{thm:affine} we have
\begin{equation}
\begin{split}
\mathcal{W}_r(BX_t(x),BX_\infty)&\leq
(\mathbb{E}[|(\Sigma^{1/2}_t\Sigma^{-1/2}_\infty-I_d)BX_\infty+(m_t(x)-\Sigma^{1/2}_t\Sigma^{-1/2}_\infty m_\infty)|^{r}])^{1/r}\\
&= (\mathbb{E}[|(\Sigma^{1/2}_t\Sigma^{-1/2}_\infty-I_d)(\Sigma^{1/2}_\infty \mathcal{N}+m_\infty)+(m_t(x)-\Sigma^{1/2}_t\Sigma^{-1/2}_\infty m_\infty)|^{r}])^{1/r}\\
&= (\mathbb{E}[|(\Sigma^{1/2}_t-\Sigma^{1/2}_\infty)\mathcal{N}+(m_t(x)-m_\infty)|^{r}])^{1/r}.
\end{split}
\end{equation}
The preceding inequality with the help of 
Minkowski's inequality and the matrix norm submultiplicativity ($|Mx|\leq \|M\|_F |x|$) yields 
\begin{equation}
\mathcal{W}_r(BX_t(x),BX_\infty)\leq \|\Sigma^{1/2}_t-\Sigma^{1/2}_\infty\|_F\left(\mathbb{E}[|\mathcal{N}|^r]\right)^{1/r}+|m_t(x)-m_\infty|
\end{equation}
for any $r\geq 1$.
By the Hemmen--Ando inequality 
(see Proposition~2.1 in \cite{HemmenAndo}) we have
\begin{equation}
\|\Sigma^{1/2}_t-\Sigma^{1/2}_\infty\|_F\leq \frac{1}{\lambda_{-}}\|\Sigma_t-\Sigma_\infty\|_F,
\end{equation}
where $\lambda_{-}$ is the smallest eigenvalue of $\Sigma_\infty$ 
due to 
\begin{align} \label{e:Endo}
y(\Sigma^{1/2}_t+\Sigma^{1/2}_\infty) y^T
= y\Sigma^{1/2}_ty^T+y\Sigma^{1/2}_\infty y^T
\geq 0 + \lambda_- |y|^2 \quad \textrm{ for all }\quad y\in \mathbb{R}^d.
\end{align}
Then we have 
\begin{equation}\label{eq:ineando}
\mathcal{W}_r(BX_t(x),BX_\infty)\leq 
\frac{1}{\lambda_{-}}\|\Sigma_t-\Sigma_\infty\|_F
\left(\mathbb{E}[|\mathcal{N}|^r]\right)^{1/r}+|m_t(x)-m_\infty|.
\end{equation}
We note that in contrast to $\|\cdot \|_F$, the norms $\|\cdot \|_1$ and $\|\cdot \|_*$ are not invariant for the transpose operator. 
By~\eqref{eq:neumann0},~\eqref{eq:neumann1} and Lemma~\ref{lem:*}\,(3) we have 

\begin{equation}\label{eq:difcov}
\begin{split}
\|\Sigma_t-\Sigma_\infty\|_F&= \left\|\sum_{j=0}^{t-1}B \cQ^j\Sigma \Xi (B\cQ^j\Sigma)^T-\sum_{j=0}^{\infty}B \cQ^j\Sigma \Xi (B\cQ^j\Sigma)^T
\right\|_F
\\
&=\left\|B\left(\sum_{j=t}^{\infty} \cQ^j\Sigma \Xi (\cQ^j\Sigma)^T\right) B^T\right\|_F\\
&\leq \|B\|_F^2 \sum_{j=t}^\infty  \|\cQ^j \Sigma \Xi \Sigma^T(Q^T)^j\|_F \\
&\leq \|B\|_F^2 \|\Sigma\|_F^2 \|\Xi\|_F \sum_{j=t}^\infty  \|\cQ^j\|_F \|(\cQ^T)^j\|_F\\
&\leq \|B\|_F^2 \|\Sigma\|_F^2 \|\Xi\|_F \sum_{j=t}^\infty  \|\cQ^j\|_F^2\\
&\leq C_*^2 \|B\|_F^2 \|\Sigma\|_F^2 \|\Xi\|_F \sum_{j=t}^\infty  \|\cQ^j\|_*^2\\
&\leq C_*^2 \|B\|_F^2 \|\Sigma\|_F^2 \|\Xi\|_F \sum_{j=t}^\infty  \|\cQ\|_*^{2j}\\
&\leq C_*^2 \|B\|_F^2 \|\Sigma\|_F^2 \|\Xi\|_F \frac{\|\cQ\|_*^{2t}}{1-\|\cQ\|_*^{2}}.
\end{split}
\end{equation}

Finally, by~\eqref{eq:meant} and~\eqref{eq:meaninf} we have 
\begin{equation}\label{eq:difmean}
\begin{split}
|m_t(x)-m_\infty|&=|B\cQ^tx+B(I_d-\cQ^{t})(I_d-\cQ)^{-1}\Sigma m-B(I_d-\cQ)^{-1}\Sigma m|\\
&=|B\cQ^t(x-(I_d-\cQ)^{-1}\Sigma m)|.
\end{split}
\end{equation}
Combining~\eqref{eq:difcov} and~\eqref{eq:difmean} in~\eqref{eq:ineando} we obtain the upper bound 
in~\eqref{eq:thgauss}.

The lower bound in~\eqref{eq:thgauss} follows from the lower bound in~\eqref{eq:thineq} of Theorem~\ref{thm:affine} with the help of~\eqref{eq:difmean}.
\end{proof}

\begin{proof}[\textbf{Proof of Corollary~\ref{cor:projected}:}]
By~\eqref{eq:meant},~\eqref{eq:neumann0} we have 
$\langle v, X_t(x) \rangle$ has Gaussian distribution with mean and variance
\begin{align}
\langle v,m_t(x)\rangle=\langle v,
\cQ^tx+(I_d-\cQ^{t})(I_d-\cQ)^{-1}\Sigma m
\rangle
\end{align}
and
\begin{align}
v^T\Sigma_t v
=\sum_{j=0}^{t-1}v^T \cQ^j\Sigma\Xi(\cQ^j\Sigma)^Tv,
\end{align}
respectively. Similarly, \eqref{eq:meaninf} and~\eqref{eq:neumann1} yield that
$\langle v, X_\infty \rangle$ has Gaussian distribution with mean and variance
\begin{align}
\langle v,m_\infty\rangle=\langle v,
(I_d-\cQ)^{-1}\Sigma m
\rangle
\end{align}
and
\begin{align}
v^T\Sigma_\infty v
=\sum_{j=0}^{\infty}v^T \cQ^j\Sigma\Xi(\cQ^j\Sigma)^Tv
\end{align}
By Remark~\ref{rem:commonOT} we have 
\begin{equation}
\begin{split}
\mathcal{W}_r(\langle v, X_t(x) \rangle,\langle v, X_\infty \rangle)&\leq |\langle v,\cQ^t (x-(I_d-\cQ)^{-1}\Sigma m)\rangle|\\
&\quad+|\sqrt{v^T\Sigma_t v}-\sqrt{v^T\Sigma_\infty v}|(\mathbb{E}[|\mathcal{N}|^r])^{1/r},
\end{split}
\end{equation}
where $\mathcal{N}$ has standard Gaussian distribution on $\mathbb{R}$.
We recall that
\[
(\mathbb{E}[|\mathcal{N}|^r])^{1/r}=\frac{\sqrt{2}(\Gamma((r+1)/2))^{1/r}}{\sqrt{\pi}^{1/r}},
\]
where $\Gamma$ is the usual Gamma function.
Since
\begin{align}
|\sqrt{v^T\Sigma_t v}-\sqrt{v^T\Sigma_\infty v}|&=
\frac{|v^T\Sigma_t v-v^T\Sigma_\infty v|}{|\sqrt{v^T\Sigma_t v}+\sqrt{v^T\Sigma_\infty v}|}\\
&
=\frac{1}{|\sqrt{v^T\Sigma_t v}+\sqrt{v^T\Sigma_\infty v}|}
|\sum_{j=t}^{\infty}v^T \cQ^j\Sigma\Xi(\cQ^j\Sigma)^Tv|\\
&\leq 
\frac{C_*^2}{\sqrt{v^T\Sigma_t v+v^T\Sigma_\infty v}}|v|^2|\|\Sigma\|_F^2|\|\Xi\|_F
\frac{\|\cQ\|_*^{2t}}{1-\|\cQ\|_*^{2}},
\end{align}
one can deduce the right-hand side of~\eqref{eq:cor}.
\end{proof}

\begin{proof}[\textbf{Proof of Corollary~\ref{cor:PWdII}:}]
By Remark~\ref{rem:commonOT} we have 
\begin{equation}
\begin{split}
\mathcal{W}_r(\langle v, X_t(x) \rangle,\langle v, X_\infty \rangle)&\leq |\langle v,\cQ^t (x-(I_d-\cQ)^{-1}\Sigma m)\rangle|+|\sqrt{v^T\Sigma_t v}-\sqrt{v^T\Sigma_\infty v}|(\mathbb{E}[|\mathcal{N}|^r])^{1/r}.
\end{split}
\end{equation}
We recall that
\[
(\mathbb{E}[|\mathcal{N}|^r])^{1/r}=\frac{\sqrt{2}(\Gamma((r+1)/2))^{1/r}}{\sqrt{\pi}^{1/r}}
\]
and
\begin{align}
|\sqrt{v^T\Sigma_t v}-\sqrt{v^T\Sigma_\infty v}|&=
\frac{|v^T\Sigma_t v-v^T\Sigma_\infty v|}{|\sqrt{v^T\Sigma_t v}+\sqrt{v^T\Sigma_\infty v}|}\\
&
=\frac{1}{|\sqrt{v^T\Sigma_t v}+\sqrt{v^T\Sigma_\infty v}|}
\left|\sum_{j=t}^{\infty}v^T \cQ^j\Sigma\Xi(\cQ^j\Sigma)^Tv\right|.
\end{align}
Since $\Sigma\Xi \Sigma^T \cQ=\cQ \Sigma\Xi \Sigma^T$, we have
$\Sigma\Xi \Sigma^T \cQ^j=\cQ^j \Sigma\Xi \Sigma^T$ for all $j\in \mathbb{N}$. Then we have
\begin{equation}
\begin{split}
 \sum_{j=t}^{\infty}v^T \cQ^j(\Sigma\Xi \Sigma^T)(\cQ^j)^Tv&=
 v^T (\Sigma\Xi \Sigma^T)\sum_{j=t}^{\infty} \cQ^j(\cQ^j)^Tv\\
 &=v^T (\Sigma\Xi \Sigma^T)
\sum_{j=t}^{\infty} \cQ^{t} \cQ^{-t}\cQ^j(\cQ^j)^T((\cQ)^T)^{-t}((\cQ)^T)^{t}v\\
&=
v^T (\Sigma\Xi \Sigma^T) \cQ^{t} S((\cQ)^T)^{t}v =
v^T  \cQ^{t} (\Sigma\Xi \Sigma^T)S((\cQ)^T)^{t}v,
\end{split}
\end{equation}
where 
\[
S=\sum_{j=0}^{\infty}\cQ^j(\cQ^j)^T=(I-\cQ \cQ^{T})^{-1}.
\]
Note that $\Sigma\Xi \Sigma^T$ and $S$ are  non-negative definite symmetric square matrices. Moreover, $\Sigma\Xi \Sigma^T \cQ^T=\cQ^T \Sigma\Xi \Sigma^T$, which implies $(\Sigma\Xi \Sigma^T) S=S(\Sigma\Xi \Sigma^T)$ and the symmetry of $S(\Sigma\Xi \Sigma^T).$
Hence, the non-negative definite square root matrix of $S(\Sigma\Xi \Sigma^T)$ exists and satisfies $\Sigma\Xi \Sigma^T S=(\Sigma\Xi \Sigma^T S)^{1/2}(\Sigma\Xi \Sigma^T S)^{1/2}$.
Therefore,
\begin{equation}
\begin{split}
|v^T  \cQ^{t} (\Sigma\Xi \Sigma^T)S((\cQ)^T)^{t}v|&=|(\Sigma\Xi \Sigma^T S)^{1/2}\cQ^t v|^2,
\end{split}
\end{equation}
which implies~\eqref{eq:PWDII}.
\end{proof}

\begin{proof}[\textbf{Proof of 
Corollary \ref{cor:aeiistable}:}]
Recall that $X_t(x)\stackrel{\mathsf{d}}{=}\lambda^t +\sum_{j=0}^{t-1} \cQ^j \Sigma\xi_{j}$, $t\in \mathbb{N}$. Then the characteristic function of $X_t(x)-\lambda^t x$ is given by
\begin{equation}\label{eq:characdos}
\begin{split}
&\mathbb{E}\left[\exp\left(\ii \left\langle \sum_{j=0}^{t-1} \cQ^j \Sigma\xi_{j}, u\right\rangle\right)\right]
=
\mathbb{E}\left[\exp\left(\ii \left\langle \sum_{j=0}^{t-1} |\lambda|^j \xi_{j}, u\right\rangle\right)\right]
=
\prod_{j=0}^{t-1}
\mathbb{E}\left[\exp\left(\ii \left\langle  |\lambda|^j \xi_{j}, u\right\rangle\right)\right]\\
&\quad =
\prod_{j=0}^{t-1}
\mathbb{E}\left[\exp\left(\ii \left\langle   \xi_{0}, |\lambda|^j u\right\rangle\right)\right]
=
\prod_{j=0}^{t-1}\exp(-c_0(|\lambda|^j|u|)^{\alpha})
=
\exp\left(-c_0|u|^{\alpha}\left(\frac{1-|\lambda|^{\alpha t}}{1-|\lambda|^{\alpha }}\right)\right),
\end{split}
\end{equation}
which implies \eqref{eq:scala}. 
The inequalities \eqref{eq:thineqdos2} follows directly by Theorem~\ref{thm:affine} with the help of Theorem~1.13 in \cite{KyprianouPardo}.
\end{proof}

\subsection{\textbf{Proofs of Subsection~\ref{ss:genericexponentialbounds}}}\label{a:proofs3.2}

\begin{proof}[Proof of Theorem~\ref{thm:genericARMA}]
We start with the second upper bound of \eqref{e:genericARMA}. 
For $p\geq 1$ note that the usual bounds of the Wasserstein distance, 
\cite[Lemma~2]{Mariucci18}, the Minkowski inequality combined with \eqref{e:Kd} yield 
\begin{equation}
\begin{split}
\cW(X_t(x),X_\infty)
|&\leq |\cQ^t x|+\cW(\sum_{j=1}^{t} \cQ^j\Sigma\xi_{j},\sum_{j=1}^{\infty} \cQ^j\Sigma\xi_{j})\\
&=  |\cQ^t x|+\cW(\sum_{j=1}^{t} \cQ^j\Sigma\xi_{j},\sum_{j=1}^{t} \cQ^j\Sigma\xi_{j}+\sum_{j=t+1}^{\infty} \cQ^j\Sigma\xi_{j})\\
&\leq  
|\cQ^t x|+
\left(
\mathbb{E}\left[|
\sum_{j=t+1}^{\infty} \cQ^j\Sigma\xi_{j}|^p \right]\right)^{1/p}\\
&\leq  |\cQ^t x|+
\left(\sum_{j=t+1}^{\infty}
\mathbb{E}\left[| \cQ^j \Sigma\xi_{j}|^p \right]\right)^{1/p}\\
&\leq  |\cQ^t x|+
\left(\sum_{j=t+1}^{\infty}
\mathbb{E}\left[K^p_{d}\| \cQ^j\|^p_{*}|\Sigma\xi_{j}|^p \right]\right)^{1/p}\\
&=  |\cQ^t x|+
K_{d}\left(\mathbb{E}\left[|\Sigma\xi_1|^p \right]\sum_{j=t+1}^{\infty}\| \cQ^j\|^p_{*}
\right)^{1/p}\\
&=  |\cQ^t x|+
K_{d}\left(\mathbb{E}\left[|\Sigma\xi_1|^p_1 \right]\right)^{1/p}\left(\sum_{j=t+1}^{\infty}
(\| \cQ\|^{p})^j_{*}
\right)^{1/p}.
\end{split}
\end{equation}
We note that
\begin{equation}
\left(\sum_{j=t+1}^{\infty}
(\| \cQ\|^{p}_{*})^j
\right)^{1/p}=\frac{\|\cQ\|^{t+1}_{*}}{(1-\|\cQ\|^p_{*})^{1/p}},
\end{equation}
which implies
\[
\cW(X_t(x), X_\infty)\leq 
 |\cQ^t x|+K_d
\left(\mathbb{E}\left[|\Sigma\xi|^p_1 \right]\right)^{1/p}\frac{\|\cQ\|^{t+1}_{*}}{(1-\|\cQ\|^p_{*})^{1/p}}.
\]
We continue with the first upper bound of \eqref{e:genericARMA}. 
Due to the Markovianity and ergodicity of the \eqref{eq:model} we have 
\begin{align*}
 \cW(X_t(x),X_\infty)&\leq \int_{\mathbb{R}^d}\cW(X_t(x),X_t(y)) \mathbb{P}(X_\infty\in \ud y)\\
 &\leq \int_{\mathbb{R}^d}|\cQ^t(x-y)|\mathbb{P}(X_\infty\in \ud y) =\mathbb{E}[|\cQ^t(x-X_\infty)|].
\end{align*}
Finally we apply Lemma~\ref{lem:*}\,(2). 

\noindent Now, we show the lower bound. 
Recall that
\begin{equation}\label{e:Neumann}
\sum_{j=t}^{\infty} \cQ^j=\cQ^t(I-\cQ)^{-1}=
 (I-\cQ)^{-1}\cQ^t
 \end{equation}
 and that $(\xi_{j})_{j}$ are identically distributed.
Using \eqref{e:meandifference} and \eqref{e:Neumann} we obtain 
\begin{align*}
\cW(X_t(x),X_\infty) 
&\geq 
 \left|\cQ^t x+\mathbb{E}[\sum_{j=1}^{t-1} \cQ^j\Sigma \xi_{j}]-\mathbb{E}[\sum_{j=1}^{\infty} \cQ^j\Sigma \xi_{j}]\right|
= 
|\cQ^t(x-(I-\cQ)^{-1}\mathbb{E}[\Sigma \xi_1])| .
\end{align*}
\end{proof}

\begin{proof}[Proof of Corollary \ref{cor:schursliced}]
Similarly to the proof of Theorem~\ref{thm:genericARMA} we have the following. 
For $p\geq 1$ 
\begin{equation}
\begin{split}
\cW(\langle X_t(x), v\rangle ,\langle X_\infty, v\rangle )
&\leq  |\langle \cQ^t x, v\rangle|+
K_{d}|v|\left(\mathbb{E}\left[|\Sigma\xi_1|^p_1 \right]\right)^{1/p}\left(\sum_{j=t+1}^{\infty}
(\| \cQ\|^{p})^j_{*}
\right)^{1/p},
\end{split}
\end{equation}
which implies
\[
\cW(\langle X_t(x), v\rangle , \langle X_\infty, v\rangle)\leq 
 |\langle \cQ^t x, v\rangle |+K_d |v|
\left(\mathbb{E}\left[|\Sigma\xi|^p_1 \right]\right)^{1/p}\frac{\|\cQ\|^{t+1}_{*}}{(1-\|\cQ\|^p_{*})^{1/p}}.
\]
Hence by Remark~\ref{rem:sliced} it follows 
\begin{align*}
S\mathcal{W}_p( X_t(x), X_\infty) 
&\leq \frac{\Gamma(\frac{d}{2})}{\Gamma\left(\frac{d+1}{2}\right)\pi^{\frac{1}{2}}} |Q^t x| + K_d
\left(\mathbb{E}\left[|\Sigma\xi|^p_1 \right]\right)^{1/p}\frac{\|\cQ\|^{t+1}_{*}}{(1-\|\cQ\|^p_{*})^{1/p}}.
\end{align*}
Moreover, 
\begin{align*}
 S\cW(X_t(x), X_\infty)
& = \int_{\{|v|=1\}} \cW(\langle X_t(x), v\rangle ,\langle X_\infty, v\rangle ) \ud\frac{\mathcal{H}_d(v)}{A_d}\\
  &\leq \int_{\{|v|=1\}}\mathbb{E}[|\langle \cQ^t(x-X_\infty), v \rangle|]\ud\frac{\mathcal{H}_d(v)}{A_d}\\
& = \mathbb{E}\left[|\cQ^t(x-X_\infty)| \int_{\{|v|=1\}}| |\langle \frac{\cQ^t(x-X_\infty)}{|\cQ^t(x-X_\infty|}, v \rangle|\ud\frac{\mathcal{H}_d(v)}{A_d}\right]\\  
 &\leq \frac{\Gamma(\frac{d}{2})}{\Gamma\left(\frac{d+1}{2}\right)\pi^{\frac{1}{2}}} \mathbb{E}[|\cQ^t(x-X_\infty)|].
\end{align*}

\noindent Now, we show the lower bound. 
Using \eqref{e:meandifference} and Remark~\ref{rem:sliced} we obtain 
\begin{align*}
S\cW(X_t(x),X_\infty) 
&\geq 
 \frac{\Gamma(\frac{d}{2})}{\Gamma\left(\frac{d+1}{2}\right)\pi^{\frac{1}{2}}} |\cQ^t(x-(I-\cQ)^{-1}\mathbb{E}[\Sigma \xi_1])|.
\end{align*}
\end{proof}

\begin{proof}[Proof of Lemma~\ref{lem:Lyapunov}]
For each $z\in \mathbb{R}^d\setminus\{0\}$ and $t\in \mathbb{N}$, let $\widetilde{z}(0):=U^{-1}z\in \mathbb{C}^d$ and 
$\widetilde{z}(t)\in \mathbb{C}^d$
given by
$\widetilde{z}_j(t):=q^t_j\widetilde{z}_j(0)$, $j=1,\ldots,d$.
For any $t\in \mathbb{N}$ we obtain
\[
\cQ^tz=UD^tU^{-1}z=UD^t\widetilde{z}(0)=U\widetilde{z}(t)
\]
and
\[
|\cQ^t z|^2=
\sum_{j=1}^{d}\left|\sum_{k=1}^{d}
U_{j,k}\widetilde{z}_k(t)\right|^2.
\]
In the sequel, we obtain the right-hand side of~\eqref{eq:cotaarriba}. By the Cauchy--Schwarz inequality we have 
\begin{equation}
\begin{split}
|\cQ^t z|^2
&=
\sum_{j=1}^{d}\left|\sum_{k=1}^{d}
U_{j,k}q^t_k \widetilde{z}_k(0)\right|^2
\leq 
\sum_{j=1}^{d}\left(\sum_{k=1}^{d} |U_{j,k}|^2\right)
\left(\sum_{\ell=1}^{d} |q_\ell|^{2t}|\widetilde{z}_{\ell}(0)|^2\right)\\
&= \|U\|^2_{\mathrm{F}}
\left(\sum_{\ell=1}^{d} |q_\ell|^{2t}|(U^{-1}z)_{\ell}|^2\right).
\end{split}
\end{equation}
We continue with  the left-hand side of~\eqref{eq:cotaarriba}. Since $U$ is invertible we may 
write for any $z\in \mathbb{R}^d$ that 
$|D^t U^{-1}z| = |U^{-1} U D^t U^{-1}z|\leq \|U^{-1}\|_{F} |U D^t U^{-1}z|$. 
such that $|U D^t U^{-1}z|\geq \|U^{-1}\|_{\mathrm{F}}^{-1} |D^t U^{-1}z|$. Consequently, 
\begin{align*}
|\cQ^t z|^2
&=
\left|U D^t \widetilde{z}_k(0)\right|^2
\geq  \|U^{-1}\|_{F}^{-2} \left|D^t \widetilde{z}_k(0)\right|^2
= \|U^{-1}\|_{F}^{-2}\sum_{j=1}^d |q_i|^{2t} |(U^{-1}z)_j|^2.   
\end{align*}
\end{proof}

\begin{proof}[Proof of Lemma \ref{thm:parallel}] The proof of \eqref{e:Panaretos} is given in \cite[p.8]{PanaretosZemel}.
For $p\geq 1$ we cite the tensorization result \cite[Lemma~3]{Mariucci18} for the upper bound, 
while the lower bound follows from \eqref{e:meandifference}: 
\begin{align}
\cW((X^{(1)}_t(x), \cdots , X^{(n)}_t(x)) ,(X_\infty^{(1)}, \cdots ,  X_\infty^{(n)}))&\geq |\mathbb{E}[(X^{(1)}_t(x), \cdots , X^{(n)}_t(x))-(X^1_\infty, \cdots , X^n_\infty)]|\\
& \geq |(\mathbb{E}[X^{(1)}_t(x)-X^{1}_\infty], \cdots ,\mathbb{E}[X^{(1)}_t(x)-X^{1}_\infty])|\\
&\geq 
\sqrt{n}|\mathbb{E}[X^{(1)}_t(x)-X^{1}_\infty]|.
\end{align}
\end{proof}

\begin{proof}[Proof of Corollary \ref{cor:empirical}]
Equation~\ref{e:empirical} is satisfied by direct calculation due to the linearity. 
Hence, Theorem~\ref{thm:genericARMA} implies the statement. It only remains to estimate 
$(\mathbb{E}[|\zeta^{(n)}_1|^p])^{1/p}$. Note that $\mathbb{E}[\zeta^{(n)}_t]=\mathbb{E}[ \xi^{(1)}_1]$ and using  the Minkowski inequality we have
\[
(\mathbb{E}[|\zeta^{(n)}_1|^p])^{1/p}\leq 
(\mathbb{E}[|\xi_1^{(1)}|^p])^{1/p}\quad \textrm{ for all }\quad n\in \mathbb{N}.
\]
\end{proof}

\section{\textbf{The numerical experiments}}\label{ap:numerics}
    
\noindent The simulations of Subsection~\ref{ss:numerical} are based on the model \eqref{e:genericARMA} under Hypothesis~\ref{hyp:hyperbolic} and Hypothesis~\ref{hyp:moment}. The source code to reproduce the results is available on \url{https://github.com/phcsta/sliced-wasserstein-bounds}.

The simulations use the bounds derived for Gaussian noise given in Corollary~\ref{cor:PWdII}. 
More precisely, 
\begin{equation}
\begin{split}
\mathcal{W}_r(\langle v, X_t(x) \rangle,\langle v, X_\infty \rangle)&\leq |\langle v,\mathcal{Q}^t (x-(I_d-\mathcal{Q})^{-1}\Sigma m)\rangle|\\
&\quad+\frac{1}{\sqrt{v^T\Sigma_\infty v}}
\frac{\sqrt{2}(\Gamma((r+1)/2))^{1/r}}{\pi^{1/(2r)}} \|S^{1/2}\mathcal{Q}^t v\|^2
\end{split}
\end{equation}
with $m=0$, $B=I_3$ and $\Sigma = I_3$. We define the lower bound $L_B$ and the upper bound $U_B$, respectively, by
\begin{align*}
L_B := |\langle v,\mathcal{Q}^t x\rangle|\quad \text{and} \quad
U_B := L_B + \frac{1}{\sqrt{v^T\Sigma_\infty v}}
\frac{\sqrt{2}(\Gamma((r+1)/2))^{1/r}}{\pi^{1/(2r)}} \|S^{1/2}\mathcal{Q}^t v\|^2.
\end{align*}
We start with the observation that 
\begin{enumerate}
    \item \textbf{The marginal law:} $\langle v, X_t \rangle \stackrel{\mathsf{d}}{=} \mathcal{N}(\mu_{t}, \sigma_{t}^2)$,
    \item \textbf{The stationary law:} $\langle v, X_\infty \rangle \stackrel{\mathsf{d}}{=} \mathcal{N}(\mu_{\infty}, \sigma_{\infty}^2)$,
\end{enumerate}
where $\mu_{t},\, \mu_{\infty},\, \sigma_{t}^2,\, \sigma_{\infty}^2$ are given below.

\noindent\textbf{Marginal distribution $\langle v, X_t \rangle$:}
Recall that the mean and the covariance of $X_t$, starting from $X_0(x) = x$ and with noise covariance $\Sigma$, are given by
\[
  \mathbb{E}[X_t] = \mathcal{Q}^t x \quad \textrm{ and } \quad \Sigma_t:= \mathsf{Cov}(X_t) = \sum_{k=0}^{t-1} \mathcal{Q}^k \Sigma (\mathcal{Q}^k)^T.
\]
Projecting these onto $v$ gives the 1-d parameters
\begin{align*}
   \mu_{t}= v^T \mathcal{Q}^t x\quad \textrm{ and } \quad 
 \sigma_{t}^2 = v^T \Sigma_t v.
\end{align*}

\noindent\textbf{Limiting distribution $\langle v, X_\infty \rangle$:}
Assuming mean-zero noise, the mean and covariance of the stationary distribution $X_\infty$ are
\[
\mathbb{E}[X_\infty] = 0 \quad \textrm { and } \quad \mathsf{Cov}(X_\infty) = \Sigma_\infty,
\]
where $\Sigma_\infty$ is the limiting covariance matrix $\Sigma_\infty$. 
Recall that $\Sigma=I_3$.
It is not hard to see that $\Sigma_\infty$ is the
unique positive definite solution to the discrete-time Lyapunov equation
\[
\Sigma_\infty = \mathcal{Q}\Sigma_\infty\mathcal{Q}^T + \Sigma.
\]
See for instance \cite{Bartels1972, Simoncini2016}. Projecting $X_\infty$ onto $v\neq 0$ yields the scalar parameters
\begin{align*}
    \mu_{\infty} = \mathbb{E}[\langle v, X_\infty \rangle] = 0\quad \textrm{and} \quad
    \sigma_{\infty}^2 = \mathsf{Var}(\langle v, X_\infty \rangle) = v^T \Sigma_\infty v.
\end{align*}
           
\noindent\textbf{Closed-form solution for $\mathcal{W}_2$.}
For any two one-dimensional Gaussian distributions, 
$P_1 = \mathcal{N}(\mu_1, \sigma_1^2)$ and $P_2 = \mathcal{N}(\mu_2, \sigma_2^2)$, recall that the Wasserstein-2 distance 
is given by
\begin{equation}
\label{eq:w2_1d_general}
    \mathcal{W}_2(P_1, P_2) = \sqrt{(\mu_1 - \mu_2)^2 + (\sigma_1 - \sigma_2)^2},
\end{equation}
which yields
\begin{align*}
    \mathcal{W}_2(\big\langle v, X_t \rangle, \langle v, X_\infty \big\rangle) 
     = \sqrt{\mu_{t}^2 + (\sigma_{t} - \sigma_{\infty})^2}.
\end{align*}
In Subsection~\ref{ss:numerical} we present the results of the implementation 
by the following the pseudo-code \ref{alg:wasserstein}. The Python implementation provided in\\
\url{https://github.com/phcsta/sliced-wasserstein-bounds} uses the Pythagorean formula \eqref{eq:w2_1d_general}.

\begin{algorithm}[H]
\caption{Computation of Sliced Wasserstein Bounds Along a Fixed Direction}
\label{alg:wasserstein}
\begin{algorithmic}[1]

\Require System matrices $(\mathbf{Q}, \mathbf{\Sigma})$, direction $\mathbf{v}$, initial state $\mathbf{x}_0$, time horizon $T_{\max}$.
\Require Evaluation mode $\mathcal{M} \in \{\texttt{ClosedForm}, \texttt{NumInt}, \texttt{MonteCarlo}\}$.
\If{$\mathcal{M} = \texttt{MonteCarlo}$}
  \Require Number of samples $N$.
\EndIf

\Ensure Sequences $L_B^{(t)}$, $U_B^{(t)}$, and $\mathcal{W}_r^{(t)}$ over $t=0,\ldots,T_{\max}$.

\Statex \textbf{1. Stationary Quantities}
\State Solve Lyapunov eq: $\mathbf{\Sigma}_\infty = \mathbf{Q}\mathbf{\Sigma}_\infty\mathbf{Q}^\top + \mathbf{\Sigma}$.
\State $\sigma_\infty \gets \sqrt{\mathbf{v}^\top \mathbf{\Sigma}_\infty \mathbf{v}}$.
\State $C_r \gets \sqrt{2}\,\Gamma\!\left(\frac{r+1}{2}\right)^{1/r}\,\pi^{-1/(2r)}$.

\Statex \textbf{2. Initialization}
\If{$\mathcal{M} \neq \texttt{MonteCarlo}$}
  \State $\mathbf{A}_0 \gets \mathbf{I}$, \qquad $\mathbf{P}_0 \gets \mathbf{0}$.
\Else
  \State Generate initial samples $\mathbf{x}_0^{(i)} = \mathbf{x}_0$, for $i=1,\dots,N$.
\EndIf

\Statex \textbf{3. Time Loop}
\For{$t = 0$ to $T_{\max}$}

  \Statex \hskip1.5em \textbf{A. Propagate State or Samples}
  \If{$\mathcal{M} \neq \texttt{MonteCarlo}$}
      \State $\mu_t \gets \mathbf{v}^\top \mathbf{A}_t \mathbf{x}_0$.
      \State $\sigma_t \gets \sqrt{\mathbf{v}^\top \mathbf{P}_t \mathbf{v}}$.
      \State $\gamma_t \gets (\mathbf{A}_t \mathbf{v})^\top \mathbf{\Sigma}_\infty(\mathbf{A}_t \mathbf{v})$.
      \State $\mathbf{P}_{t+1} \gets \mathbf{P}_t + \mathbf{A}_t \mathbf{\Sigma}\mathbf{A}_t^\top$.
      \State $\mathbf{A}_{t+1} \gets \mathbf{Q}\mathbf{A}_t$.
  \Else
      \State $\mathbf{x}_{t+1}^{(i)} \gets \mathbf{Q}\mathbf{x}_t^{(i)} + \xi^{(i)}$.
      \State Estimate $\hat{\mu}_t$, $\hat{\sigma}_t$ from $\mathbf{v}^\top\mathbf{x}_t^{(i)}$.
      \State $\mu_t \gets \hat{\mu}_t$, \quad $\sigma_t \gets \hat{\sigma}_t$, \quad $\gamma_t \gets \sigma_\infty^2$.
  \EndIf

  \Statex \hskip1.5em \textbf{B. Lower and Upper Bounds}
  \State $L_B^{(t)} \gets |\mu_t|$.
  \State $U_B^{(t)} \gets L_B^{(t)} + (C_r / \sigma_\infty)\, \gamma_t$.

  \Statex \hskip1.5em \textbf{C. Wasserstein Distance}
  \If{$\mathcal{M} = \texttt{ClosedForm}$}
      \State $\mathcal{W}_2^{(t)} \gets \sqrt{\mu_t^2 + (\sigma_t - \sigma_\infty)^2}$.
  \ElsIf{$\mathcal{M} = \texttt{NumInt}$}
      \State $F^{\leftarrow}_{t}(u) \gets \mu_t + \sigma_t\,\Phi^{-1}(u)$.
      \State $F^{\leftarrow}_{\infty}(u) \gets \sigma_\infty\,\Phi^{-1}(u)$.
      \State $I_r \gets \int_0^1 |F^{\leftarrow}_{t}(u) - |F^{\leftarrow}_{\infty}(u)|^r\,du$.
      \State $\mathcal{W}_r^{(t)} \gets I_r^{1/r}$.
  \Else
      \State Sort projections $X_{(1)} \leq \dots \leq X_{(N)}$ from $\mathbf{v}^\top \mathbf{x}_t^{(i)}$.
      \State $Y_{(k)} \gets \sigma_\infty \,\Phi^{-1}\!\left(\frac{k-0.5}{N}\right)$.
      \State $\mathcal{W}_r^{(t)} \gets \left(\frac{1}{N}\sum_{k=1}^N |X_{(k)} - Y_{(k)}|^r\right)^{1/r}$.
  \EndIf

\EndFor

\State \Return $L_B^{(t)},\, U_B^{(t)},\, \mathcal{W}_r^{(t)}$.

\end{algorithmic}
\end{algorithm}

\bigskip 
\section{\textbf{The rate of convergence in Example~\ref{ex:Bernoulli}}}\label{a:Bernoulli}
For $d=1$ consider the i.i.d. symmetric Rademacher $(\xi_t)_{t\in \mathbb{N}}$, $\xi_t \stackrel{\mathsf{d}}{=} \frac{1}{2}\delta_{-1} + \frac{1}{2}\delta_1$ and the respective autoregressive model 
\[
X_{t}(x) = \frac{1}{2}X_{t-1} + \frac{1}{2}\xi_t, \qquad X_0(x) = x. 
\]

\noindent The  affine interpolation condition given in~\eqref{e:affine-ergodi-interpolation} requires that the 
marginal distribution at time $t$ be a deterministic affine transformation of the limiting distribution $X_{\infty} \stackrel{\mathsf{d}}{=} U([-1,1])$. While this 
assumption is satisfied for Gaussian or $\alpha$-stable distributions, it trivially fails in the Bernoulli setting due to 
the topological incompatibility between the discrete marginals and the uniform distribution.

In the sequel, we compute $\mathcal{W}_2(X_t(x), X_\infty)$. Recursively, we obtain
\[ 
X_{t} \stackrel{\mathsf{d}}{=} \left(\frac{1}{2}\right)^{t} x + \frac{1}{2}\sum_{k=0}^{t-1} \left(\frac{1}{2}\right)^k \xi_{k},
\]
with $X_0 = x$. Hence, the stationary solution, denoted by $X_\infty$ can be written as follows
\[
    X_\infty \stackrel{\mathsf{d}}{=} \sum_{k=0}^{\infty} \left(\frac{1}{2}\right)^{k+1} \xi_k.
\]
Classical results by Erd\H{o}s \cite{erdos1939} and 
Jessen and Wintner \cite{jessen1935} establish this convergence, distinguishing the case the contraction rate $q=1/2$ from general contractions $|q|<1$.

\medskip

Recall that
\[
\mathcal{W}_2^2(X_t, X_\infty) = \int_0^1 |F_t^{\leftarrow}(u) - F^{\leftarrow}_\infty(u)|^2 \, \mathrm{d} u,
\]
where $F^{\leftarrow}_t$ and $F^{\leftarrow}_\infty$ are the quantiles of the law of $X_t$ and $X_\infty$, respectively.

\medskip

Note that
\begin{equation} \label{eq:split}
    X_\infty \stackrel{\mathsf{d}}{=}  \underbrace{\sum_{k=0}^{t-1} \left(\frac{1}{2}\right)^{k+1} \xi_k}_{=:S_t} + \underbrace{\sum_{k=t}^{\infty} \left(\frac{1}{2}\right)^{k+1} \xi_k}_{=:R_t}.
\end{equation}
Let $N_t = 2^t$ and observe that
\begin{align*}
    S_t & = \sum_{j=0}^{t-1} \left(\frac{1}{2}\right)^{j+1} \xi_j 
          = \frac{1}{N_t} \underbrace{\sum_{j=0}^{t-1} 2^{t-1-j} \xi_j}_{=:Y_t}.
\end{align*}
Since for each $j$, $\xi_j \stackrel{\mathsf{d}}{=} 2\beta_j - 1$, where $\beta_j$ are Bernoulli random variables with parameter $1/2$, 
we have
\[
Y_t\stackrel{\mathsf{d}}{=} \sum_{j=0}^{t-1} (2\beta_j - 1) 2^j = 2K_t - (N_t -1),
\]
where $K_t = \sum_{j=0}^{t-1} \beta_j 2^j$ which takes values in $\{0, \dots, N_t-1\}$. Consequently, $Y_t$ takes values in the set of 
all (odd) integers  in $[-(N_t-1),N_t-1]$. 
\noindent The cumulative distribution function $F_{S_t}$ is a step function reflecting the uniform probability mass $1/N_t$ at each support point 
\begin{equation}
s_k = \frac{2k+1}{N_t} - 1 \quad \text{for } k = 0, \dots, N_t-1,
\end{equation}
which represents a finite-time approximation of the uniform limit with cumulative distribution function $F_\infty(x) = \frac{x+1}{2}$ on $[-1,1]$.   

In the sequel, we compute   function $F_t^{\leftarrow}(u) = \inf \{ x\in \mathbb{R} : F_t(x) \geq u \}$, $u\in (0,1)$. Note that  $F_t^{\leftarrow}$ is piecewise constant on the intervals $I_k:= ( \frac{k}{N_t}, \frac{k+1}{N_t} ]$. For any $u \in I_k$, the inverse maps to the $k$-th mass point $s_k$ is shifted by $q_t:= 2^{-t}x$. Using the explicit formula for $s_k$, we have
\begin{equation}
    F_t^{\leftarrow}(u) = q_t + s_k = q_t + \left( \frac{2k+1}{N_t} - 1 \right) \quad \text{ for }\quad u \in I_k,
\end{equation}
which gives
\begin{align*}
    \mathcal{W}_2^2(X_t,X_\infty) 
        & = \sum_{k=0}^{N_t-1} \int_{k/N_t}^{(k+1)/N_t} \left(q_t + s_k - (2u - 1) \right)^2 \mathrm{d} u.
\end{align*}
By the change of variable $u = \frac{k}{N_t} + \frac{1}{2N_t} + v$ which centers the integral on 
$v \in [-\frac{1}{2N_t}, \frac{1}{2N_t}]$, we have
\begin{align*}
    \mathcal{W}_2^2(X_t,X_\infty)
        & = \sum_{k=0}^{N_t-1} \int_{-\frac{1}{2N_t}}^{\frac{1}{2N_t}} (q_t - 2v)^2\, \mathrm{d} v
          = \sum_{k=0}^{N_t-1} \int_{-\frac{1}{2N_t}}^{\frac{1}{2N_t}} (q_t^2 - 4q_tv + 4v^2)\, \mathrm{d} v\\
        & = \sum_{k=0}^{N_t-1} \left( \frac{q_t^2}{N_t} + \frac{1}{3N^3_t} \right) 
         = q_t^2 + \frac{1}{3N^2_t}.
\end{align*}
Recalling that $q_t = x/2^t$ and $N_t = 2^t$, we obtain
\begin{equation}
    \mathcal{W}_2^2(X_t(x), X_\infty) = \frac{1}{2^{2t}} \left( x^2 + \frac{1}{3} \right).
\end{equation}

\bigskip 
\section*{\textbf{Declarations}}

\noindent
\textbf{Acknowledgments:} Gerardo Barrera would like to express his gratitude to the Center for Mathematical Analysis, Geometry and Dynamical Systems CAMGSD and Instituto Superior T\'ecnico (IST) Lisbon 
for all the facilities used along the realization of this work.
Michael A. H\"ogele express his gratitute to the CAMGSD at IST Lisbon for the kind hospitality during a secondment of LiBERA project in June-July 2025. 
P.H. da Costa acknowledges the Department of Mathematics at the University of Brasília (UnB) for the infrastructure provided during the preparation of this work.

\noindent
\textbf{Availability of data and material:} 
The results generated in the numerical experiment are available upon request.
\hfill

\noindent
\textbf{Conflict of interests:} The authors declare that they have no conflict of interest.
\hfill

\noindent
\textbf{Authors' contributions:} All authors have contributed equally to the paper.

\noindent
\textbf{Ethical approval:} Not applicable.

\noindent
\textbf{Funding:} The research of Michael A. H\"ogele has been supported by the project ``Mean deviation frequencies and the cutoff phenomenon'' (INV-2023-162-2850) of the School of Sciences (Facultad de Ciencias) at Universidad de los Andes, Bogot\'a, Colombia.
The research of Gerardo Barrera and Michael A. H\"ogele is partially supported by European Union’s Horizon Europe research and innovation programme under the Marie Sk\l{}odowska-Curie Actions Staff Exchanges (Grant agreement No.~101183168 -- LiBERA, Call: HORIZON-MSCA-2023-SE-01). Also, the research of Gerardo Barrera is partially funded by Funda\c{c}\~ao para a Ci\^encia e Tecnologia (FCT), Portugal, through grant FCT/Portugal  project no. UID/04459/2025 with DOI identifier 
10-54499/UID/04459/2025.

\noindent 
\textbf{Disclaimer:} Funded by the European Union. Views and opinions expressed are however those of the author(s) only and do not necessarily reflect those of the European Union or the European Education and Culture Executive Agency (EACEA). Neither the European Union nor EACEA can be held responsible for them.

\end{document}